\documentclass[11pt]{amsart}

\usepackage{amsmath}
\usepackage{amssymb}
\usepackage{amsthm}
\usepackage{mathrsfs}
\usepackage{enumitem}
\usepackage[all]{xypic}
\usepackage[dvipdfmx,bookmarks=true,bookmarksnumbered=true]{hyperref}
\usepackage[hmargin=2cm,vmargin=3cm]{geometry}

\setlist[enumerate]{itemsep=5pt, topsep=5pt}
\setlist[itemize]{itemsep=5pt, topsep=5pt}

\numberwithin{equation}{section}

\theoremstyle{plain}
\newtheorem{thm}{Theorem}[section]
\newtheorem{lem}[thm]{Lemma}
\newtheorem{prop}[thm]{Proposition}
\newtheorem{cor}[thm]{Corollary}
\newtheorem*{GPconj}{Gross--Prasad conjecture}

\theoremstyle{definition}

\theoremstyle{remark}
\newtheorem{rem}[thm]{Remark}

\def\Aut{\operatorname{Aut}}

\def\disc{\operatorname{disc}}
\def\End{\operatorname{End}}
\def\ev{\operatorname{ev}}
\def\Herm{\operatorname{Herm}}
\def\Hom{\operatorname{Hom}}
\def\Ind{\operatorname{Ind}}
\def\Irr{\operatorname{Irr}}
\def\Isom{\operatorname{Isom}}
\def\Ker{\operatorname{Ker}}
\def\Lie{\operatorname{Lie}}
\def\Norm{\operatorname{Norm}}
\def\N{\operatorname{N}}
\def\Rat{\operatorname{Rat}}
\def\Re{\operatorname{Re}}
\def\Res{\operatorname{Res}}
\def\supp{\operatorname{supp}}
\def\Tr{\operatorname{Tr}}

\DeclareMathOperator*{\ord}{ord}

\def\an{\mathrm{an}}
\def\Ad{\mathrm{Ad}}
\def\As{\mathrm{As}}
\def\GL{\mathrm{GL}}
\def\LS{\mathrm{LS}}
\def\Mp{\mathrm{Mp}}
\def\Sh{\mathrm{Sh}}
\def\SL{\mathrm{SL}}

\def\Sp{\mathrm{Sp}}
\def\U{\mathrm{U}}

\def\WD{\mathit{WD}}
\def\Chi{\boldsymbol{\chi}}
\def\spl{\boldsymbol{\mathit{spl}}}

\def\AA{\mathbb{A}}
\def\CC{\mathbb{C}}
\def\EE{\mathbb{E}}
\def\FF{\mathbb{F}}
\def\RR{\mathbb{R}}
\def\VV{\mathbb{V}}
\def\WW{\mathbb{W}}
\def\XX{\mathbb{X}}
\def\ZZ{\mathbb{Z}}

\def\A{\mathcal{A}}
\def\M{\mathcal{M}}
\def\R{\mathcal{R}}
\def\T{\mathcal{T}}

\def\H{\mathscr{H}}
\def\s{\mathscr{S}}
\def\V{\mathscr{V}}

\def\aa{\mathfrak{a}}
\def\ff{\mathfrak{f}}

\DeclareSymbolFont{bbold}{U}{bbold}{m}{n}
\DeclareMathSymbol{\bbmu}{\mathord}{bbold}{"16}
\DeclareMathSymbol{\Eins}{\mathord}{bbold}{"31}
\def\1{\Eins}

\makeatletter
\def\varddots{\mathinner{\mkern1mu
    \raise\p@\hbox{.}\mkern2mu\raise4\p@\hbox{.}\mkern2mu
    \raise7\p@\vbox{\kern7\p@\hbox{.}}\mkern1mu}}
\makeatother

\newcommand{\BIGOP}[1]{\mathop{\mathchoice%
{\raise-0.22em\hbox{\huge $#1$}}%
{\raise-0.05em\hbox{\Large $#1$}}{\hbox{\large $#1$}}{#1}}}

\newcommand{\BIGboxplus}{\mathop{\mathchoice%
{\raise-0.35em\hbox{\huge $\boxplus$}}%
{\raise-0.15em\hbox{\Large $\boxplus$}}{\hbox{\large $\boxplus$}}{\boxplus}}}

\title{The Gross--Prasad conjecture and local theta correspondence}
\author{Wee Teck Gan}
\address{Department of Mathematics, National University of Singapore, 10 Lower Kent Ridge Road, Singapore 119076}
\email{matgwt@nus.edu.sg}
\author{Atsushi Ichino}
\address{Department of Mathematics, Kyoto University, Kitashirakawa Oiwake-cho, Sakyo-ku, Kyoto 606-8502, Japan}
\email{ichino@math.kyoto-u.ac.jp}
\date{\today}

\begin{document}

\begin{abstract}
We establish the Fourier--Jacobi case of the local Gross--Prasad conjecture for unitary groups, by using local theta correspondence to relate the Fourier--Jacobi case with the Bessel case established by Beuzart-Plessis.
To achieve this, we prove two conjectures of D.~Prasad on the precise description of the local theta correspondence for (almost) equal rank unitary dual pairs in terms of the local Langlands correspondence.
The proof uses Arthur's multiplicity formula and thus is one of the first examples of a concrete application of this ``global reciprocity law''.
\end{abstract}

\maketitle

\section{\textbf{Introduction}}

In \cite{gp1}, \cite{gp2}, \cite{ggp1}, \cite{ggp2}, a restriction problem in the representation theory of classical groups was studied and a precise conjecture was formulated for this restriction problem.
This so-called \emph{Gross--Prasad} (GP) conjecture has generated much interest in recent years. 

\subsection{Restriction problem}

In this paper, we shall focus on the restriction problem for unitary groups.
Thus, let $F$ be a nonarchimedean local field of characteristic $0$ and residue characteristic $p$, and let $E$ be a quadratic field extension of $F$.
Let $V_{n+1}$ be a Hermitian space of dimension $n+1$ over $E$ and $W_n$ a skew-Hermitian space of dimension $n$ over $E$.
Let $V_n \subset V_{n+1}$  be a nondegenerate subspace of codimension $1$,  so that we have a natural inclusion of their corresponding unitary groups $\U(V_{n}) \hookrightarrow \U(V_{n+1})$. 
In particular, if we set 
\[  G_n =  \U(V_n) \times \U(V_{n+1}) \quad \text{or} \quad \U(W_n) \times \U(W_n) \]
and
\[   H_n = \U(V_n) \quad \text{or} \quad \U(W_n), \]
then we have a diagonal embedding
\[ \Delta:  H_n \hookrightarrow G_n. \]

Let $\pi$ be an irreducible smooth representation of $G_n$. In the Hermitian case, one is interested in determining 
\[  \dim_\CC \Hom_{\Delta H_n} ( \pi, \CC). \]
We shall call this the \emph{Bessel} case (B) of the GP conjecture.
In the skew-Hermitian case, the restriction problem requires another piece of data: a Weil representation $\omega_{\psi, \chi, W_n}$, where $\psi$ is a nontrivial additive character of $F$ and $\chi$ is a character of $E^{\times}$ whose restriction to $F^{\times}$ is the quadratic character $\omega_{E/F}$ associated to $E/F$ by local class field theory.
Then  one is interested in determining
\[  \dim_\CC \Hom_{\Delta H_n} ( \pi, \omega_{\psi,\chi, W_n}). \]
We shall call this the \emph{Fourier--Jacobi} case (FJ) of the GP conjecture. To unify notation, we shall let $\nu = \CC$ or $\omega_{\psi,\chi, W_n}$ in the respective cases.

By surprisingly recent results of Aizenbud--Gourevitch--Rallis--Schiffmann \cite{agrs} and Sun \cite{sun}, it is known that the above Hom spaces have dimension at most $1$. Thus the main issue is to determine when the Hom space is nonzero. 
In \cite{ggp1}, an answer for this issue is formulated  in the framework of the local Langlands correspondence,  in its enhanced form due to Vogan \cite{v} which takes into account all pure inner forms.  

\subsection{Local Langlands correspondence}

More precisely, a pure inner form of $\U(V_n)$ is simply a group of the form $\U(V_n')$, where $V_n'$ is a  Hermitian space of dimension $n$ over $E$; likewise in the skew-Hermitian case.  
Thus, a pure inner form of $G_n$ is a group of the form
\[  G_n' = \U(V_n') \times \U(V'_{n+1}) \quad \text{or} \quad \U(W'_n) \times \U(W''_n). \]
We say that such a pure inner form is \emph{relevant} if
\[  V'_n \subset V'_{n+1} \quad \text{or} \quad   W'_n = W''_n, \] 
and 
\[   V_{n+1}'/V_{n}' \cong V_{n+1}/V_{n} \]
in the Hermitian case.
If $G_n'$ is relevant, we set
\[   H'_n = \U(V'_n) \quad \text{or} \quad \U(W'_n), \]
so that we have a diagonal embedding
\[ \Delta:  H'_n \hookrightarrow G'_n. \]

Now suppose that $\phi$ is an $L$-parameter for the group $G_n$.
Then $\phi$ gives rise to a Vogan $L$-packet $\Pi_{\phi}$ consisting of certain irreducible smooth representations of $G_n$ and its (not necessarily relevant) pure inner forms $G_n'$.
Moreover, after fixing a Whittaker datum for $G_n$, there is a natural bijection
\[  \Pi_{\phi} \longleftrightarrow \Irr(S_{\phi}), \]
where  $S_{\phi}$ is the component group associated to $\phi$. 
Thus an irreducible smooth representation of $G_n$ is labelled by a pair $(\phi, \eta)$, where $\phi$ is an $L$-parameter for $G_n$ and $\eta$ is an irreducible character of $S_{\phi}$. 

By the recent work of Arthur \cite{a}, Mok \cite{mok}, and Kaletha--M\'inguez--Shin--White \cite{kmsw}, together with the stabilization of the twisted trace formula established by Waldspurger and M{\oe}glin--Waldspurger \cite{mw-2014}, the local Langlands correspondence for unitary groups is now unconditional, expect that the general case of the weighted fundamental lemma has not been written; the work of Chaudouard and Laumon \cite{cl} is limited to the case of split groups.

\subsection{Gross--Prasad conjecture}

With this short preparation, the GP conjecture can be loosely stated as follows:

\begin{GPconj}
\hfill
\begin{enumerate}
\item Given a generic $L$-parameter $\phi$ for $G_n$, there is a unique representation $\pi(\phi,\eta)$ in the Vogan $L$-packet $\Pi_{\phi}$ such that $\pi(\phi,\eta)$ is a representation of a relevant pure inner form $G_n'$ and such that $\Hom_{\Delta H'_n} (\pi(\phi,\eta), \nu)  \ne 0$.
\item There is a precise recipe for the distinguished character $\eta$ (which we will recall in \S \ref{SS:eta} below).
\end{enumerate}
\end{GPconj}

In a stunning series of papers \cite{w1}, \cite{w2}, \cite{w3}, \cite{w4}, Waldspurger has established the Bessel case of the GP conjecture for special orthogonal groups in the case of tempered $L$-parameters; the case of general generic $L$-parameters is then dealt with by M{\oe}glin--Waldspurger \cite{mw}. Beuzart-Plessis \cite{bp1}, \cite{bp2}, \cite{bp3} has since extended Waldspurger's techniques to settle the Bessel case of the GP conjecture for unitary groups in the tempered case.

\subsection{Purpose of this paper}

The purpose of this paper is to establish the Fourier--Jacobi case of the GP conjecture, as well as two conjectures of D.~Prasad concerning local theta correspondence in the (almost) equal rank case.

Let us describe the main idea of the proof.
For simplicity, we restrict ourselves to the case of tempered $L$-parameters here.
The Bessel and Fourier--Jacobi cases of the GP conjecture are related by the local theta correspondence. More precisely, there is a see-saw diagram
\[
 \xymatrix{
  \U(W_n)  \times \U(W_n)  \ar@{-}[dr] \ar@{-}[d] & \U(V_{n+1}) 
     \ar@{-}[d] \\
  \U(W_n) \ar@{-}[ur] &  \U(V_n) \times \U(V_1)}
\]  
and the associated see-saw identity reads:
\[  \Hom_{\U(W_n)}( \Theta_{\psi, \chi, V_n, W_n}(\sigma) \otimes \omega_{\psi, \chi, V_1, W_n},  \pi) \cong \Hom_{\U(V_n)}(\Theta_{\psi, \chi, V_{n+1}, W_n}(\pi),  \sigma) \]
for irreducible smooth representations $\pi$ of $\U(W_n)$ and $\sigma$ of $\U(V_n)$. Hence the left-hand side of the see-saw identity concerns the Fourier--Jacobi case (FJ) whereas the right-hand side concerns the Bessel case (B). It is thus apparent that precise knowledge of the local theta correspondence for unitary groups of (almost) equal rank will give the precise relation of (FJ) to (B).  

More precisely, one would need to know:
\begin{itemize}
\item[($\Theta$)] 
For irreducible tempered representations $\pi$ and $\sigma$, the big theta lifts $\Theta_{\psi, \chi, V_{n+1}, W_n}(\pi)$ and $\Theta_{\psi, \chi, V_n, W_n}(\sigma)$ are irreducible (if nonzero).

\item[(P1)] If $\sigma$ has parameter $(\phi, \eta)$ and $\Theta_{\psi, \chi, V_n, W_n}(\sigma)$ has parameter $(\phi', \eta')$, then $(\phi',\eta')$ can be precisely described in terms of $(\phi,\eta)$.

\item[(P2)]  Likewise, if $\pi$ has parameter $(\phi, \eta)$ and $\Theta_{\psi, \chi, V_{n+1}, W_n}(\pi)$ has parameter $(\phi', \eta')$, then $(\phi',\eta')$ can be precisely described in terms of $(\phi,\eta)$.
\end{itemize}

In fact, in \cite{p1}, \cite{p2}, D.~Prasad has formulated precise conjectures regarding (P1) and (P2) for the theta correspondence for $\U(V_n) \times \U(W_n)$ and $\U(V_{n+1}) \times \U(W_n)$ respectively; we shall recall his conjectures precisely in \S \ref{S:Pconj}.  We shall also denote by (weak P1) the part of the conjecture (P1) concerning only the correspondence of $L$-parameters $\phi \mapsto \phi'$; likewise we have (weak P2). Then we recall that in our earlier paper \cite{gi}, we have shown:

\begin{prop} \label{P:1.1}
The statements $(\Theta)$, $(\emph{weak P1})$ and $(\emph{weak P2})$ hold.
\end{prop}

Using Proposition \ref{P:1.1}, the first observation of this paper is:

\begin{prop} \label{P:1.2}
Assume $(\emph{B})$ and $(\emph{P2})$. Then $(\emph{FJ})$ and $(\emph{P1})$ follow. 
\end{prop}

In view of Proposition \ref{P:1.2} and the work of Beuzart-Plessis \cite{bp1}, \cite{bp2}, \cite{bp3}, it remains to show the statement (P2), and our main result is:

\begin{thm}
\label{t:main}
The conjecture $(\emph{P2})$, and hence $(\emph{FJ})$ and $(\emph{P1})$, holds.
\end{thm}

Let us make a few comments about the results:

\begin{itemize}
\item
In fact, we prove (P1) and (P2) for all (not necessarily tempered nor generic) $L$-parameters.

\item
We mention a related result of M{\oe}glin \cite{m} about the local theta correspondence for symplectic-orthogonal dual pairs of arbitrary rank.
She considered $A$-packets for a large class of $A$-parameters, including all tempered $L$-parameters,
and then determined the analog of the correspondence $(\phi, \eta) \mapsto (\phi', \eta')$ in the sense of Arthur,
assuming that the correspondence is known for supercuspidal (and slightly more general) representations.

\item 
It is interesting to note that in Proposition \ref{P:1.2}, the roles of (P1) and (P2) can be switched. In other words, it is also sufficient to prove (P1) in order to prove (FJ). We shall explain in the next subsection why we prefer to prove (P2). 

\item
In \cite{ggp1}, both the Bessel (B) and Fourier--Jacobi (FJ) cases of the GP conjecture were formulated for pairs of spaces $V_n \subset V_{n+2k+1}$ or $W_n \subset W_{n+2k}$ for any nonnegative integer $k$ and for any generic $L$-parameters for $\U(V_n) \times \U(V_{n+2k+1})$ or $\U(W_n) \times \U(W_{n+2k})$.
Beuzart-Plessis \cite{bp1}, \cite{bp2}, \cite{bp3} has in fact verified (B) for all tempered $L$-parameters for $\U(V_n) \times \U(V_{n+2k+1})$.
In \S \ref{s:generic}, we check that the argument as in \cite{mw} gives (B) for all generic $L$-parameters for $\U(V_n) \times \U(V_{n+2k+1})$ and then show that Theorem \ref{t:main} continues to hold for all generic $L$-parameters for $\U(W_n) \times \U(W_n)$.

\item
On the other hand, it was shown in \cite[Theorem 19.1]{ggp1} that the GP conjecture in the case of generic $L$-parameters for $\U(W_n) \times \U(W_{n+2k})$ (for all $k>0$) follows from that for $\U(W_n) \times \U(W_n)$.
Namely, we can deduce from Theorem \ref{t:main} the following:
\end{itemize}

\begin{cor}
The Fourier--Jacobi case of the GP conjecture holds for all generic $L$-parameters for $\U(W_n) \times \U(W_{n+2k})$ for any $k \ge 0$.
\end{cor}

\subsection{Prasad's conjectures}

Given Proposition \ref{P:1.1}, the main work is to determine how $\eta'$ depends on $(\phi, \eta)$ in (P1) and (P2). In fact, the precise determination of $\eta'$ in (P1) is a very subtle issue, as it depends on certain local roots numbers.
In the case of (P2), the dependence of $\eta'$ on $(\phi,\eta)$ is more simplistic. 

The proof of (P2) proceeds by the following steps:

\begin{itemize}

\item First, by our results in \cite{gi}, the nontempered case can be reduced to the tempered case on smaller unitary groups.

\item Next, we show that the tempered case can be reduced to the square-integrable case on smaller unitary groups. This is achieved by a nontrivial extension of  the techniques in the PhD thesis of the second author \cite{i} and uses the delicate details of the normalization of the intertwining operators involved in the local intertwining relation \cite{a}, \cite{mok}, \cite{kmsw}.

\item Finally, we show the square-integrable case by a global argument. More precisely, we shall globalize an irreducible square-integrable representation $\pi$ of $\U(W_n)$ to an irreducible cuspidal automorphic representation $\Pi = \otimes_v \Pi_v$ such that

\begin{itemize}

\item $\Pi_v$ is not square-integrable for all places outside the place of interest, so that (P2) is known for $\Pi_v$ outside the place of interest,

\item $\Pi$ has tempered $A$-parameter whose global component group is equal to  the local component group of the $L$-parameter of $\pi$,

\item $\Pi$ has nonzero global theta lift to a unitary group which globalizes $\U(V_{n+1})$.
 
\end{itemize}
The desired result then follows for the place of interest by applying Arthur's multiplicity formula for the automorphic discrete spectrum, which can be viewed as a sort of product formula (see \eqref{eq:arthur}).
\end{itemize}

We can now explain why we prefer to prove (P2) rather than (P1). Note that one could attempt to follow the same strategy of proof for the statement (P1). However, in the globalization step above, we need to ensure that $\Pi$ has nonzero global theta lift to a certain unitary group. For the case of (P1), the nonvanishing of this global theta lift is controlled by the nonvanishing of $L(\frac{1}{2}, \Pi)$, and it is well-known that the nonvanishing of this central critical value is a very subtle issue with arithmetic implications.  On the other hand, for the statement (P2), the nonvanishing of the global theta lift of $\Pi$ is governed by the nonvanishing of $L(1,\Pi)$. Now it is certainly much easier to ensure the nonvanishing of $L(1, \Pi)$ compared to $L(\frac{1}{2},\Pi)$. For example, if $\Pi$ has tempered $A$-parameter, then one knows that $L(1, \Pi) \ne 0$. It is for this reason that we prove (P2) rather than (P1).

\subsection{3 birds and 2 stones}

To summarise, in proving our main theorem, we have killed ``3 birds'' (i.e.~(FJ), (P1) and (P2)) with ``2 stones'' (i.e.~(B) and Arthur's multiplicity formula), though it is probably more accurate to describe the latter as 2 cannon balls. We stress however that no animals (besides the two authors) have suffered in the preparation of this article. 

\subsection*{Acknowledgements}

We would like to thank Tasho Kaletha for useful discussions.
The first author is partially supported by a Singapore government MOE Tier 2 grant R-146-000-175-112.
The second author is partially supported by JSPS Grant-in-Aid for Scientific Research (B) 26287003.
This material is based upon work supported by the National Science Foundation under Grant No.~0932078 000 while the authors were in residence at the Mathematical Sciences Research Institute in Berkeley, California, during the Fall 2014 semester.

\subsection*{Notation}

Let $F$ be a nonarchimedean local field of characteristic $0$ and residue characteristic $p$.
We fix an algebraic closure $\bar{F}$ of $F$.
Let $\Gamma = \operatorname{Gal}(\bar{F}/F)$ be the absolute Galois group of $F$ and $W_F$ the Weil group of $F$.
Let $|\cdot|_F$ be the normalized absolute value on $F$.
We fix a nontrivial additive character $\psi$ of $F$.

Let $E$ be a quadratic field extension of $F$ and $\omega_{E/F}$ the quadratic character of $F^{\times}$ associated to $E/F$ by local class field theory.
Let $c$ denote the nontrivial Galois automorphism of $E$ over $F$.
Let $\Tr_{E/F}$ and $\N_{E/F}$ be the trace and norm maps from $E$ to $F$.
We choose an element $\delta \in E^{\times}$ such that $\Tr_{E/F}(\delta) = 0$.
We write $| \cdot | = |\cdot|_E$ for the normalized absolute value on $E$.
Let $\psi_E$ be the nontrivial additive character of $E$ defined by $\psi_E = \psi \circ \Tr_{E/F}$.

If $G$ is a linear algebraic group over $F$, we identify $G$ with its group of $F$-rational points $G(F)$.
For any totally disconnected locally compact group $G$, let $\1_G$ be the trivial representation of $G$ and $\Irr(G)$ the set of equivalence classes of irreducible smooth representations of $G$.
For any set $X$, let $1_X$ be the identity map of $X$.
For any positive integer $n$, let $1_n$ be the identity matrix in $\GL_n$.

\section{\textbf{Local Langlands correspondence}}

In this section, we summarize some properties of the local Langlands correspondence for unitary groups.

\subsection{Hermitian and skew-Hermitian spaces}

Fix $\varepsilon = \pm 1$.
Let $V$ be a finite dimensional vector space over $E$ equipped with a nondegenerate $\varepsilon$-Hermitian $c$-sesquilinear form $\langle \cdot, \cdot \rangle_V : V \times V \rightarrow E$.
Thus we have
\[
 \langle a v, b w \rangle_V = a b^c \langle v, w \rangle_V, \quad
 \langle w, v \rangle_V = \varepsilon \cdot \langle v, w \rangle_V^c
\]
for $v, w \in V$ and $a, b \in E$.
Put $n = \dim V$ and $\disc V = (-1)^{(n-1)n/2} \cdot \det V$, so that 
\[
 \disc V \in
 \begin{cases}
  F^{\times} / \mathrm{N}_{E/F}(E^{\times})
  & \text{if $\varepsilon = +1$;} \\
  \delta^n \cdot F^{\times} / \mathrm{N}_{E/F}(E^{\times})
  & \text{if $\varepsilon = -1$.}
 \end{cases}
\]
We define $\epsilon(V) = \pm 1$ by 
\[
 \epsilon(V) = 
 \begin{cases}
  \omega_{E/F}(\disc V) & \text{if $\varepsilon = +1$;} \\
  \omega_{E/F}(\delta^{-n} \cdot \disc V) & \text{if $\varepsilon = -1$.}
 \end{cases}
\]
Given a positive integer $n$, there are precisely two isometry classes of $n$-dimensional $\varepsilon$-Hermitian spaces $V$, 
which are distinguished from each other by their signs $\epsilon(V)$.
Note that $\epsilon(V)$ depends on the choice of $\delta$ if $\varepsilon = -1$ and $n$ is odd.
Let $\U(V)$ be the unitary group of $V$, i.e.~the connected reductive linear algebraic group over $F$ defined by
\[
  \U(V) = \{ g \in \GL(V) \, | \,
 \text{$\langle g v, g w \rangle_V =  \langle v, w \rangle_V$ for $v, w \in V$}
 \}.
\]
If $n=0$, we interpret $\U(V)$ as the trivial group $\{ 1 \}$.

\subsection{$L$-parameters and component groups}

Let $W_E$ be the Weil group of $E$ and $\WD_E = W_E \times \SL_2(\CC)$ the Weil-Deligne group of $E$.
We say that a continuous homomorphism $\phi: \WD_E \rightarrow \GL_n(\CC)$ is a representation of $\WD_E$ if 
\begin{itemize}
 \item $\phi$ is semisimple, 
 \item the restriction of $\phi$ to $\SL_2(\CC)$ is algebraic.
\end{itemize}
We say that $\phi$ is tempered if the image of $W_E$ is bounded.
Let $\phi^{\vee}$ be the contragredient representation of $\phi$ defined by $\phi^\vee(w) = {}^t\phi(w)^{-1}$.
Fix $s \in W_F \smallsetminus W_E$ and define a representation $\phi^c$ of $\WD_E$ by $\phi^c(w) = \phi(sws^{-1})$.
Then the equivalence class of $\phi^c$ is independent of the choice of $s$.
We say that $\phi$ is conjugate self-dual if there is a nondegenerate bilinear form $B : \CC^n \times \CC^n \rightarrow \CC$ which satisfies
\[
 B(\phi(w) x, \phi^c(w) y) = B(x, y)
\]
for all $w \in \WD_E$ and $x, y \in \CC^n$.
Namely, $\phi$ is conjugate self-dual if and only if $\phi^c$ is equivalent to $\phi^\vee$.
For $b = \pm 1$, we say that $\phi$ is conjugate self-dual with sign $b$ if there is a nondegenerate bilinear form $B : \CC^n \times \CC^n \rightarrow \CC$ which satisfies the above condition and the condition that 
\[
 B(y, x) = b \cdot B(x, \phi(s^2) y)
\]
for all $w \in \WD_E$ and $x, y \in \CC^n$.
Note that the sign $b$ depends not only on $\phi$ but also on $B$.
If $\phi$ is conjugate self-dual with sign $b$ (with respect to a bilinear form $B$), 
then $\det \phi$ is conjugate self-dual with sign $b^n$.

By \cite[\S 8]{ggp1}, an $L$-parameter for the unitary group $\U(V)$ is an $n$-dimensional conjugate self-dual representation $\phi$ of $\WD_E$ with sign $(-1)^{n-1}$.
We may decompose $\phi$ into a direct sum
\[  \phi = \bigoplus_i m_i \phi_i \]
with pairwise inequivalent irreducible representations $\phi_i$ of $\WD_E$ and multiplicities $m_i$.
We say that $\phi$ is square-integrable if it is multiplicity-free (so that $m_i=1$ for all $i$) and $\phi_i$ is conjugate self-dual with sign $(-1)^{n-1}$ for all $i$.

For an $L$-parameter $\phi$ for $\U(V)$, fix a bilinear form $B$ as above and let $\Aut(\phi,B)$ be the group of elements in $\GL_n(\CC)$ which centralize the image of $\phi$ and preserve $B$.
Let
\[  S_{\phi}  = \Aut(\phi, B) / \Aut(\phi, B)^0 \]
be the component group of $\phi$, 
where $\Aut(\phi, B)^0$ is the identity component of $\Aut(\phi, B)$.
As shown in \cite[\S 8]{ggp1}, $S_{\phi}$ has an explicit description of the form
\[  S_{\phi}  = \prod_j  (\ZZ / 2\ZZ) a_j \]
with a canonical basis $\{ a_j \}$, where the product ranges over all $j$ such that $\phi_j$ is conjugate self-dual with sign $(-1)^{n-1}$.
In particular, $S_{\phi}$ is an elementary abelian $2$-group.
We shall let $z_\phi$ denote the image of $-1 \in \GL_n(\CC)$ in $S_\phi$.
More explicitly, we have
\[
 z_{\phi} = (m_j a_j) \in \prod_j  (\ZZ / 2\ZZ) a_j.
\]

\subsection{Local Langlands correspondence}

The local Langlands correspondence for general linear groups, which was established by Harris--Taylor \cite{ht}, Henniart \cite{he}, and Scholze \cite{scholze}, is a certain bijection between $\Irr(\GL_n(E))$ and equivalence classes of $n$-dimensional representations of $\WD_E$.
This bijection satisfies natural properties which determine it uniquely.
For example, if $\pi$ is an irreducible smooth representation of $\GL_n(E)$ with central character $\omega_{\pi}$ and $\phi$ is the $n$-dimensional representation of $\WD_E$ associated to $\pi$, then 
\begin{itemize}
 \item $\omega_\pi = \det \phi$,
 \item $\pi$ is essentially square-integrable if and only if $\phi$ is irreducible,
 \item $\pi$ is tempered if and only if $\phi$ is tempered.
\end{itemize}

The local Langlands correspondence (as enhanced by Vogan \cite{v}) for unitary groups says that there is a canonical partition
\[  \Irr(\U(V^+)) \sqcup \Irr(\U(V^-))   =  \bigsqcup_{\phi}  \Pi_{\phi}, \]
where $V^+$ and $V^-$ are the $n$-dimensional $\varepsilon$-Hermitian spaces
with $\epsilon(V^+) = +1$ and $\epsilon(V^-) = -1$,
the disjoint union on the right-hand side runs over all equivalence classes of $L$-parameters $\phi$ for $\U(V^\pm)$, and $\Pi_{\phi}$ is a finite set of representations known as a Vogan $L$-packet.
We may decompose $\Pi_{\phi}$ as 
\[  \Pi_{\phi}  = \Pi_{\phi}^+\sqcup   \Pi_{\phi}^-, \]
where for $\epsilon = \pm 1$, $\Pi_{\phi}^{\epsilon}$ consists of the representations of $\U(V^{\epsilon})$ in $\Pi_{\phi}$.

\subsection{Whittaker data}

To describe the $L$-packet $\Pi_{\phi}$ more precisely, it is necessary to choose a Whittaker datum, which is a conjugacy class of pairs $(N, \psi_N)$, where
\begin{itemize}
\item $N$ is the unipotent radical of a Borel subgroup of the quasi-split unitary group $\U(V^+)$,
\item $\psi_N$ is a generic character of $N$.
\end{itemize}
Then relative to this datum, there is a canonical bijection
\[  J_{\psi_N}:   \Pi_{\phi}  \longleftrightarrow \Irr( S_{\phi}). \]
When $n$ is odd, such a datum is canonical.
When $n$ is even, as explained in \cite[\S 12]{ggp1}, it is determined by the choice of an $\N_{E/F}(E^{\times})$-orbit of nontrivial additive characters
\[
\begin{cases}
 \psi^E:E/F \rightarrow \CC^{\times} & \text{if $\varepsilon = +1$;} \\
 \psi:F \rightarrow \CC^{\times} & \text{if $\varepsilon = -1$.}
\end{cases}
\]
According to this choice, we write
\[
\begin{cases}
 J_{\psi^E} & \text{if $\varepsilon = +1$;} \\
 J_{\psi} & \text{if $\varepsilon = -1$}
\end{cases} 
\]
for $J_{\psi_N}$.
We formally adopt the same notation when $n$ is odd.
Suppose that $\varepsilon = +1$, so that $V^+$ and $V^-$ are Hermitian spaces.
Let $W^+ = \delta \cdot V^+$ be the space $V^+$ equipped with the skew-Hermitian form $\delta \cdot \langle \cdot, \cdot \rangle_{V^+}$.
Similarly, we define the skew-Hermitian space $W^- = \delta \cdot V^-$.
Then for $\epsilon = \pm 1$, $\U(V^\epsilon)$ and $\U(W^\epsilon)$ are physically equal.
For a given $\phi$, let $J_{\psi^E}$ and $J_{\psi}$ be the above bijections for $\U(V^{\pm})$ and $\U(W^{\pm})$ respectively.
One has:
\begin{itemize}
\item if $n$ is even, then
\[
 J_{\psi^E}  =  J_{\psi} \Longleftrightarrow \psi^E(x)  = \psi( \tfrac{1}{2}  \Tr_{E/F}(\delta x)), 
\]
\item if $n$ is odd, then $J_{\psi^E}  = J_{\psi}$. 
\end{itemize}

Having fixed the Whittaker datum $(N, \psi_N)$, we shall write $\pi(\phi, \eta)$ or simply $\pi(\eta)$ for the irreducible smooth representation in $\Pi_{\phi}$ corresponding to $\eta \in \Irr(S_{\phi})$ under the bijection $J_{\psi_N}$.
If $\phi$ is tempered, then for any Whittaker datum $(N, \psi'_N)$,
there is a unique $(N,\psi'_N)$-generic representation of $\U(V^+)$ in $\Pi_{\phi}$ by \cite[Lemme 7.10.1]{bp2},
and the irreducible characters of $S_{\phi}$ associated to these generic representations under the bijection $J_{\psi_N}$ are described as follows:
\begin{itemize}
 \item The unique $(N,\psi_N)$-generic representation of $\U(V^+)$ in $\Pi_\phi$ corresponds to the trivial character of $S_{\phi}$.
 \item When $n$ is even, there are precisely two Whittaker datum.
 If $(N, \psi_N')$ is not conjugate to $(N, \psi_N)$, then by \cite[\S 3]{ka},
 the unique $(N,\psi_N')$-generic representation of $\U(V^+)$ in $\Pi_\phi$ corresponds to the character $\eta_-$ of $S_{\phi}$ given by
\[
 \eta_-(a_j) = (-1)^{\dim \phi_j}.
\]
\end{itemize}
The character $\eta_-$ has a role even when $n$ is odd.
Indeed, if $n$ is odd, we may take $V^- = a \cdot V^+$, 
i.e.~the space $V^+$ equipped with the Hermitian form $a \cdot \langle \cdot, \cdot \rangle_{V^+}$,
where $a \in F^{\times} \smallsetminus \N_{E/F}(E^{\times})$.
Then $\U(V^+)$ and $\U(V^-)$ are physically equal.
Under this identification, we have
\[
 \Pi^+_{\phi} = \Pi^-_{\phi}
\]
for any $\phi$.
Let $\pi = \pi(\phi, \eta)$ be a representation of $\U(V^+)$ in $\Pi_{\phi}$.
If we regard $\pi$ as a representation of $\U(V^-)$ via the above identification,
then it has associated character $\eta \cdot \eta_-$.
In particular, if $\phi$ is tempered, then the unique $(N,\psi_N)$-generic representation of $\U(V^-)$ in $\Pi_\phi$ corresponds to $\eta_-$.

\subsection{Properties of the local Langlands correspondence}
\label{SS:llc}

We highlight some properties of the local Langlands correspondence which are used in this paper:

\begin{itemize}
\item
$\pi(\phi,\eta)$ is a representation of $\U(V^{\epsilon})$ if and only if $\eta(z_\phi)  = \epsilon$.

\item $\pi(\phi,\eta)$ is square-integrable if and only if $\phi$ is square-integrable.

\item $\pi(\phi,\eta)$ is tempered if and only if $\phi$ is tempered.

\item
If $\phi$ is tempered but not square-integrable, then we can write
\[
 \phi = \phi_1 \oplus \phi_0 \oplus (\phi_1^c)^\vee,
\]
where
\begin{itemize}
 \item $\phi_1$ is a $k$-dimensional irreducible representation of $\WD_E$ for some positive integer $k$,
 \item $\phi_0$ is a tempered $L$-parameter for $\U(V_0^{\pm})$,
 where $V_0^{\pm}$ are the $\varepsilon$-Hermitian spaces of dimension $n-2k$ over $E$.
\end{itemize}
Note that there is a natural embedding $S_{\phi_0} \hookrightarrow S_\phi$.
Let $\eta_0 \in \Irr(S_{\phi_0})$ and put $\epsilon = \eta_0(z_{\phi_0})$.
We can write
\[
 V^\epsilon = X \oplus V_0^\epsilon \oplus X^*,
\]
where $X$ and $X^*$ are $k$-dimensional totally isotropic subspaces of $V^\epsilon$ such that $X \oplus X^*$ is nondegenerate and orthogonal to $V_0^\epsilon$.
Let $P$ be the maximal parabolic subgroup of $\U(V^\epsilon)$ stabilizing $X$ and $M$ its Levi component stabilizing $X^*$, so that 
\[
 M \cong \GL(X) \times \U(V_0^\epsilon).
\]
Let $\tau$ be the irreducible (unitary) square-integrable representation of $\GL(X)$ associated to $\phi_1$, 
and let $\pi_0 = \pi(\phi_0, \eta_0)$ be the irreducible tempered representation of $\U(V_0^\epsilon)$ in $\Pi_{\phi_0}$ corresponding to $\eta_0$.
Then the induced representation $\Ind^{\U(V^\epsilon)}_P(\tau \otimes \pi_0)$ has a decomposition 
\[
 \Ind^{\U(V^\epsilon)}_P(\tau \otimes \pi_0) = \bigoplus_\eta \pi(\phi, \eta), 
\]
where the sum ranges over all $\eta \in \Irr(S_{\phi})$ such that $\eta|_{S_{\phi_0}} = \eta_0$.
Moreover, if $\phi_1$ is conjugate self-dual, let 
\[
 R(w, \tau \otimes \pi_0) \in \End_{\U(V^\epsilon)}(\Ind^{\U(V^\epsilon)}_P(\tau \otimes \pi_0))
\]
be the normalized intertwining operator defined in \S \ref{SS:normalization} below, where $w$ is the unique nontrivial element in the relative Weyl group for $M$.
Then the restriction of $R(w, \tau \otimes \pi_0)$ to $\pi(\phi, \eta)$ is the scalar multiplication by 
\begin{equation}
\label{eq:lir}
 \begin{cases}
 \epsilon^k \cdot \eta(a_1) & \text{if $\phi_1$ has sign $(-1)^{n-1}$;} \\
 \epsilon^k & \text{if $\phi_1$ has sign $(-1)^n$.}
 \end{cases}
\end{equation}
These properties follow from the definition of $\eta$, induction in stages \cite[\S 2.7]{kmsw}, and the local intertwining relation \cite[Theorem 3.4.3]{mok}, \cite[Theorem 2.6.2]{kmsw}.
We also remark that the factor $\epsilon^k$ arises from the splitting $s' : W_\psi(M,G) \rightarrow \pi_0(N_\psi(M,G))$ defined in \cite[\S 2.4.1]{kmsw}, which can be explicated by using an analog of Lemma \ref{lem:weyl} below for the dual group.

\item 
In general, we can write
\[
 \phi = \phi_1 \oplus \cdots \oplus \phi_r \oplus \phi_0
 \oplus (\phi_r^c)^\vee \oplus \cdots \oplus (\phi_1^c)^\vee,
\]
where
\begin{itemize}
 \item for $i = 1, \ldots, r$, $\phi_i$ is a $k_i$-dimensional representation of $\WD_E$ of the form $\phi_i = \phi_i' \otimes |\cdot|^{e_i}$ for some tempered representation $\phi_i'$ of $\WD_E$ and real number $e_i$ such that
\[
 e_1 > \cdots > e_r > 0,
\]
 \item $\phi_0$ is a tempered $L$-parameter for $\U(V_0^{\pm})$,
 where $V_0^{\pm}$ are the $\varepsilon$-Hermitian spaces of dimension $n-2(k_1+\cdots+k_r)$ over $E$.
\end{itemize}
Note that the natural map $S_{\phi_0} \rightarrow S_\phi$ is an isomorphism.
Let $\eta \in \Irr(S_{\phi})$ and put $\epsilon = \eta(z_\phi)$.
We can write
\[
 V^\epsilon = X_1 \oplus \cdots \oplus X_r \oplus V_0^\epsilon \oplus X^*_r \oplus \cdots \oplus X^*_1,
\]
where $X_i$ and $X_i^*$ are $k_i$-dimensional totally isotropic subspaces of $V^\epsilon$ such that $X_i \oplus X_i^*$ are nondegenerate, mutually orthogonal,
 and orthogonal to $V_0^\epsilon$.
Let $P$ be the parabolic subgroup of $\U(V^\epsilon)$ stabilizing the flag
\[
 X_1 \subset X_1 \oplus X_2 \subset \dots \subset X_1 \oplus \cdots \oplus X_r
\]
and $M$ its Levi component stabilizing the flag
\[
 X^*_1 \subset X^*_1 \oplus X^*_2 \subset \dots \subset X^*_1 \oplus \cdots \oplus X^*_r,
\]
so that 
\[
 M \cong \GL(X_1) \times \cdots \times \GL(X_r) \times \U(V_0^\epsilon).
\]
Then $\pi(\phi, \eta)$ is the unique irreducible quotient of the standard module
\[
 \Ind^{\U(V^\epsilon)}_P(\tau_1 \otimes \cdots \otimes \tau_r \otimes \pi_0), 
\]
where for $i = 1, \ldots, r$, $\tau_i$ is the irreducible essentially tempered representation of $\GL(X_i)$ associated to $\phi_i$,
and $\pi_0 = \pi(\phi_0, \eta_0)$ is the irreducible tempered representation of $\U(V_0^\epsilon)$ in $\Pi_{\phi_0}$ corresponding to $\eta_0 := \eta|_{S_{\phi_0}} \in \Irr(S_{\phi_0})$.

\item 
If $\pi = \pi(\phi, \eta)$, then the contragredient representation $\pi^\vee$ of $\pi$ has $L$-parameter $\phi^{\vee}$ and associated character $\eta_{\pi^\vee} = \eta \cdot \nu$, where 
\[  
\nu(a_j)  = \begin{cases}
\omega_{E/F}(-1)^{\dim \phi_j} & \text{if $n$ is even;}  \\
1 & \text{if $n$ is odd.} \end{cases} \]
Note that the component groups $S_{\phi}$ and $S_{\phi^\vee}$ are canonically identified.
In the case of unitary groups, this property follows from a result of Kaletha \cite[\S 4]{ka}.

\end{itemize}

\section{\textbf{Gross--Prasad conjecture}}

In this section, we explicate the statement of the Gross--Prasad conjecture for unitary groups. In particular, we recall the definition of the distinguished character $\eta$ of the component group.

\subsection{Pairs of spaces}

For $\epsilon = \pm 1$, let $V_n^\epsilon$ denote the $n$-dimensional Hermitian space with $\epsilon(V_n^\epsilon) = \epsilon$ and $W_n^\epsilon$ the $n$-dimensional skew-Hermitian space with $\epsilon(W_n^\epsilon) = \epsilon$, so that $W_n^\epsilon = \delta \cdot V_n^\epsilon$.
For the Gross--Prasad conjecture, we consider the pair of spaces:
\[  V_n^+ \subset V_{n+1}^+ \quad \text{or} \quad W_n^+ = W_n^+. \]
Then the relevant pure inner form (other than itself) is 
\[  V_n^- \subset V_{n+1}^-  \quad \text{or} \quad W_n^- = W_n^- \]
and observe that
\[  V_{n+1}^{\epsilon}/V_n^{\epsilon}   \cong L_{(-1)^n}, \]
where for $a \in F^{\times}$, $L_a$ denotes the Hermitian line with form $a \cdot \N_{E/F}$.
We have the groups
\[  G_n^\epsilon =  \U(V_n^\epsilon) \times \U(V_{n+1}^\epsilon) \quad \text{or} \quad \U(W_n^\epsilon) \times \U(W_n^\epsilon) \]
and
\[   H_n^\epsilon = \U(V_n^\epsilon) \quad \text{or} \quad \U(W_n^\epsilon), \]
and the embedding
\[ \Delta : H_n^\epsilon \hookrightarrow G_n^\epsilon.\]

We also have the Langlands--Vogan parametrization (depending on the choice of the Whittaker datum) relative to the fixed pair of spaces.
For an $L$-parameter $\phi = \phi^{\diamondsuit} \times \phi^{\heartsuit}$ for $G_n^\pm$, the component group is:
\[  S_{\phi}  = S_{\phi^{\diamondsuit}}  \times S_{\phi^{\heartsuit}}. \]
In particular, under the local Langlands correspondence, the representation $\pi(\eta) \in \Pi_{\phi}$ is a representation of a relevant pure inner form if and only if
\[  \eta(z_{\phi^{\diamondsuit}}, z_{\phi^{\heartsuit}})  = 1, \]
and $\pi(\eta)$ is a representation of $G_n^{\epsilon}$ if and only if 
\[  \eta(z_{\phi^{\diamondsuit}}, 1)  = \eta(1,z_{\phi^{\heartsuit}})  = \epsilon. \]

\subsection{The distinguished character $\eta$} 
\label{SS:eta}

We shall now define a distinguished character $\eta \in \Irr(S_\phi)$ when $\phi =  \phi^{\diamondsuit} \times \phi^{\heartsuit}$.
Writing 
\[  S_{\phi^{\diamondsuit}}  = \prod_i  (\ZZ /2 \ZZ) a_i \quad \text{and} \quad S_{\phi^{\heartsuit}} = \prod_j  (\ZZ / 2\ZZ) b_j, \]
we thus need to specify the signs $\eta(a_i) = \pm 1$ and $\eta(b_j) = \pm 1$.
We consider the Bessel and Fourier--Jacobi cases separately.

\begin{itemize}

\item \textbf{Bessel case.}
We fix a nontrivial character $\psi^E$ of $E/F$ which determines the local Langlands correspondence for the even unitary group in $G_n^\epsilon = \U(V_n^\epsilon) \times \U(V_{n+1}^\epsilon)$.
We set $\psi^E_{-2}(x)  = \psi^E(-2 x)$ and define:
\[
\begin{cases}
 \eta^{\spadesuit}(a_i)  = \epsilon(\frac{1}{2}, \phi^{\diamondsuit}_i \otimes \phi^{\heartsuit}, \psi^E_{-2}); \\
 \eta^{\spadesuit}(b_j)  = \epsilon(\frac{1}{2},  \phi^{\diamondsuit} \otimes \phi^{\heartsuit}_{j}, \psi^E_{-2}).
\end{cases}
\]

\item \textbf{Fourier--Jacobi case.}
In this case, we need to fix a nontrivial character $\psi$ of $F$ and a character $\chi$ of $E^{\times}$ with $\chi|_{F^{\times}} = \omega_{E/F}$ to specify the Weil representation $\nu = \omega_{\psi,\chi, W_n^\epsilon}$ of $\U(W_n^\epsilon)$.
The recipe for the distinguished character $\eta^{\clubsuit}$ of $S_{\phi}$ depends on the parity of $n = \dim W_n^\epsilon$.
\begin{itemize}
\item If $n$ is odd, recall that $\det W^+_n  \in  \delta \cdot \N_{E/F}(E^{\times})$ and define
\[ \begin{cases}
 \eta^{\clubsuit}(a_i) = \epsilon(\frac{1}{2}, \phi^{\diamondsuit}_{i} \otimes \phi^{\heartsuit} \otimes \chi^{-1},  \psi_2^E);  \\
  \eta^{\clubsuit}(b_j)  = \epsilon(\frac{1}{2}, \phi^{\diamondsuit} \otimes \phi^{\heartsuit}_{j} \otimes \chi^{-1}, \psi_2^E), \end{cases} \]
  where
  \[  \psi^E_2(x)  = \psi(\Tr_{E/F}(\delta x)).  \]
\item If $n$ is even, the fixed character $\psi$  is used to fix the local Langlands correspondence for $\U(W_n^\epsilon)$.
 We set
\[  \begin{cases}
\eta^{\clubsuit}(a_i) = \epsilon(\frac{1}{2}, \phi^{\diamondsuit}_i \otimes \phi^{\heartsuit} \otimes \chi^{-1},  \psi^E); \\
  \eta^{\clubsuit}(b_j)  = \epsilon(\frac{1}{2}, \phi^{\diamondsuit} \otimes \phi^{\heartsuit}_{j} \otimes \chi^{-1}, \psi^E), \end{cases} \]
where the $\epsilon$-factors are defined using any nontrivial additive character $\psi^E$ of $E/F$.
(The result is independent of this choice.)
\end{itemize}
\end{itemize}

We refer the reader to \cite[\S 18]{ggp1} for a discussion of the various subtleties in the definition of $\eta^{\spadesuit}$ or $\eta^{\clubsuit}$.

\subsection{Conjectures (B) and (FJ)}

Let us formally state the statements (B)$_n$ and (FJ)$_n$:

\begin{itemize}

\item[(B)$_n$]
Given a \emph{tempered} $L$-parameter $\phi$ for $G_n^\pm = \U(V_n^\pm) \times \U(V_{n+1}^\pm)$ and a representation $\pi(\eta) \in \Pi_{\phi}$ of a relevant pure inner form $G_n^\epsilon$,
\[  \Hom_{\Delta H_n^\epsilon}(\pi(\eta), \CC) \ne 0 \Longleftrightarrow  \eta = \eta^{\spadesuit}. \]

\item[(FJ)$_n$]
Given a \emph{tempered} $L$-parameter $\phi$ for $G_n^\pm = \U(W_n^\pm) \times \U(W_n^\pm)$ and a representation $\pi(\eta) \in \Pi_{\phi}$ of a relevant pure inner form $G_n^\epsilon$, 
\[ \Hom_{\Delta H_n^\epsilon}(\pi(\eta), \nu) \ne 0 \Longleftrightarrow \eta = \eta^{\clubsuit}. \]

\end{itemize}

We shall denote by (B) the collection of statements (B)$_n$ for all $n \ge 0$, and by (FJ) the collection of statements (FJ)$_n$ for all $n \ge 0$.
We stress that both (B) and (FJ) are considered only for tempered representations in this paper (except in \S \ref{s:generic} where we treat the case of generic $L$-parameters).

\section{\textbf{Local theta correspondence and Prasad's conjectures}}
\label{S:Pconj}

In this section, we explicate the statement of Prasad's conjectures on the local theta correspondence for unitary groups of (almost) equal rank.

\subsection{Weil representations}
\label{SS:Weil}

Let $V$ be a Hermitian space and $W$ a skew-Hermitian space. To consider the theta correspondence for the reductive dual pair $\U(V) \times \U(W)$, one requires certain additional data:

\begin{enumerate}
\item a nontrivial additive character $\psi$ of $F$;

\item a pair of characters $\chi_V$ and $\chi_W$ of $E^{\times}$ such that
\[  \chi_V |_{F^{\times}}  = \omega_{E/F}^{\dim V} \quad \text{and} \quad \chi_W|_{F^{\times}} = \omega_{E/F}^{\dim W}. \]
One way to fix such a pair is simply to fix a character $\chi$ of $E^{\times}$ such that $\chi|_{F^{\times}} = \omega_{E/F}$ and then set 
\[  \chi_V = \chi^{\dim V} \quad \text{and} \quad \chi_W = \chi^{\dim W}. \]

\item a trace zero element $\delta \in E^{\times}$.
\end{enumerate}
To elaborate, the tensor product $V \otimes W$ has a natural symplectic form defined by 
\[  \langle v_1 \otimes w_1,  v_2 \otimes w_2 \rangle =  \Tr_{E/F}(\langle v_1,v_2 \rangle_V \cdot \langle w_1, w_2 \rangle_W). \]
Then there is a natural map
\[   \U(V) \times \U(W) \longrightarrow \Sp(V \otimes W). \]
One has the metaplectic $S^1$-cover $\Mp(V \otimes W)$ of $\Sp(V \otimes W)$, and the character $\psi$ (together with the form $\langle \cdot, \cdot \rangle$ on $V \otimes W$) determines a Weil representation $\omega_{\psi}$ of $\Mp(V \otimes W)$.
The data $(\psi, \chi_V,\chi_W, \delta)$ then allows one to specify a splitting of the metaplectic cover over $\U(V) \times \U(W)$, as shown in \cite{k}, \cite{hks}.
In fact, it does not depend on the choice of $\delta$.

Hence, we have a Weil representation $\omega_{\psi, \chi_V, \chi_W, V,W}$ of $\U(V) \times \U(W)$. The Weil representation $\omega_{\psi, \chi_V, \chi_W, V,W}$ depends only on the orbit of $\psi$ under $\N_{E/F}(E^{\times})$.

\subsection{Local theta correspondence}

Given an irreducible smooth representation $\pi$ of $\U(W)$, the maximal $\pi$-isotypic quotient of $\omega_{\psi, \chi_V, \chi_W, V,W}$ is of the form
\[
 \Theta_{\psi, \chi_V, \chi_W, V,W}(\pi) \boxtimes \pi
\]
for some smooth representation $\Theta_{\psi, \chi_V, \chi_W, V,W}(\pi)$ of $\U(V)$ of finite length.
By the Howe duality, which was proved by Waldspurger \cite{w} for $p \ne 2$ and by the first author and Takeda \cite{gt1}, \cite{gt2} for any $p$ (so that the assumption $p \ne 2$ can be removed from the results of \cite{gi} stated below),
the maximal semisimple quotient $\theta_{\psi, \chi_V, \chi_W, V,W}(\pi)$ of $\Theta_{\psi, \chi_V, \chi_W, V,W}(\pi)$ is either zero or irreducible.
If $\chi_V$ and $\chi_W$ are clear from the context, we simply write $\Theta_{\psi, V,W}(\pi) = \Theta_{\psi, \chi_V, \chi_W, V,W}(\pi)$ and $\theta_{\psi, V,W}(\pi) = \theta_{\psi, \chi_V, \chi_W, V,W}(\pi)$.

In this paper, we consider the theta correspondence for $\U(V) \times \U(W)$ with $|\dim V - \dim W| \le 1$. We will state two conjectures of D.~Prasad which describe the local theta correspondence in terms of the local Langlands correspondence. 

\subsection{Equal rank case} 

We first consider the case $\dim V = \dim W = n$. We shall consider the theta correspondence for $\U(V_n^\epsilon) \times \U(W_n^{\epsilon'})$.
The following summarises some results of \cite{gi}:

\begin{thm}  \label{T:equal}
Let $\phi$ be an $L$-parameter for $\U(W_n^\pm)$.
Then we have:
\begin{enumerate}
\item For any fixed $\pi \in \Pi_{\phi}^{\epsilon'}$, exactly one of $\Theta_{\psi, V_n^+,W_n^{\epsilon'}}(\pi)$ or $\Theta_{\psi, V_n^-,W_n^{\epsilon'}}(\pi)$ is nonzero.

\item $\Theta_{\psi, V_n^\epsilon, W_n^{\epsilon'}}(\pi) \ne 0$ if and only if
\[  \epsilon(\tfrac{1}{2}, \phi  \otimes  \chi_V^{-1}, \psi^E_2)  = \epsilon \cdot \epsilon', \]
where
\[  \psi^E_2(x)  = \psi(\Tr_{E/F}(\delta x)).  \]

\item If $\Theta_{\psi, V_n^{\epsilon},W_n^{\epsilon'}}(\pi)$ is nonzero, then $\theta_{\psi, V_n^{\epsilon},W_n^{\epsilon'}}(\pi)$ has $L$-parameter
\[   \theta(\phi)  =  \phi \otimes \chi_V^{-1} \chi_W. \]

\item The theta correspondence $\pi \mapsto   \theta_{\psi, V_n^{\epsilon}, W_n^{\epsilon'}}(\pi)$ gives a bijection
\[  \Pi_{\phi}  \longleftrightarrow   \Pi_{\theta(\phi)}. \]

\item If $\phi$ is tempered and $\Theta_{\psi, V_n^{\epsilon}, W_n^{\epsilon'}}(\pi)$ is nonzero, then $\Theta_{\psi, V_n^{\epsilon}, W_n^{\epsilon'}}(\pi)$ is irreducible.

\end{enumerate}
\end{thm}

\subsection{Conjecture (P1)}

After the above theorem, the remaining question is to specify the bijection of Vogan $L$-packets given in (iv).
We shall do this using the bijections
\[  J_{\psi}:  \Pi_{\phi} \longleftrightarrow \Irr(S_{\phi})   \quad \text{and} \quad J_{\psi^E}:  \Pi_{\theta(\phi)} \longleftrightarrow \Irr(S_{\theta(\phi)}), \]
where
\begin{equation}  \label{E:psipsiE2}
  \psi^E(x)  = \psi(\tfrac{1}{2} \Tr_{E/F}(\delta x)). \end{equation}
Note that the bijections $J_{\psi}$ and $J_{\psi^E}$ are independent of $\psi$ and $\psi^E$ when $n$ is odd, but when $n$ is even, they do depend on these additive characters and it is crucial for $\psi$ and $\psi^E$ to be related as in \eqref{E:psipsiE2} for what follows to hold. 

Having fixed the bijections $J_{\psi}$ and $J_{\psi^E}$, we need to describe the bijection
\begin{align*}
 \Irr(S_{\phi}) & \longleftrightarrow \Irr(S_{\theta(\phi)}) \\
 \eta & \longleftrightarrow \theta(\eta)
\end{align*}
induced by the theta correspondence.
Note that the component groups $S_{\phi}$ and $S_{\theta(\phi)}$ are canonically identified, since $\theta(\phi)$ is simply a twist of $\phi$ by a conjugate orthogonal character. 

Now the first conjecture of Prasad states the following.

\begin{itemize}

\item[(P1)$_n$]
Let $\phi$ be an $L$-parameter for $\U(W_n^\pm)$ and let $\eta \in \Irr(S_\phi)$. Suppose that
\[  S_{\phi} = S_{\theta(\phi)} = \prod_i (\ZZ /2\ZZ) a_i. \]
Then, relative to $J_{\psi}$ and $J_{\psi^E}$ as above, 
\[  \theta(\eta)(a_i) / \eta(a_i)  = \epsilon(\tfrac{1}{2}, \phi_i \otimes  \chi_V^{-1},  \psi^E_2), \]
where 
\[  \psi^E_2(x)  = \psi(\Tr_{E/F}(\delta x)).  \]
\end{itemize}

We shall denote by (P1) the collection of all statements (P1)$_n$ for all $n \ge 0$. Note that we consider (P1) for all $L$-parameters, and not just tempered ones. However, we note:

\begin{prop}  \label{P:P1nt}
Suppose that $(\emph{P1})_k$ holds for all tempered $L$-parameters for all $k < n$. Then $(\emph{P1})_k$ holds for all nontempered $L$-parameters for all $k \le n$.
\end{prop}

\begin{proof}
This follows from the analog of \cite[Theorem 8.1(iii)]{gs} for unitary groups.
\end{proof}

Moreover, the following is a corollary of Theorem \ref{T:equal}(ii):

\begin{cor}  \label{C:stable}
The statement $(\emph{P1})_n$ holds if $\phi$ is irreducible.
\end{cor}

\subsection{Almost equal rank case}

Now we consider the case $\dim V = n+1$ and $\dim W = n$. We shall consider the theta correspondence for $\U(V_{n+1}^\epsilon) \times \U(W_n^{\epsilon'})$. The following summarises some results of \cite{gi}:

\begin{thm}  \label{T:al-equal}

Let $\phi$ be an $L$-parameter for $\U(W_n^{\pm})$. Then we have:

\begin{enumerate}

\item Suppose that $\phi$ does not contain $\chi_V$.

\begin{enumerate}

 \item For any $\pi \in \Pi_{\phi}^{\epsilon'}$, 
 $\Theta_{\psi, V_{n+1}^{\epsilon}, W_n^{\epsilon'}}(\pi)$ is nonzero and $\theta_{\psi, V_{n+1}^{\epsilon}, W_n^{\epsilon'}}(\pi)$ has $L$-parameter
 \[  \theta(\phi) = (\phi \otimes \chi_V^{-1} \chi_W) \oplus  \chi_W. \]

 \item For each $\epsilon = \pm 1$, the theta correspondence $\pi \mapsto \theta_{\psi, V_{n+1}^{\epsilon}, W_n^{\epsilon'}}(\pi)$ gives a bijection
\[  \Pi_{\phi} \longleftrightarrow \Pi_{\theta(\phi)}^{\epsilon}.  \]  

\end{enumerate} 

\item Suppose that $\phi$ contains $\chi_V$.

\begin{enumerate}

\item For any fixed $\pi \in \Pi_{\phi}^{\epsilon'}$, exactly one of $\Theta_{\psi, V_{n+1}^+, W_n^{\epsilon'}}(\pi)$ or $\Theta_{\psi, V_{n+1}^-, W_n^{\epsilon'}}(\pi)$ is nonzero.

\item If $\Theta_{\psi, V_{n+1}^{\epsilon}, W_n^{\epsilon'}}(\pi)$ is nonzero, then $\theta_{\psi, V_{n+1}^{\epsilon}, W_n^{\epsilon'}}(\pi)$ has $L$-parameter
 \[  \theta(\phi) = (\phi \otimes \chi_V^{-1} \chi_W) \oplus  \chi_W. \]

\item The theta correspondence $\pi \mapsto \theta_{\psi, V_{n+1}^{\epsilon}, W_n^{\epsilon'}}(\pi)$ gives a bijection
 \[  \Pi_{\phi} \longleftrightarrow \Pi_{\theta(\phi)}. \]

\end{enumerate} 

\item If $\phi$ is tempered and $\Theta_{\psi, V_{n+1}^{\epsilon}, W_n^{\epsilon'}}(\pi)$ is nonzero, then $\Theta_{\psi, V_{n+1}^{\epsilon}, W_n^{\epsilon'}}(\pi)$ is irreducible.

\end{enumerate}

\end{thm}

\subsection{Conjecture (P2)} After the above theorem, it remains to specify the bijections given in (i)(b) and (ii)(c). 
As in the case of (P1), we shall do this using the bijections
 \[  J_{\psi}:  \Pi_{\phi} \longleftrightarrow \Irr(S_{\phi})   \quad \text{and} \quad J_{\psi^E}:  \Pi_{\theta(\phi)} \longleftrightarrow \Irr(S_{\theta(\phi)}), \]
where 
\[  \psi^E(x)  = \psi(\tfrac{1}{2} \Tr_{E/F}(\delta x)). \]
Note that $J_{\psi}$ is independent of $\psi$ when $n$ is odd, whereas $J_{\psi^E}$ is independent of $\psi^E$ when $n$ is even. 

Observe that:

\begin{itemize}

\item If $\phi$ does not contain $\chi_V$, then 
\[  S_{\theta(\phi)}  = S_{\phi} \times (\ZZ /2 \ZZ) a_0, \]
where the extra copy of $\ZZ/2\ZZ$ arises from the summand $\chi_W$ in $\theta(\phi)$. Thus, for each $\epsilon$, one has a canonical bijection
\begin{align*}
 \Irr(S_{\phi}) & \longleftrightarrow \Irr^{\epsilon}(S_{\theta(\phi)}) \\
 \eta & \longleftrightarrow \theta(\eta)
\end{align*}
induced by the theta correspondence, where $\Irr^{\epsilon}(S_{\theta(\phi)})$ is the set of irreducible characters $\eta'$ of $S_{\theta(\phi)}$ such that $\eta'(z_{\theta(\phi)}) = \epsilon$.

\item On the other hand, if $\phi$ contains $\chi_V$, then $\phi \otimes \chi_V^{-1} \chi_W$ contains $\chi_W$, so that
 \[
   S_{\theta(\phi)}  =  S_{\phi}. \]
Thus, one has a canonical bijection
\begin{align*}
  \Irr(S_{\phi}) & \longleftrightarrow \Irr(S_{\theta(\phi)}) \\
  \eta & \longleftrightarrow \theta(\eta)
\end{align*}
induced by the theta correspondence.

\end{itemize}

Now we can state the second conjecture of Prasad. 

\begin{itemize}
\item[(P2)$_n$]  

Let $\phi$ be an $L$-parameter for $\U(W_n^\pm)$ and let $\eta \in \Irr(S_\phi)$.
Fix the bijections $J_{\psi}$ and $J_{\psi^E}$ as above.

\begin{itemize}
\item If $\phi$ does not contain $\chi_V$, then $\theta(\eta)$ is the unique irreducible character in $\Irr^{\epsilon}(S_{\theta(\phi)})$ such that
 \[  \theta(\eta)|_{S_{\phi}}  =\eta. \]

\item On the other hand, if $\phi$ contains $\chi_V$, then
\[
 \theta(\eta) = \eta.
\]

\end{itemize}

\end{itemize}

We shall denote by (P2) the collection of all the statements (P2)$_n$ for all $n \ge 0$. Note that we consider (P2) for all $L$-parameters, and not just tempered ones. However, we note:

\begin{prop}  \label{P:P2nt}
Suppose that $(\emph{P2})_k$ holds for all tempered $L$-parameters for all $k < n$. Then $(\emph{P2})_k$ holds for all nontempered $L$-parameters for all $k \le n$.
\end{prop}

\begin{proof}
This follows from \cite[Proposition C.4(ii)]{gi}.
\end{proof}

\section{\textbf{(B) $+$ (P2) $\Longrightarrow$ (FJ) $+$ (P1)}}
\label{s:see-saw}

In this section, we shall show that Conjectures (FJ) and (P1) follow from Conjectures (B) and (P2), together with Theorems \ref{T:equal} and \ref{T:al-equal}.

Suppose that we are given tempered $L$-parameters $\phi^{\diamondsuit}$ and $\phi^{\heartsuit}$ for $\U(W_n^\pm)$. Let
\[ \pi^{\diamondsuit} = \pi(\eta^{\diamondsuit}) \in \Pi_{\phi^{\diamondsuit}}^{\epsilon'} \quad \text{and} \quad 
\pi^{\heartsuit} = \pi(\eta^{\heartsuit}) \in \Pi_{\phi^{\heartsuit}}^{\epsilon'} \]
be representations such that
\[  \Hom_{\U(W_n^{\epsilon'})}(\pi^{\diamondsuit} \otimes \pi^{\heartsuit}, \omega_{\psi, \chi, W_n^{\epsilon'}}) \ne 0. \]
We first show that
\[ \eta^{\diamondsuit} \otimes \eta^{\heartsuit} = \eta^{\clubsuit}. \]
Since the representations involved are unitary (as $\phi^{\diamondsuit}$ and $\phi^{\heartsuit}$ are tempered), 
\[ \Hom_{\U(W_n^{\epsilon'})}(\pi^{\diamondsuit} \otimes \pi^{\heartsuit}, \omega_{\psi, \chi, W_n^{\epsilon'}}) \ne 0 \]
if and only if
\[ \Hom_{\U(W_n^{\epsilon'})}( (\pi^{\diamondsuit})^{\vee} \otimes \omega_{\psi, \chi, W_n^{\epsilon'}},  \pi^{\heartsuit}) \ne 0. \]

\subsection{See-Saw}

Now we consider the see-saw diagram (for an $\epsilon$ to be determined soon):
\[
 \xymatrix{
  \U(W_n^{\epsilon'})  \times \U(W_n^{\epsilon'})  \ar@{-}[dr] \ar@{-}[d] & \U(V^{\epsilon}_{n+1}) \ar@{-}[d] \\
  \U(W_n^{\epsilon'}) \ar@{-}[ur] &  \U(V^{\epsilon}_n) \times \U(L_{(-1)^n})}.
\]  
We shall consider the local theta correspondence for the above see-saw diagram. For this, we need to specify precisely the data used in setting up the theta correspondence. More precisely, for the dual pair $\U(V_{n+1}^\epsilon) \times \U(W_n^{\epsilon'})$, we shall use the characters
\[  \chi_{V_{n+1}^\epsilon} = \chi^{n+(-1)^n} \quad \text{and} \quad \chi_{W_n^{\epsilon'}} = \chi^n, \]
and for the dual pair $\U(V_n^\epsilon) \times \U(W_n^{\epsilon'})$, we use
\[  \chi_{V_n^\epsilon}  = \chi_{W_n^{\epsilon'}}  = \chi^n. \]
Then for the dual pair $\U(L_{(-1)^n}) \times \U(W_n^{\epsilon'})$, we have no choice but to use
\[  \chi_{L_{(-1)^n}}  = \chi^{(-1)^n} \quad \text{and} \quad \chi_{W_n^{\epsilon'}}  = \chi^n. \]
In particular, the restriction of $\omega_{\psi, \chi_{L_{(-1)^n}}, \chi_{W_n^{\epsilon'}}, L_{(-1)^n}, W_n^{\epsilon'}}$ to $\U(W_n^{\epsilon'})$ is equal to
\[  \begin{cases}
\omega_{\psi, \chi, W_n^{\epsilon'}} & \text{if $n$ is even;} \\
\omega_{\psi, \chi, W_n^{\epsilon'}}^{\vee} & \text{if $n$ is odd.} \end{cases} \]
In any case, having fixed these normalizations, we shall suppress them from the notation for simplicity.

Because of the above differences for even and odd $n$, it will now be convenient to treat the even and odd cases separately. 

\subsection{Even case}

Assume first that $n$ is even. 
By Theorem \ref{T:equal}, we may choose $\sigma \in \Irr(\U(V_n^{\epsilon}))$ such that 
\[  \Theta_{\psi, V_n^{\epsilon}, W_n^{\epsilon'}}(\sigma) =  (\pi^{\diamondsuit})^{\vee}. \] 
This uniquely determines $\epsilon$. Moreover, by Theorem \ref{T:equal}, we know that $\sigma$ has $L$-parameter
\[  \phi_{\sigma}  =  (\phi^{\diamondsuit})^{\vee}, \]
since the $L$-parameter of $ (\pi^{\diamondsuit})^{\vee}$ is $(\phi^{\diamondsuit})^{\vee}$.

Taking the representation $\pi^{\heartsuit}$ on $\U(W_n^{\epsilon'})$ and the representation $\sigma$ on $\U(V_n^{\epsilon})$, the resulting see-saw identity reads:
\[  
 0 \ne \Hom_{\U(W_n^{\epsilon'})}( (\pi^{\diamondsuit})^{\vee} \otimes \omega_{\psi, \chi, W_n^{\epsilon'}}, \pi^{\heartsuit}) 
   = \Hom_{\U(V_n^{\epsilon})}( \Theta_{\psi, V_{n+1}^{\epsilon}, W_n^{\epsilon'}}(\pi^{\heartsuit}),  \sigma). \]
By Theorem \ref{T:al-equal}, 
\[ \tau := \Theta_{\psi, V_{n+1}^{\epsilon}, W_n^{\epsilon'}}(\pi^{\heartsuit}) \]
has $L$-parameter
\[  \phi_{\tau} =  (\phi^{\heartsuit} \otimes \chi^{-1}) \oplus \chi^n. \]

Recall that we have used the character $\psi$ to fix the local Langlands correspondence for $\U(W_n^{\epsilon'})$.
The component group $S_{\phi^{\heartsuit}}$ is of the form 
\[  S_{\phi^{\heartsuit}} = \prod_j  (\ZZ / 2\ZZ) b_j \]
and there is a natural embedding $S_{\phi^{\heartsuit}} \hookrightarrow S_{\phi_\tau}$.
Now, by (P2), the representation $\tau$ has associated character $\eta_{\tau} \in \Irr(S_{\phi_{\tau}})$ which satisfies:
\[  \eta_{\tau}   = \eta^{\heartsuit} \quad \text{on} \quad S_{\phi^{\heartsuit}}. \]
On the other hand, by (B), one knows exactly what $\eta_{\tau}$ is. Namely, (B) gives: 
\[  \eta_{\tau}(b_j)  = \epsilon(\tfrac{1}{2}, \phi_{\sigma}^{\vee} \otimes \phi_j^{\heartsuit} \otimes \chi^{-1}, \psi^E) = 
\epsilon(\tfrac{1}{2}, \phi^{\diamondsuit} \otimes \phi_j^{\heartsuit} \otimes \chi^{-1}, \psi^E) = \eta^{\clubsuit}(b_j), \]
where $\psi^E$ is any nontrivial character of $E/F$. Thus, we deduce that
\[  \eta^{\heartsuit}   =   \eta^{\clubsuit} \quad \text{on} \quad S_{\phi^{\heartsuit}}. \]
Now of course we could reverse the role of $\pi^{\diamondsuit}$ and $\pi^{\heartsuit}$ in the above argument. Then we conclude that
\[  \eta^{\diamondsuit} \otimes \eta^{\heartsuit}  =  \eta^{\clubsuit} \]
as desired.

\subsection{Odd case}

Now suppose that $n$ is odd. Then we use the character  
\[  \psi^E(x) = \psi(\tfrac{1}{2} \Tr_{E/F}(\delta x)) \]
of $E/F$ to specify the local Langlands correspondence for $\U(V_{n+1}^{\epsilon})$.
By Theorem \ref{T:equal}, we may choose $\sigma \in \Irr(\U(V_n^{\epsilon}))$ such that 
\[  \Theta_{\psi, V_n^{\epsilon},  W_n^{\epsilon'}}(\sigma)  =  \pi^{\diamondsuit}. \] 
This uniquely determines $\epsilon$. Moreover, by Theorem \ref{T:equal}, we know that $\sigma$ has $L$-parameter
\[  \phi_{\sigma} =  \phi^{\diamondsuit}. \]

Taking the representation $(\pi^{\heartsuit})^{\vee}$ on $\U(W_n^{\epsilon'})$ and the representation $\sigma$ on $\U(V_n^{\epsilon})$, the resulting see-saw identity reads:
 \[  
 0 \ne \Hom_{\U(W_n^{\epsilon'})}( \pi^{\diamondsuit} \otimes \omega_{\psi, \chi, W_n^{\epsilon'}}^{\vee}, (\pi^{\heartsuit})^{\vee}) 
  = \Hom_{\U(V_n^{\epsilon})}( \Theta_{\psi, V_{n+1}^{\epsilon}, W_n^{\epsilon'}}((\pi^{\heartsuit})^{\vee}),  \sigma). \]
By Theorem \ref{T:al-equal}, 
 \[ \tau := \Theta_{\psi, V_{n+1}^{\epsilon}, W_n^{\epsilon'}}((\pi^{\heartsuit})^{\vee}) \]
has $L$-parameter
\[  \phi_{\tau} = ((\phi^{\heartsuit})^{\vee} \otimes  \chi) \oplus \chi^n. \]

Now by (P2), the representation $\tau$ has associated character $\eta_{\tau} \in \Irr(S_{\phi_{\tau}})$ satisfying:
\[  \eta_{\tau}  = 
 \eta^{\heartsuit} \quad \text{on} \quad S_{\phi^{\heartsuit}} = S_{(\phi^{\heartsuit})^{\vee} \otimes \chi}. \]
On the other hand, by (B), we know that
\[  \eta_{\tau}(b_j)  = \epsilon(\tfrac{1}{2}, \phi_\sigma^\vee \otimes (\phi_j^{\heartsuit})^{\vee} \otimes \chi, \psi^E_{-2})
=  \epsilon(\tfrac{1}{2}, \phi^{\diamondsuit} \otimes \phi^{\heartsuit}_j \otimes \chi^{-1}, \psi^E_2) = \eta^{\clubsuit}(b_j). \]
Hence, we conclude that
\[  \eta^{\heartsuit}  = \eta^{\clubsuit}  \quad \text{on} \quad S_{\phi^{\heartsuit}}. \]
Reversing the role of $\pi^{\diamondsuit}$ and $\pi^{\heartsuit}$ in the above argument, we conclude that
\[  \eta^{\diamondsuit} \otimes \eta^{\heartsuit}  =  \eta^{\clubsuit} \]
as desired.

\subsection{Proof of (FJ)} 

At this point, we have shown that if 
\[  \Hom_{\U(W_n^{\epsilon'})}(\pi^{\diamondsuit} \otimes \pi^{\heartsuit}, \omega_{\psi, \chi, W_n^{\epsilon'}}) \ne 0, \]
then  $\eta^{\diamondsuit} \otimes \eta^{\heartsuit} $ is equal to the distinguished character $\eta^{\clubsuit}$. To complete the proof of (FJ), it remains to show that the above Hom space is nonzero for some $\epsilon'$ and pair of representations $(\pi^{\diamondsuit}, \pi^{\heartsuit}) \in \Pi^{\epsilon'}_{\phi^{\diamondsuit}}  \times \Pi^{\epsilon'}_{\phi^{\heartsuit}}$.
This will follow from the above see-saw diagram, Theorems \ref{T:equal} and \ref{T:al-equal}. Let us illustrate this in the case when $n$ is even; the case when $n$ is odd is similar. 

Consider the tempered $L$-parameters $\phi:= (\phi^{\heartsuit} \otimes \chi^{-1}) \oplus \chi^n$ for $\U(V_{n+1}^{\pm})$ and $\phi' := (\phi^{\diamondsuit})^{\vee}$ for $\U(V_n^{\pm})$. By (B), there is a pair of representations
\[  (\tau  , \tau') \in \Pi^{\epsilon}_{\phi} \times \Pi^{\epsilon}_{\phi'}  \]
such that
\[  \Hom_{\U(V_n^{\epsilon})}( \tau, \tau')  \ne 0. \]
By Theorem \ref{T:al-equal}, we can find a unique $\pi^{\heartsuit} \in \Pi^{\epsilon'}_{\phi^{\heartsuit}}$ (which determines $\epsilon'$) such that
\[ \tau   = \Theta_{\psi, V^{\epsilon}_{n+1}, W^{\epsilon'}_n}( \pi^{\heartsuit}  ). \]
Now the see-saw identity gives
\[   0 \ne \Hom_{\U(V_n^{\epsilon})}(  \Theta_{\psi, V^{\epsilon}_{n+1},  W^{\epsilon'}_n}( \pi^{\heartsuit}),  \tau')
 = \Hom_{\U(W_n^{\epsilon'})} (\Theta_{\psi, V_n^{\epsilon}, W_n^{\epsilon'}}(\tau') \otimes \omega_{\psi, \chi, W_n^{\epsilon'}}, \pi^{\heartsuit}). \]
In particular, 
 \[  \pi^{\diamondsuit} := \Theta_{\psi, V_n^{\epsilon}, W_n^{\epsilon'}}(\tau')^{\vee} \ne 0 \]
and by Theorem \ref{T:equal}, it has $L$-parameter $(\phi')^\vee = \phi^{\diamondsuit}$. Thus we see that for some $(\pi^{\diamondsuit}, \pi^{\heartsuit})  \in \Pi^{\epsilon'}_{\phi^{\diamondsuit}} \times \Pi^{\epsilon'}_{\phi^{\heartsuit}}$, we have
\[
 \Hom_{\U(W^{\epsilon'}_n)} ( (\pi^{\diamondsuit})^{\vee} \otimes \omega_{\psi,\chi, W_n^{\epsilon'}} ,  \pi^{\heartsuit}) \ne 0 
\]
as desired. This completes the proof of (FJ).
 
\subsection{Proof of (P1)}

Now we come to the proof of (P1). In particular, we consider the theta correspondence for $\U(V_n^\epsilon) \times \U(W_n^{\epsilon'})$ relative to the Weil representation $\omega_{\psi, \chi_V, \chi_W, V_n^{\epsilon}, W_n^{\epsilon'}}$. Given an $L$-parameter $\phi$ for $\U(W_n^{\pm})$, we would like to explicate the bijection
\[  \theta:  \Irr(S_{\phi}) \longleftrightarrow \Irr(S_{\theta(\phi)}) \]
furnished by Theorem \ref{T:equal}, with $\theta(\phi) = \phi \otimes \chi_V^{-1} \chi_W$.
Here, recall that
\[  S_{\phi}  = S_{\theta(\phi)}    = \prod_i  (\ZZ / 2 \ZZ) a_i. \]
Since we now have (B), (FJ) and (P2) at our disposal, we shall be able to determine $\theta$ using the see-saw diagram.

More precisely, we start with a tempered $L$-parameter $\phi$ and consider an irreducible tempered representation $\pi = \pi(\eta) \in \Pi_{\phi}^{\epsilon'}$.
One knows by Theorem \ref{T:equal} that $\Theta_{\psi, V^{\epsilon}_n, W_n^{\epsilon'}}(\pi) \in \Pi_{\theta(\phi)}^{\epsilon}$ is a nonzero irreducible tempered representation of $\U(V_n^\epsilon)$ for a unique $\epsilon$. 
By the analog of \cite[Lemma 12.5]{gs} for unitary groups, one can find an irreducible tempered representation $\sigma$ of $\U(V_{n-1}^{\epsilon})$ such that
\[  \Hom_{\U(V_{n-1}^{\epsilon})}(\Theta_{\psi, V^{\epsilon}_n, W_n^{\epsilon'}}(\pi), \sigma)  \ne 0. \]
By (B),  one has
\[   \theta(\eta)(a_i)  =  \epsilon(\tfrac{1}{2}, \phi_{\sigma}^{\vee} \otimes  \phi_i \otimes \chi_V^{-1} \chi_W,  \psi^E_{-2}),   \]
where $\phi_\sigma$ is the $L$-parameter of $\sigma$.

On the other hand, one has the see-saw diagram
\[
 \xymatrix{
  \U(W_n^{\epsilon'})  \times \U(W_n^{\epsilon'})  \ar@{-}[dr] \ar@{-}[d] & \U(V^{\epsilon}_n) \ar@{-}[d] \\
  \U(W_n^{\epsilon'}) \ar@{-}[ur] &  \U(V^{\epsilon}_{n-1}) \times \U(L_{(-1)^{n-1}})}.
\] 
We consider the theta correspondence for $\U(L_{(-1)^{n-1}}) \times \U(W_n^{\epsilon'})$ relative to the pair of characters $(\chi^{(-1)^{n-1}}, \chi_W)$, so that the theta correspondence for $\U(V^{\epsilon}_{n-1}) \times \U(W_n^{\epsilon'})$ is with respect to the pair $(\chi_V \chi^{(-1)^n}, \chi_W)$. We shall suppress these pairs of characters from the notation in the following.
By Theorem \ref{T:al-equal}, the representation
\[  \tau := \Theta_{\psi, V^{\epsilon}_{n-1}, W_n^{\epsilon'}}(\sigma) \ne 0 \]
is irreducible and tempered.
Moreover, $\tau$ has $L$-parameter
\begin{equation} \label{E:tau}
  \phi_{\tau}  =  (\phi_{\sigma} \otimes \chi_V \chi_W^{-1} \chi^{(-1)^n} )
  \oplus \chi_V \chi^{(-1)^n}. \end{equation}

It will now be convenient to consider the even and odd cases separately.

\subsection{Even case}

Assume first that $n$ is even. By the see-saw identity, one has
\[ 0 \ne \Hom_{\U(W_n^{\epsilon'})}( \Theta_{\psi, V^{\epsilon}_{n-1}, W_n^{\epsilon'}}(\sigma) \otimes \omega_{\psi , \chi,  W_n^{\epsilon'}}^{\vee}, \pi) = \Hom_{\U(W_n^{\epsilon'})} (\tau \otimes \pi^{\vee}, \omega_{\psi,\chi, W_n^{\epsilon'}}). \]
It follows by (FJ) that
\[ 
 \eta(a_i) \cdot \omega_{E/F}(-1)^{\dim \phi_i} = \eta_{\pi^{\vee}}(a_i)  =
 \epsilon(\tfrac{1}{2}, \phi_{\tau} \otimes \phi_i^{\vee} \otimes \chi^{-1}, \psi_2^E), \]
where the local root number appearing here is independent of the choice of the additive character of $E/F$ used since $\dim \phi_\tau = n$ is even. Hence, by \eqref{E:tau}, one has
\[  \eta(a_i)  = 
  \epsilon(\tfrac{1}{2}, \phi_{\sigma} \otimes \phi_i^{\vee} \otimes \otimes \chi_V \chi_W^{-1}, \psi_2^E)
  \cdot \epsilon(\tfrac{1}{2}, \phi_i^{\vee} \otimes \chi_V, \psi_2^E)
  \cdot  \omega_{E/F}(-1)^{\dim \phi_i}.  
\] 
Noting that $\phi_i$ is conjugate symplectic, we may compute:
\begin{align*}
  \theta(\eta)(a_i) / \eta(a_i) & = \epsilon(\tfrac{1}{2}, \phi_i^{\vee} \otimes \chi_V, \psi_2^E) \cdot \omega_{E/F}(-1)^{\dim \phi_i}   \\
  & =  \epsilon(\tfrac{1}{2}, \phi_i^{\vee} \otimes \chi_V, \psi_{-2}^E) \\
  & =  \epsilon(\tfrac{1}{2}, \phi_i \otimes \chi_V^{-1}, \psi_2^E) \\
\end{align*}
 as desired.

\subsection{Odd case}

Now suppose that $n$ is odd. By the see-saw identity, one has
 \[  \Hom_{\U(W_n^{\epsilon'})}(\tau \otimes \omega_{\psi,\chi, W_n^{\epsilon'}},  \pi) \ne  0, \]
so that
\[  \Hom_{\U(W_n^{\epsilon'})}(\tau^{\vee} \otimes \pi, \omega_{\psi, \chi, W_n^{\epsilon'}}) \ne 0. \]
By (FJ), one has
\begin{align*}
  \eta(a_i)  & = \epsilon(\tfrac{1}{2}, \phi_{\tau}^{\vee} \otimes \phi_i \otimes \chi^{-1},  \psi^E_2) \\
  & =  \epsilon(\tfrac{1}{2}, \phi_{\sigma}^{\vee} \otimes \phi_i \otimes \chi_V^{-1} \chi_W, \psi^E_2) \cdot \epsilon(\tfrac{1}{2}, \phi_i \otimes \chi_V^{-1}, \psi^E_2),
  \end{align*}
where the second equality follows from \eqref{E:tau}. On the other hand, we have seen that
  \[  
 \theta(\eta)(a_i)  =  \epsilon(\tfrac{1}{2}, \phi_{\sigma}^{\vee} \otimes \phi_i \otimes \chi_V^{-1} \chi_W, \psi^E_{-2}) =
  \epsilon(\tfrac{1}{2}, \phi_{\sigma}^{\vee} \otimes \phi_i \otimes \chi_V^{-1} \chi_W , \psi^E_2), \]
where the second equality follows because $\dim \phi_{\sigma}^{\vee}  = n-1$ is even. 
Hence, we conclude that
\[   \theta(\eta)(a_i) / \eta(a_i) = \epsilon(\tfrac{1}{2},  \phi_i \otimes \chi_V^{-1},  \psi^E_2) \]
as desired.

We have thus shown Conjecture (P1) for tempered $L$-parameters. For nontempered $L$-parameters, (P1) follows from the tempered case by Proposition \ref{P:P1nt}.

To summarise, we have shown the following proposition:

\begin{prop} \label{P:BP2-FJP1}
Assume that $(\emph{B})_k$ and $(\emph{P2})_k$ hold for all tempered $L$-parameters for all $k \le n$. Then $(\emph{FJ})_k$ and $(\emph{P1})_k$ also hold for all tempered $L$-parameters for all $k \le n$.
\end{prop}

\subsection{(B) $+$ (P1) $\Longrightarrow$ (FJ) $+$ (P2)}

Instead of assuming (B) and (P2) as we have done above, one may assume (B) and (P1). Using the same arguments as above, together with Theorems \ref{T:equal} and \ref{T:al-equal}, one can then deduce (FJ) and (P2). We state this formally as a proposition and leave the details of the proof to the reader.

\begin{prop}  \label{P:BP1-FJP2}
Assume that $(\emph{B})_k$ and $(\emph{P1})_k$ hold for all tempered $L$-parameters for all $k \le n$. Then $(\emph{FJ})_k$ and $(\emph{P2})_k$ also hold for all tempered $L$-parameters for all $k \le n$.
\end{prop}

\section{\textbf{Proof of (P2)}}

After the previous section, and in view of the results of Beuzart-Plessis \cite{bp1}, \cite{bp2}, \cite{bp3} (who proves (B)), it remains to prove (P2)$_n$. We shall prove (P2)$_n$ by using induction on $n$. 

\subsection{The base cases}

For (P2)$_0$, there is nothing to prove. By \cite{hks}, \cite{ggp2} and \cite{bp2}, we know that (B)$_1$ and (P1)$_1$ hold. Hence it follows by Proposition \ref{P:BP1-FJP2} that (P2)$_1$ holds.

For (P2)$_2$, the nontempered case follows from the tempered case by Proposition \ref{P:P2nt}. To show (P2)$_2$ for tempered $L$-parameters, it follows by Proposition \ref{P:BP1-FJP2} that it suffices to  show (P1)$_2$ for tempered $L$-parameters. Now (P1)$_2$ was shown in \cite[Theorem 11.2]{ggp2} by a global argument, appealing to the analog of (P1)$_2$ at archimedean places. However, we can also give a purely local proof here.

Suppose that $\phi$ is a tempered $L$-parameter for $\U(W_2^{\pm})$ and we are considering the theta correspondence for $\U(V_2^{\epsilon}) \times \U(W_2^{\epsilon'})$ with respect to a pair of characters $(\chi_V, \chi_W)$. If $\phi$ is irreducible, then Corollary \ref{C:stable} guarantees that (P1)$_2$ holds. Hence we shall assume that $\phi = \phi_1 \oplus \phi_2$ with $1$-dimensional characters $\phi_i$. If $\phi_1$ or $\phi_2$ is not conjugate symplectic, then $S_\phi$ is trivial and (P1)$_2$ follows from Theorem \ref{T:equal}. Thus, we shall further assume that both $\phi_1$ and $\phi_2$ are conjugate symplectic, so that
\[  S_{\phi}  = \begin{cases}
 (\ZZ/2\ZZ)a_1 \times (\ZZ/2\ZZ) a_2  &  \text{if $\phi_1 \ne \phi_2$;} \\ 
 ((\ZZ/2\ZZ) a_1 \times (\ZZ/2\ZZ) a_2 ) / \Delta \ZZ/2\ZZ  & \text{if $\phi_1  = \phi_2$.} \end{cases} \]
To unify notation in the two cases, we shall regard $\Irr(S_{\phi})$ as a subset of the irreducible characters of $(\ZZ/2\ZZ) a_1 \times (\ZZ /2\ZZ) a_2$ even when $\phi_1 = \phi_2$.

Let $\pi = \pi(\eta) \in \Pi_{\phi}^{\epsilon'}$. By Theorem \ref{T:equal}, we know that the theta lift of $\pi$ to $\U(V_2^{\epsilon})$ is nonzero for a uniquely determined $\epsilon$ given by
\[  \epsilon  = \epsilon(\tfrac{1}{2}, \phi \otimes \chi_V^{-1},  \psi^E_2) \cdot \epsilon',\]
and has $L$-parameter
\[  \theta(\phi)  =  \phi \otimes \chi_V^{-1} \chi_W. \]
Set
\[  \sigma = \Theta_{\psi, V_2^{\epsilon}, W_2^{\epsilon'}}(\pi) \in \Pi_{\theta(\phi)}^{\epsilon}  \]
and let $\theta(\eta) \in \Irr(S_{\theta(\phi)})$ be the irreducible character associated to $\sigma$. Then we need to compute $\theta(\eta)(a_i) / \eta(a_i)$. 
  
Consider the decomposition
\[  V_2^{\epsilon}  = V_1^{\epsilon}  \oplus L_{-1}, \]
and choose a character $\mu \in \Irr(\U(V_1^{\epsilon}))$ such that
\[   \Hom_{\U(V_1^{\epsilon})}(\sigma,  \mu) \ne 0. \]
Then by (B)$_1$, one sees that
\begin{equation}  \label{E:eta'}
  \theta(\eta)(a_i)  = \epsilon(\tfrac{1}{2}, \mu_E^{-1} \phi_i \chi_V^{-1} \chi_W, \psi^E_{-2})
= \epsilon(\tfrac{1}{2}, \mu_E^{-1} \phi_i \chi_V^{-1} \chi_W, \psi^E_2) \cdot \omega_{E/F}(-1),
\end{equation}
where $\mu_E$ is the character of $E^{\times}$ given by $\mu_E(x) = \mu(x/x^c)$.

On the other hand, consider the see-saw diagram
\[
 \xymatrix{
  \U(W_2^{\epsilon'})  \times \U(W_2^{\epsilon'})  \ar@{-}[dr] \ar@{-}[d] & \U(V^{\epsilon}_2) \ar@{-}[d] \\
  \U(W_2^{\epsilon'}) \ar@{-}[ur] &  \U(V^{\epsilon}_1) \times \U(L_{-1})}.
\]
For a conjugate symplectic character $\chi$ of $E^{\times}$, we consider the theta correspondences for
\[  \U(V_1^{\epsilon} )\times \U(W^{\epsilon'}_2)  \quad \text{with respect to} \quad (\chi_V \chi, \chi_W) \]
and
\[  \U(L_{-1}) \times \U(W_2^{\epsilon'}) \quad \text{with respect to} \quad (\chi^{-1},  \chi_W). \]
Set
\[  \tau :=  \Theta_{\psi, \chi_V \chi, \chi_W, V_1^{\epsilon}, W_2^{\epsilon'}}(\mu) \quad \text{on} \quad \U(W_2^{\epsilon'}). \]
Then Theorem \ref{T:al-equal} implies that $\tau$ has $L$-parameter
\[  \phi_{\tau}  = \mu_E \chi_V \chi_W^{-1}  \chi \oplus \chi_V \chi. \]
Now the see-saw identity then gives
\[  
 0 \ne \Hom_{\U(V_1^{\epsilon})}(\sigma,  \mu)  = \Hom_{\U(W_2^{\epsilon'})}( \tau \otimes \omega_{\psi, \chi, W_2^{\epsilon'}}^{\vee},  \pi). \]
Since we do not know (FJ)$_2$ at this point, this nonvanishing does not give us the desired information about $\eta$. However, we note that
\[   \Hom_{\U(W_2^{\epsilon'})}( \tau \otimes \omega_{\psi, \chi, W_2^{\epsilon'}}^{\vee},  \pi) 
=  \Hom_{\U(W_2^{\epsilon'})}(\pi^{\vee}  \otimes \omega_{\psi, \chi, W_2^{\epsilon'}}^{\vee},  \tau^{\vee}). \]
This allows one to exchange the roles of $\pi$ and $\tau$ in a variant of the above see-saw diagram.

More precisely, since $\phi  = \phi_1 \oplus \phi_2$ with conjugate symplectic characters $\phi_i$, it follows by (P2)$_1$ (which we have shown) that the $L$-packet $\Pi_{\phi^{\vee}}$ can be constructed via theta lifts from $\U(V_1^{\pm})$.
Namely, if we start with the $L$-parameter 
\[  \phi' := \phi_1^{-1} \phi_2 \chi_W \quad \text{for} \quad \U(V_1^{\pm}) \]
and consider the theta correspondence  for $\U(V_1^{\epsilon''}) \times \U(W_2^{\epsilon'})$ with respect to the pair $(\phi_2^{-1},  \chi_W)$, then the theta lifts of $\Pi_{\phi'}$ give the $L$-packet $\Pi_{\phi^{\vee}}$. In particular, we see that
\[   \pi^{\vee}  =   \Theta_{\psi, \phi_2^{-1}, \chi_W, V_1^{\epsilon''}, W_2^{\epsilon'}}(\mu')  \]
for a unique $\mu' \in \Pi_{\phi'}^{\epsilon''}$ (which determines $\epsilon''$).
Indeed, (P2)$_1$ says that 
\begin{equation}  \label{E:e''}
 \epsilon'' = \eta_{\pi^{\vee}}(a_1) = \eta(a_1)  \cdot \omega_{E/F}(-1).
\end{equation}

Thus, we may consider the see-saw diagram
\[
 \xymatrix{
  \U(W_2^{\epsilon'})  \times \U(W_2^{\epsilon'})  \ar@{-}[dr] \ar@{-}[d] & \U(V^{\epsilon''}_2) \ar@{-}[d] \\
\U(W_2^{\epsilon'}) \ar@{-}[ur] &  \U(V^{\epsilon''}_1) \times \U(L_{-1})},
\]
and the theta correspondences for
\[ \U(V_1^{\epsilon''}) \times \U(W_2^{\epsilon'}) \quad \text{with respect to} \quad (\phi_2^{-1}, \chi_W) \]
and
\[  \U(L_{-1}) \times \U(W_2^{\epsilon'}) \quad \text{with respect to} \quad (\chi^{-1}, \chi_W), \]
so that the theta correspondence for
\[  \U(V_2^{\epsilon''}) \times \U(W_2^{\epsilon'})\]
is  with respect to $(\phi_2^{-1} \chi^{-1},  \chi_W)$. The see-saw identity then reads:
\[
  0 \ne \Hom_{\U(W_2^{\epsilon'})}(  \pi^{\vee} \otimes  \omega_{\psi, \chi, W_2^{\epsilon'}}^{\vee},  \tau^{\vee})
   = \Hom_{\U(V_1^{\epsilon''})} (\Theta_{\psi, \phi_2^{-1} \chi^{-1},  \chi_W, V_2^{\epsilon''}, W_2^{\epsilon'}}(\tau^{\vee}),  \mu'). \]
In particular, $\Theta_{\psi, \phi_2^{-1} \chi^{-1},  \chi_W, V_2^{\epsilon''}, W_2^{\epsilon'}}(\tau^{\vee}) \ne 0$ on $\U(V_2^{\epsilon''})$.
By Theorem \ref{T:equal}(ii), one deduces that
\begin{align*}
  \epsilon''  \cdot  \epsilon'  & =  \epsilon(\tfrac{1}{2}, \phi_{\tau}^\vee \otimes \phi_2 \chi, \psi^E_2) \\
 & =  \epsilon(\tfrac{1}{2}, \mu_E^{-1} \phi_2 \chi_V^{-1} \chi_W, \psi^E_2)  \cdot \epsilon(\tfrac{1}{2}, \phi_2 \chi_V^{-1} , \psi^E_2).
\end{align*}
By \eqref{E:eta'} and \eqref{E:e''}, and noting that $\epsilon' = \eta(a_1) \cdot \eta(a_2)$, we see that
 \[  \eta (a_2)  =  \theta(\eta)(a_2)  \cdot \epsilon(\tfrac{1}{2},  \phi_2 \chi_V^{-1}, \psi^E_2)  \]
as desired. It then follows by Theorem \ref{T:equal}(ii) that 
\[  \eta (a_1)  =  \theta(\eta)(a_1)  \cdot \epsilon(\tfrac{1}{2}, \phi_1 \chi_V^{-1} , \psi^E_2)  \] 
as well. 

Thus, we have demonstrated (P1)$_2$, and hence (P2)$_2$. 

\subsection{Inductive step}

Now we assume that $n \ge 3$ and (P2)$_k$ holds for all $k < n$. Proposition \ref{P:P2nt} implies that (P2)$_n$ holds for all nontempered $L$-parameters. We are thus reduced to the case of tempered $L$-parameters. Then we have the following theorem whose proof will be given in the next two sections:

\begin{thm}
\label{T:ind}
If $(\emph{P2})_k$ holds for all tempered $L$-parameters for all $k < n$, then $(\emph{P2})_n$ holds for all tempered but non-square-integrable $L$-parameters.
\end{thm}

The proof of this theorem is an elaborate extension of the techniques developed in the PhD thesis of the second author \cite{i}. Assuming this theorem for the moment, we are thus reduced to the case of square-integrable $L$-parameters. 

\subsection{Square-integrable case}

We now consider (P2)$_n$ for a square-integrable $L$-parameter  
\[  \phi  = \phi_1 \oplus  \cdots \oplus \phi_r \]
for $\U(W_n^\pm)$.
Thus $\phi$ is multiplicity-free and each $\phi_i$ is an $n_i$-dimensional irreducible conjugate self-dual representation of $\WD_E$ with sign $(-1)^{n-1}$.
Recall that the component group $S_{\phi}$ is of the form
\[   S_{\phi}  = \prod_{i=1}^r  (\ZZ / 2\ZZ) a_i. \]

We shall first assume that $r > 1$. Then either $r \ge 3$ or else $r =2$ in which case we may assume that $n_1 = \dim \phi_1 \ge 2$. 

Let $\pi = \pi(\eta) \in \Pi_{\phi}^{\epsilon'}$ be an irreducible square-integrable representation of $\U(W_n^{\epsilon'})$ with associated character $\eta \in \Irr(S_{\phi})$. We consider the theta correspondence for $\U(V_{n+1}^\epsilon) \times \U(W_n^{\epsilon'})$ with respect to the data $(\psi, \chi_V, \chi_W)$, and suppose that $\pi' := \Theta_{\psi, V_{n+1}^{\epsilon}, W_n^{\epsilon'}}(\pi) \ne 0$. Then by Theorem \ref{T:al-equal}, $\pi'  = \pi'(\eta')  \in \Pi_{\theta(\phi)}^\epsilon$ is an irreducible tempered representation of $\U(V_{n+1}^\epsilon)$ with associated character $\eta' \in \Irr(S_{\theta(\phi)})$. We want to determine $\eta'$ in terms of $\eta$. Indeed, recall that there is a natural embedding
\[  S_{\phi}  \hookrightarrow S_{\theta(\phi)}  \]
and we need to show that $\eta'(a_i)  = \eta(a_i)$. We shall do so by a global argument.

\subsection{Globalization}

Let us begin the process of globalization which is the most delicate part of the argument. Choose a number field $\FF$ and a quadratic field extension $\EE$ of $\FF$ such that 
\begin{itemize}
\item $\FF$ is totally complex;
\item $\EE_{v_0} / \FF_{v_0}  = E/F$ for a finite place $v_0$ of $\FF$;
\item there is a fixed finite place $w$ of $\FF$ which is split in $\EE$. 
\end{itemize}
Fix:
\begin{itemize}
\item a nontrivial additive character $\Psi$ of $\AA / F$ such that $\Psi_{v_0}  = \psi$ (in its $\N_{E/F}(E^{\times})$-orbit);
\item a conjugate symplectic Hecke character $\chi$ of $\AA_{\EE}^{\times}$; 
\item a trace zero element $\delta \in \EE^{\times}$ so that the signs of the skew-Hermitian spaces $W_n^{\pm}$ at the place $v_0$ are defined using $\delta$.
\end{itemize}
Let $S$ be a sufficiently large finite set of inert finite places of $\FF$, not containing $v_0$, such that for all $v \notin S \cup \{ v_0 \}$, either $v$ is split in $\EE$ or else $\EE_v/\FF_v$, $\Psi_v$ and $\chi_v$ are all unramified. Moreover, $S$ can be made arbitrarily large.

We shall globalize the $L$-parameter $\phi$ to a tempered $A$-parameter $\Sigma$.

\begin{enumerate}

\item 
At $v_0$, consider the given irreducible representation $\phi_i$ of $\WD_E$. Since $\phi_i$ is conjugate self-dual with sign $(-1)^{n-1}$, it may not be an $L$-parameter for $\U(W_{n_i}^{\pm})$. Instead,  the representation 
\[  \phi_{i,v_0}' := \phi_i \otimes \chi_{v_0}^{n_i-n} \]
is conjugate self-dual with sign $(-1)^{n_i-1}$, and thus defines an $L$-parameter for $\U(W_{n_i}^{\pm})$.

\item
At $v \in S$, choose a representation $\phi_{i,v}$ of $\WD_E$ which is the multiplicity-free sum of $1$-dimensional conjugate self-dual characters with sign $(-1)^{n-1}$. As above, $\phi_{i,v}$ is conjugate self-dual with sign $(-1)^{n-1}$ and thus may not be an $L$-parameter for $\U(W_{n_i,v}^{\pm})$, where $W_{n_i,v}^{\pm}$ are the $n_i$-dimensional skew-Hermitian spaces over $E_v$. We set 
\[  \phi_{i,v}' := \phi_{i,v} \otimes \chi_v^{n_i-n}, \]
so that $\phi_{i,v}'$ is an $L$-parameter for $\U(W_{n_i,v}^\pm)$. The local component group $S_{\phi'_{i,v}}$ of $\phi'_{i,v}$ is of the form
\[  S_{\phi'_{i,v}}  =  (\ZZ / 2\ZZ)^{n_i} \]
and the Vogan $L$-packet $\Pi_{\phi'_{i,v}}$ consists of $2^{n_i}$ irreducible square-integrable representations of $\U(W_{n_i,v}^{\pm})$.

\item  We require in addition that, for all $v \in S$, 
\[  \phi_v := \phi_{1,v} \oplus \cdots \oplus \phi_{r,v} \] 
is not multiplicity-free, i.e.~$\phi_v$ is not a square-integrable $L$-parameter for $\U(W_{n,v}^{\pm})$. To achieve this, we pick a character $\mu_v$ contained in $\phi_{1,v}$ and then ensure that $\mu_v$ is also contained in $\phi_{i_v,v}$ for some $i_v \ge 2$. It is here that we use the assumption that $r > 1$. Moreover, we may ensure that
\[
 i_v \ne i_{v'} 
\]
for some distinct $v, v' \in S$ if $r > 2$.

\item For each $v \in S$, there is a natural map
\[  (\ZZ / 2 \ZZ)^r  = \prod_{i=1}^r (\ZZ / 2\ZZ) a_i \longrightarrow S_{\phi_v}  \]
which sends $a_i$ to the image of the element $-1_{\phi_{i,v}}$ in $S_{\phi_v}$. In view of (iii), for $\# S$ large enough (indeed, for $\# S \geq 2$), the induced diagonal map
\[  (\ZZ / 2\ZZ)^r  \longrightarrow  \prod_{v \in S} S_{\phi_v}  \]
is injective.  

\item Now for each $i=1, \ldots, r$, we have a collection of square-integrable $L$-parameters $\phi'_{i,v}$ for $v \in S \cup \{ v_0 \}$.
For each $v \in S \cup \{ v_0 \}$, pick an irreducible square-integrable representation $\pi_v \in \Pi_{\phi'_{i,v}}^+$.
Let $\WW_{n_i}^+$ be the $n_i$-dimensional skew-Hermitian space over $\EE$ whose localization at each inert $v$ is $W_{n_i,v}^+$, where we have used the trace zero element $\delta \in \EE^{\times}$ to define the sign of a skew-Hermitian space over $E_v$.
Then by a result of Shin \cite[Theorem 5.13]{sh} (proved using the trace formula), one can find an irreducible cuspidal automorphic representation $\Pi'_i$ of $\U(\WW_{n_i}^+)(\AA)$ such that
\begin{itemize}
\item $\Pi'_{i,v}  = \pi_v$ for all $v \in S\cup \{ v_0 \}$;
\item $\Pi'_{i,v}$ is unramified for all inert $v \notin S \cup \{ v_0 \}$;
\item $\Pi'_{i,w}$ is an irreducible supercuspidal representation of $\U(\WW_{n_i,w}^+) \cong \GL_{n_i}(\FF_w)$. 
\end{itemize}

\item By results of Mok \cite{mok}, the representation $\Pi'_i$ has tempered $A$-parameter $\Sigma'_i$, which is an irreducible conjugate self-dual cuspidal automorphic representation of $\GL_{n_i}(\AA_{\EE})$ with sign $(-1)^{n_i-1}$.
The cuspidality of $\Sigma'_i$ is a consequence of the fact that $\Pi'_{i,w}$ is supercuspidal at the split place $w$. If we set
\[  \Sigma_i  = \Sigma_i' \otimes \chi^{n-n_i}, \]
then $\Sigma_i$ is an irreducible conjugate self-dual cuspidal automorphic representation of $\GL_{n_i}(\AA_{\EE})$ with sign $(-1)^{n-1}$.
In particular, setting
\[  \Sigma  = \BIGboxplus_{i=1}^r  \Sigma_i, \]
we see that $\Sigma$ is a tempered $A$-parameter for $\U(\WW_n)$, where $\WW_n$ is an $n$-dimensional skew-Hermitian space over $\EE$. 
\end{enumerate}
 
\subsection{Properties of $\Sigma$}
\label{SS:Sigma}

We have completed the construction of a global tempered $A$-parameter $\Sigma$. Let us examine some crucial properties of $\Sigma$. 

\begin{itemize}

\item (Local components)
It follows by construction that the local components of the $A$-parameter $\Sigma$ are given as follows:
\begin{itemize}
\item at the place $v_0$, $\Sigma_{v_0}$ has $L$-parameter $\phi$;
\item at all places $v \in S$, $\Sigma_v$ has $L$-parameter $\phi_v$;
\item at all inert places $v \notin S \cup \{ v_0 \}$, $\Sigma_v$ is unramified.
\end{itemize}
In particular, we have found a globalization $\Sigma$ of the given local $L$-parameter $\phi$ so that at all inert places $v \ne v_0$ of $\FF$, $\Sigma_v$ defines a non-square-integrable $L$-parameter for $\U(W_{n,v}^{\pm})$.

\item (Whittaker data)
We shall use the additive character $\Psi = \otimes_v \Psi_v$ to fix the Whittaker datum at each place $v$. Together with the fixed trace zero element $\delta \in \EE^{\times}$, we have thus fixed the local Langlands correspondence for $\U(W_{n,v}^{\pm})$ for each $v$. 

\item (Component groups)
The global component group $S_{\Sigma}$ of the $A$-parameter $\Sigma$ admits a natural map $S_{\Sigma} \rightarrow S_{\Sigma_v}$ for each place $v$. For $v = v_0$, this natural map is an isomorphism, so that we have a canonical identification:
\[  S_{\Sigma}  =  S_{\Sigma_{v_0}} =  \prod_{i=1}^r (\ZZ /2 \ZZ) a_i. \]
On the other hand, in view of (iv) above, we see that the diagonal map
\[   S_{\Sigma}  \longrightarrow \prod_{v \ne v_0}  S_{\Sigma_v} \]
is injective.
Thus, given any $\eta \in \Irr(S_{\phi}) = \Irr(S_{\Sigma_{v_0}})$, one can find $\eta_v \in \Irr(S_{\Sigma_v})$ for $v \ne v_0$ so that
\[  \Bigl( \eta \otimes \Bigl( \bigotimes_{v \ne v_0} \eta_v \Bigr) \Bigr) \circ \Delta =  \1_{S_{\Sigma}}, \]
where 
\[  \Delta:  S_{\Sigma} \longrightarrow \prod_v S_{\Sigma_v} \]
is the diagonal map.

\item (Arthur's multiplicity formula) 
Consider the global $A$-packet associated to $\Sigma$.
For any collection $\eta_v \in \Irr(S_{\Sigma_v})$ of irreducible characters with associated representations $\pi(\eta_v)$ of local unitary groups $\U(W_{n,v}^{\epsilon_v'})$, consider the representation
\[  \Pi : = \bigotimes_v \pi(\eta_v) \]
of the adelic unitary group $\prod'_v \U(W_{n,v}^{\epsilon_v'})$.
Arthur's multiplicity formula \cite[Theorem 1.7.1]{kmsw} then states that the following are equivalent:
\begin{itemize}
\item the adelic unitary group $\prod'_v \U(W_{n,v}^{\epsilon_v'})$ is equal to $\U(\WW_n)(\AA)$ for a skew-Hermitian space $\WW_n$ over $\EE$ and $\Pi$ occurs in the automorphic discrete spectrum of $\U(\WW_n)(\AA)$;
\item the character $(\otimes_v \eta_v) \circ \Delta$ of $S_{\Sigma}$ is trivial.
\end{itemize}
\end{itemize}

By the above discussion combined with a result of Wallach \cite{wallach}, \cite[Proposition 4.10]{clozel}, we may find an $n$-dimensional skew-Hermitian space $\WW_n$ over $\EE$ and an irreducible cuspidal automorphic representation $\Pi$ of $\U(\WW_n)(\AA)$ in the global $A$-packet associated to $\Sigma$ such that $\Pi_{v_0}  = \pi(\eta)$. For each $v$, we shall write the local component $\Pi_v$ as $\pi(\eta_v)$.

\subsection{Global theta correspondence}

Now we shall construct a Hermitian space $\VV_{n+1}$ of dimension $n+1$ over $\EE$, and consider the global theta correspondence for $\U(\VV_{n+1}) \times \U(\WW_n)$. To define such a global theta correspondence, we shall use the fixed additive character $\Psi$ of $\AA / \FF$, and we also need to fix a pair of Hecke characters $\chi_{\VV}$ and $\chi_{\WW}$ of $\AA_{\EE}^{\times}$ such that
\[  \chi_{\VV}|_{\AA^{\times}}  = \omega_{\EE/ \FF}^{n+1} \quad \text{and} \quad \chi_{\WW}|_{\AA^{\times}}  = \omega_{\EE/ \FF}^n, \]
where $\omega_{\EE/ \FF}$ is the quadratic Hecke character of $\AA^{\times}$ associated to $\EE/\FF$ by global class field theory. We pick these so that, in addition:
\begin{enumerate}[label=(\alph*)]
\item at the place $v_0$, we have
\[  \chi_{\VV,v_0}  = \chi_V \quad \text{and} \quad  \chi_{\WW,v_0} = \chi_W; \]
\item at some place $v_1 \in S$, $\chi_{\VV,v_1}$ is not contained in the $L$-parameter associated to $\Sigma_{v_1}$. 
\end{enumerate}
Indeed, since $\EE^{\times} / \FF^{\times} \cong \Ker(\N_{\EE/\FF})$ is anisotropic,
for given conjugate orthogonal characters $\mu_i$ of $\EE_{v_i}^{\times}$,
there is a conjugate orthogonal Hecke character $\mu$ of $\AA_{\EE}^{\times}$ such that $\mu_{v_i} = \mu_i$ for $i=0,1$.
Thus, we can achieve (a) and (b) by replacing $\chi_{\VV}$ and $\chi_{\WW}$ by their twists by conjugate orthogonal Hecke characters of $\AA_\EE^{\times}$ if necessary.
The condition (b) guarantees that at the place $v_1$, the representation $\Pi_{v_1}$ has nonzero local theta lift to both $\U(V_{n+1,v_1}^+)$ and $\U(V_{n+1, v_1}^-)$ by Theorem \ref{T:al-equal}(i)(a).
Moreover, the conservation relation (proved by Sun--Zhu \cite{sz}) implies that the theta lifts of $\Pi_{v_1}$ to $\U(V_{n-1,v_1}^+)$ and $\U(V_{n-1, v_1}^-)$ are both zero. 

Now we note:

\begin{lem}
There is a Hermitian space $\VV_{n+1}$ of dimension $n+1$ over $\EE$ such that:
\begin{itemize}
\item at the place $v_0$, $\VV_{n+1,v_0}$ is equal to the given Hermitian space $V_{n+1}^\epsilon$;
\item for all places $v$, the representation $\Pi_v$ has nonzero local theta lift to $\U(\VV_{n+1,v})$ with respect to the theta lift defined by the data $(\Psi_v, \chi_{\VV_v}, \chi_{\WW_v})$.
\end{itemize}
\end{lem}

\begin{proof}
For all $v \ne v_0, v_1$, we may pick $\VV_{n+1, v}$ so that the local theta lift of $\Pi_v$ to $\U(\VV_{n+1, v})$ is nonzero, and then complete these to a coherent collection of Hermitian spaces by picking $V_{n+1}^\epsilon$ at $v_0$ and the uniquely determined Hermitian space at $v_1$. 
\end{proof}

\subsection{Completion of the proof}

Consider the global theta lift $\Pi' := \Theta_{\Psi, \VV_{n+1}, \WW_n}(\Pi)$ to $\U(\VV_{n+1})(\AA)$. The condition (b) above ensures that $\Pi'$ is cuspidal.
To show that $\Pi'$ is nonzero,
we consider the standard $L$-function $L(s, \Pi)$ of $\Pi$ defined using the doubling zeta integral of Piatetski-Shapiro--Rallis \cite{psr}, \cite{lr}.
Observe that the partial $L$-function $L^{S \cup \{v_0\}}(s, \Pi)$ agrees with the partial standard $L$-function $L^{S \cup \{v_0\}}(s, \Sigma)$ of $\Sigma$, so that 
\[  L^{S \cup \{v_0\}}(1, \Pi) = L^{S \cup \{v_0\}}(1, \Sigma) =  \prod_{i=1}^r  L^{S \cup \{v_0\}}(1, \Sigma_i) \ne 0 \]
since $\Sigma_i$ is unitary and cuspidal.
By \cite[Proposition 5]{lr}, the local standard $L$-factor $L(s, \Pi_v)$ at $v \in S \cup \{ v_0\}$ is holomorphic and nonzero at $s=1$ since $\Pi_v$ is tempered.
Hence
\[
 L(1, \Pi) \ne 0
\]
and it follows by \cite[Theorem 1.4]{gqt} that $\Pi'$ is nonzero.
Thus $\Pi'$ is an irreducible cuspidal automorphic representation of $\U(\VV_{n+1})(\AA)$ such that $\Pi'_{v_0}   = \pi'(\eta')$.

Recall that we have fixed the local Langlands correspondence for $\U(\WW_{n,v})$ for each $v$ using the Whittaker datum determined by the additive character $\Psi_v$ together with the trace zero element $\delta$.
To fix the local Langlands correspondence for $\U(\VV_{n+1,v})$ for each $v$, we shall use the Whittaker datum determined by the additive character $\Psi^{\EE_v}_v = \Psi_v(\frac{1}{2} \Tr_{\EE_v/\FF_v}(\delta \, \cdot \,))$.
Then we may write 
\[
 \Pi  = \bigotimes_v \pi(\eta_v) \quad \text{and} \quad
 \Pi' = \bigotimes_v \pi'(\eta'_v)
\]
with associated irreducible characters $\eta_v$ and $\eta'_v$ of the local component groups.

Recall that $\Pi$ has tempered $A$-parameter.
By Theorem \ref{T:al-equal}, $\Pi'$ also has tempered $A$-parameter.
Hence, applying Arthur's multiplicity formula \cite[Theorem 1.7.1]{kmsw} to $\Pi$ and $\Pi'$, we see that
\begin{equation}
\label{eq:arthur}
 \prod_v \eta_v(a_{i,v}) = 1 \quad \text{and} \quad \prod_v \eta_v'(a_{i,v}) = 1
\end{equation}
for all $i$, where $a_{i,v}$ is the image of $a_i$ in $S_{\Sigma_v}$.
However, for all places $v \ne v_0$, either $v$ is split, or else the $L$-parameter of $\Pi_v$ is not square-integrable. Thus, for all inert $v \ne v_0$, one knows that (P2)$_n$ holds. In particular,
\[  \eta'_v(a_{i,v})  =  \eta_v(a_{i,v})  \]
for all $v \ne v_0$.
Thus, we conclude that at the place $v_0$, we have
\[  \eta' (a_i)  = \eta(a_i)  \]
as desired. 

We have thus completed the proof of (P2)$_n$ when $r >1$, i.e.~when $\phi$ is reducible. To deal with the case when $\phi$ is irreducible, with $r = 1$, we can again appeal to a variation of the global argument as above.
Namely, in the globalization step above, we may now take the $L$-parameter $\phi_v$ for $v \in S$ to be square-integrable $L$-parameters which are \emph{reducible}. Then the rest of the argument is the same, using the fact that we have shown (P2)$_n$ for every place $v \ne v_0$. This completes the proof of (P2)$_n$.

\section{\textbf{Preparations for the proof of Theorem \ref{T:ind}}}

To finish the proof of (P2), it now remains to prove Theorem \ref{T:ind}.
For this, we need to introduce more notation.
Fix $\varepsilon = \pm 1$. 
In this and next sections, we shall let $V$ and $W$ be an $\varepsilon$-Hermitian space and a $(-\varepsilon)$-Hermitian space respectively.
Put
\[
 m = \dim V \quad \text{and} \quad n = \dim W. 
\]

\subsection{Parabolic subgroups}

Let $r$ be the Witt index of $V$ and $V_\an$ an anisotropic kernel of $V$.
Choose a basis $\{ v_i, v_i^* \, | \, i = 1, \ldots, r \}$ of the orthogonal complement of $V_\an$ such that 
\[
 \langle v_i, v_j \rangle_V = \langle v^*_i, v^*_j \rangle_V = 0, \quad 
 \langle v_i, v^*_j \rangle _V = \delta_{i,j}
\]
for $1 \le i, j \le r$.
Let $k$ be a positive integer with $k \le r$ and set 
\[
 X = E v_1 \oplus \cdots \oplus E v_k, \quad X^* = E v^*_1 \oplus \cdots \oplus E v^*_k.
\]
Let $V_0$ be the orthogonal complement of $X \oplus X^*$ in $V$, so that $V_0$ is an $\varepsilon$-Hermitian space of dimension $m_0 = m -2k$ over $E$.
We shall write an element in the unitary group $\U(V)$ as a block matrix relative to the decomposition $V = X \oplus V_0 \oplus X^*$.
Let $P = M_P U_P$ be the maximal parabolic subgroup of $\U(V)$ stabilizing $X$,
where $M_P$ is the Levi component of $P$ stabilizing $X^*$ and $U_P$ is the unipotent radical of $P$. We have
\begin{align*}
 M_P & = \{ m_P(a) \cdot h_0 \, | \, a \in \GL(X), \, h_0 \in \U(V_0) \}, \\
 U_P & = \{ u_P(b) \cdot u_P(c) \, | \, b \in \Hom(V_0, X), \,  c \in \Herm(X^*,X) \}, 
\end{align*}
where 
\[
 m_P(a) = 
 \begin{pmatrix}
  a & & \\
  & 1_{V_0} & \\
  & & (a^*)^{-1}
 \end{pmatrix}, \quad
 u_P(b) = 
 \begin{pmatrix}
  1_X & b & - \tfrac{1}{2} b b^* \\
  & 1_{V_0} & -b^* \\
  & & 1_{X^*}
 \end{pmatrix}, \quad
 u_P(c) = 
 \begin{pmatrix}
  1_X & & c \\
  & 1_{V_0} & \\
  & & 1_{X^*}
 \end{pmatrix},
\]
and 
\[
 \Herm(X^*, X) = \{ c \in \Hom(X^*,X) \, | \, c^* = -c \}.
\]
Here, the elements $a^* \in \GL(X^*)$, $b^* \in \Hom(X^*, V_0)$, and $c^* \in \Hom(X^*, X)$ are defined by requiring that
\[
 \langle a x, x' \rangle_V = \langle x, a^* x' \rangle_V, \quad
 \langle b v, x' \rangle_V = \langle v, b^* x' \rangle_V, \quad
 \langle c x', x'' \rangle_V  = \langle x', c^* x'' \rangle_V
\]
for $x \in X$, $x', x'' \in X^*$, and $v \in V_0$.
In particular, $M_P \cong \GL(X) \times \U(V_0)$ and 
\[
 1 \longrightarrow \Herm(X^*,X) \longrightarrow U_P \longrightarrow 
 \Hom(V_0, X) \longrightarrow 1.
\]
Put
\[
 \rho_P = \frac{m_0+k}{2}, \quad 
 w_P =
 \begin{pmatrix}
  & & -I_X \\
  & 1_{V_0} & \\
  -\varepsilon I_X^{-1} & & 
 \end{pmatrix},
\]
where $I_X \in \Isom(X^*, X)$ is defined by $I_X v_i^* = v_i$ for $1 \le i \le k$.

Similarly, let $r'$ be the Witt index of $W$ and choose a basis $\{ w_i, w_i^* \, | \, i = 1, \ldots, r' \}$ of the orthogonal complement of an anisotropic kernel of $W$ such that 
\[
 \langle w_i, w_j \rangle_W = \langle w^*_i, w^*_j \rangle_W = 0, \quad \langle w_i, w^*_j \rangle_W = \delta_{i,j}
\]
for $1 \le i, j \le r'$.
We assume that $k \le r'$ and set 
\[
 Y = E w_1  \oplus \cdots \oplus E w_k, \quad Y^* = E w^*_1 \oplus \cdots \oplus E w^*_k.
\]
Let $W_0$ be the orthogonal complement of $Y \oplus Y^*$ in $W$, so that $W_0$ is a $(- \varepsilon)$-Hermitian space of dimension $n_0 = n -2k$ over $E$.
Let $Q = M_Q U_Q$ be the maximal parabolic subgroup of $\U(W)$ stabilizing $Y$, where $M_Q$ is the Levi component of $Q$ stabilizing $Y^*$ and $U_Q$ is the unipotent radical of $Q$.
Then $M_Q \cong \GL(Y) \times \U(W_0)$ and 
\[
 1 \longrightarrow \Herm(Y^*,Y) \longrightarrow U_Q \longrightarrow 
 \Hom(W_0, Y) \longrightarrow 1.
\]
For $a \in \GL(Y)$, $b \in \Hom(W_0, Y)$, and $c \in \Herm(Y^*,Y)$,
we define elements $m_Q(a) \in M_Q$ and $u_Q(b), u_Q(c) \in U_Q$ as above.
Put
\[
 \rho_Q = \frac{n_0+k}{2}, \quad
 w_Q =
 \begin{pmatrix}
  & & -I_Y \\
  & 1_{W_0} & \\
  \varepsilon I_Y^{-1} & & 
 \end{pmatrix},
\]
where $I_Y \in \Isom(Y^*, Y)$ is defined by $I_Y w_i^* = w_i$ for $1 \le i \le k$.

\subsection{Haar measures}

We need to choose Haar measures on various groups.
In particular, we shall define Haar measures on $U_P$ and $U_Q$ in the following.

Recall the symplectic form
$\langle \cdot, \cdot \rangle = \Tr_{E/F}(\langle \cdot, \cdot \rangle_V \otimes \langle \cdot, \cdot \rangle_W)$ on $V \otimes W$ over $F$.
We consider the following spaces and pairings:
\begin{itemize}
 \item $(x, y) \mapsto \psi(\langle x, I_Y^{-1} y \rangle)$
 for $x, y \in V \otimes Y$;
 \item $(x, y) \mapsto \psi(\langle x, I_Y y \rangle)$
 for $x, y \in V_0 \otimes Y^*$;
 \item $(x, y) \mapsto \psi(\langle I_X^{-1} x, y \rangle)$
 for $x, y \in X \otimes W_0$;
 \item $(x, y) \mapsto \psi(\langle I_X x, y \rangle)$
 for $x, y \in X^* \otimes W_0$; 
 \item $(x, y) \mapsto \psi(\langle I_X^{-1} x, I_Y y \rangle)$
 for $x, y \in X \otimes Y^*$; 
 \item $(x, y) \mapsto \psi(\langle I_X x, I_Y^{-1} y \rangle)$
 for $x, y \in X^* \otimes Y$; 
 \item $(x, y) \mapsto \psi(\langle I_X x, I_Y y \rangle)$
 for $x, y \in X^* \otimes Y^*$. 
\end{itemize}
On these spaces, we take the self-dual Haar measures with respect to these pairings.
Put
\[
  e^{**} = v^*_1 \otimes w^*_1 + \cdots + v_k^* \otimes w^*_k \in X^* \otimes Y^*.
\]
\begin{itemize}
 \item We transfer the Haar measure on $V_0 \otimes Y^*$ to $\Hom(X^*, V_0)$ via the isomorphism $x \mapsto x e^{**}$ for $x \in \Hom(X^*, V_0)$.
 \item We transfer the Haar measure on $\Hom(X^*,V_0)$ to $\Hom(V_0, X)$ via the isomorphism $x \mapsto x^*$ for $x \in \Hom(V_0, X)$.
 \item Similarly, we define the Haar measure on $\Hom(W_0, Y)$.
\end{itemize}
Furthermore:
\begin{itemize}
 \item We transfer the Haar measure on $X \otimes Y^*$ to $\Hom(X^*, X)$ via the isomorphism $x \mapsto x e^{**}$ for $x \in \Hom(X^*, X)$. This Haar measure on $\Hom(X^*, X)$ is self-dual with respect to the pairing
$(x, y) \mapsto \psi(\langle I_X^{-1} x e^{**}, I_Y y e^{**} \rangle)$.
 \item We take the Haar measure $|2|_F^{-k^2/2} \, dx$ on $\Herm(X^*,X)$, where $dx$ is the self-dual Haar measure on $\Herm(X^*,X)$ with respect to the pairing
$(x, y) \mapsto \psi(\langle I_X^{-1} x e^{**}, I_Y y e^{**} \rangle)$.
 \item Similarly, we define the Haar measure on $\Herm(Y^*, Y)$.
\end{itemize}
Then: 
\begin{itemize}
 \item We take the Haar measure $du = db \, dc$ on $U_P$ for $u = u_P(b) u_P(c)$ with $b \in \Hom(V_0,X)$ and $c \in \Herm(X^*,X)$.
 \item Similarly, we define the Haar measure on $U_Q$.
\end{itemize}

We note the following Fourier inversion formula:

\begin{lem}
\label{L:fi}
For $\varphi \in \s(X \otimes Y^*)$, we have
\[
 \int_{\Herm(Y^*,Y)} \left( \int_{\Hom(X^*,X)}
 \varphi(x e^{**}) \psi( \langle x e^{**}, c e^{**} \rangle) \, d x \right) d c
 = \int_{\Herm(X^*,X)} \varphi(c e^{**}) \, d c.
\]
\end{lem}

\begin{proof}
We consider the nondegenerate symmetric bilinear form $(x, y) \mapsto \langle I_X^{-1} x, I_Y y \rangle$ on $X \otimes Y^*$ over $F$, and the subspaces 
\[
 \Herm(X^*, X) e^{**} \quad \text{and} \quad I_X I_Y^{-1} \Herm(Y^*, Y) e^{**}
\]
of $X \otimes Y^* = \Hom(X^*, X) e^{**}$.
For $x \in \Hom(X^*, X)$ and $y \in \Herm(Y^*, Y)$, we have
\[
 \langle I_X^{-1} x e^{**}, I_Y I_X I_Y^{-1} y e^{**} \rangle
 = \langle I_X^{-1} x e^{**}, I_X y e^{**} \rangle
 = \langle I_X^* I_X^{-1} x e^{**}, y e^{**} \rangle
 = \varepsilon \cdot \langle x e^{**}, y e^{**} \rangle
\]
since $I_X^* = \varepsilon I_X$.
For $x \in \Herm(X^*, X)$ and $y \in \Herm(Y^*, Y)$, 
noting that $x^* = -x$, $y^* = -y$, and $x$ commutes with $y$, we have
\[
 \langle x e^{**}, y e^{**} \rangle
 = \langle y^* e^{**}, x^* e^{**} \rangle
 = \langle y e^{**}, x e^{**} \rangle
 = - \langle x e^{**}, y e^{**} \rangle, 
\]
so that
\[
\langle x e^{**}, y e^{**} \rangle = 0. 
\]
Since $\Hom(X^*, X) e^{**}$ is nondegenerate with respect to the above bilinear form,
we see that $X \otimes Y^*$ decomposes as the orthogonal direct sum
\[
 X \otimes Y^* = \Herm(X^*, X) e^{**} \oplus I_X I_Y^{-1} \Herm(Y^*, Y) e^{**}.
\]
These yield the desired Fourier inversion formula.
\end{proof}

\subsection{Normalized intertwining operators}
\label{SS:normalization}

In this subsection, we define the normalized intertwining operator which is used to describe the local Langlands correspondence.

Let $\tau$ be an irreducible (unitary) square-integrable representation of $\GL(X)$ on a space $\V_{\tau}$ with central character $\omega_{\tau}$.
For any $s \in \CC$, we realize the representation $\tau_s := \tau \otimes |\det|^s$ on $\V_{\tau}$ by setting $\tau_s(a) v := |\det a|^s \tau(a) v$ for $a \in \GL(X)$ and $v \in \V_{\tau}$.
Let $\sigma_0$ be an irreducible tempered representation of $\U(V_0)$ on a space $\V_{\sigma_0}$.
We consider the induced representation
\[
 \Ind^{\U(V)}_P(\tau_s \otimes \sigma_0)
\]
of $\U(V)$, which is realized on the space of smooth functions $\Phi_s : \U(V) \rightarrow \V_\tau \otimes \V_{\sigma_0}$ such that 
\[
 \Phi_s(u m_P(a) h_0 h) = |\det a|^{s+\rho_P} \tau(a) \sigma_0(h_0) \Phi_s(h)
\]
for all $u \in U_P$, $a \in \GL(X)$, $h_0 \in \U(V_0)$, and $h \in \U(V)$.
Let $A_P$ be the split component of the center of $M_P$ and $W(M_P) = \Norm_{\U(V)}(A_P)/M_P$ the relative Weyl group for $M_P$. Noting that $W(M_P) \cong \ZZ/2 \ZZ$, we denote by $w$ the nontrivial element in $W(M_P)$. For any representative $\tilde{w} \in \U(V)$ of $w$, we define an unnormalized intertwining operator
\[
 \M(\tilde{w}, \tau_s \otimes \sigma_0) : \Ind^{\U(V)}_P(\tau_s \otimes \sigma_0) \longrightarrow \Ind^{\U(V)}_P(w(\tau_s \otimes \sigma_0))
\]
by (the meromorphic continuation of) the integral
\[
 \M(\tilde{w}, \tau_s \otimes \sigma_0) \Phi_s(h)
 = \int_{U_P} \Phi_s(\tilde{w}^{-1} u h) \, du,
\]
where $w(\tau_s \otimes \sigma_0)$ is the representation of $M_P$ on $\V_{\tau} \otimes \V_{\sigma_0}$ given by $(w(\tau_s \otimes \sigma_0))(m) = (\tau_s \otimes \sigma_0)(\tilde{w}^{-1} m \tilde{w})$ for $m \in M_P$.

Now, following \cite{a}, \cite{mok}, \cite{kmsw}, we shall normalize the intertwining operator $\M(\tilde{w}, \tau_s \otimes \sigma_0)$, depending on the choice of the Whittaker datum.
Having fixed the additive character $\psi$ and the trace zero element $\delta$, we define the sign $\epsilon(V)$ and use the Whittaker datum relative to 
\[
 \begin{cases}
  \psi^E = \psi(\frac{1}{2} \Tr_{E/F}(\delta \, \cdot \,)) & \text{if $\varepsilon = +1$;} \\
  \psi & \text{if $\varepsilon = -1$.}
 \end{cases}
\]
The definition of the normalized intertwining operator is very subtle because one has to choose the following data appropriately:
\begin{itemize}
 \item a representative $\tilde{w}$;
 \item a normalizing factor $r(w, \tau_s \otimes \sigma_0)$;
 \item an intertwining isomorphism $\A_w$.
\end{itemize}

Following the procedure of \cite[\S 2.1]{ls}, \cite[\S 2.3]{a}, \cite[\S 3.3]{mok}, \cite[\S 2.3]{kmsw}, we take the representative $\tilde{w} \in \U(V)$ of $w$ defined by
\[
 \tilde{w} = w_P \cdot m_P((-1)^{m'} \cdot \kappa_V \cdot J) \cdot (- 1_{V_0})^k,  
\]
where $m' = [\frac{m}{2}]$,
\[
 \kappa_V = 
 \begin{cases}
  -\delta & \text{if $m$ is even and $\varepsilon = +1$;} \\
  1 & \text{if $m$ is even and $\varepsilon = -1$;} \\
  -1 & \text{if $m$ is odd and $\varepsilon = +1$;} \\
  -\delta & \text{if $m$ is odd and $\varepsilon = -1$,}
 \end{cases}
\]
and
\[
 J = 
 \begin{pmatrix}
  & & & (-1)^{k-1} \\
  & & \varddots & \\
  & -1 & & \\
  1 & & & 
 \end{pmatrix}
 \in \GL_k(E).
\]
Here, we have identified $\GL(X)$ with $\GL_k(E)$ using the basis $\{ v_1, \ldots, v_k \}$.
This element $\tilde{w}$ arises as follows.

First assume that $\epsilon(V) = +1$.
In particular, $\U(V)$ is quasi-split.
We have $V_\an = \{ 0 \}$ if $m$ is even and $V_\an = E v_\an$ for some $v_\an \in V_\an$ such that 
\[
 \langle v_\an, v_\an \rangle_V = 
 \begin{cases}
  1 & \text{if $\varepsilon = +1$;} \\
  \delta & \text{if $\varepsilon = -1$}
 \end{cases}
\]
if $m$ is odd.
Via the decomposition
\[
 V = E v_1 \oplus \dots  \oplus E v_r \oplus V_\an \oplus E v_r^* \oplus \dots  \oplus E v_1^*,
\]
we regard $\U(V)$ as a subgroup of $\GL_m(E)$, which induces an isomorphism $\U(V)(\bar{F}) \cong \GL_m(\bar{F})$.
Let $\spl = (B,T,\{ X_i \})$ be the $F$-splitting of $\U(V)$ consisting of the Borel subgroup $B$ stabilizing the flag
\[
 E v_1 \subset E v_1 \oplus E v_2 \subset \dots \subset E v_1 \oplus \dots  \oplus E v_r,
\]
the maximal torus $T$ of diagonal matrices, and the set $\{ X_i \, | \, i = 1, \ldots, m-1 \}$ of simple root vectors given as follows:
\begin{itemize}
 \item $X_i = E_{i,i+1}$ for $1 \le i \le r-1$;
 \item $X_i = -E_{i,i+1}$ for $m-r+1 \le i \le m-1$;
 \item if $m$ is even, then
 \[
  X_r = 
  \begin{cases}
   \delta^{-1} \cdot E_{r,r+1} & \text{if $\varepsilon = +1$,} \\
   E_{r,r+1} & \text{if $\varepsilon = -1$;}
  \end{cases}
 \]
 \item if $m$ is odd, then $X_r = E_{r,r+1}$ and 
 \[
  X_{r+1} = 
  \begin{cases}
   -E_{r+1,r+2} & \text{if $\varepsilon = +1$,} \\
   \delta^{-1} \cdot E_{r+1,r+2} & \text{if $\varepsilon = -1$.}
  \end{cases}
 \]
\end{itemize}
Here, $E_{i,j} \in \operatorname{Lie} \U(V)(\bar{F}) \cong \mathrm{M}_m(\bar{F})$ is the matrix with one at the $(i,j)$-th entry and zero elsewhere.
Then $\spl$ and $\psi$ give rise to the above Whittaker datum.
Let $\tilde{w}^{\LS}$ be the representative of $w$ defined in \cite[\S 2.1]{ls}, \cite[\S 2.3]{a}, \cite[\S 3.3]{mok} with respect to $\spl$.

\begin{lem}
\label{lem:weyl}
We have $\tilde{w}^{\LS} = \tilde{w}$.
\end{lem}

\begin{proof}
First, we review the case of $\SL_2$.
We take an $F$-splitting of $\SL_2$ consisting of the Borel subgroup of upper triangular matrices, the maximal torus of diagonal matrices, and a simple root vector
\[
 X = \begin{pmatrix} 0 & a \\ 0 & 0 \end{pmatrix}.
\]
Let $\{ H, X, Y \}$ be the $\mathfrak{sl}_2$-triple containing $X$, so that
\[
 Y = \begin{pmatrix} 0 & 0 \\ a^{-1} & 0 \end{pmatrix}.
\]
If $s$ is the simple reflection with respect to $X$, then the representative of $s$ defined in \cite[\S 2.1]{ls} is
\[
 \exp(X) \exp(-Y) \exp(X) = \begin{pmatrix} & a \\ -a^{-1} & \end{pmatrix}.
\]

Now we compute $\tilde{w}^\LS$.
Let $\iota_i : \GL(E v_i \oplus E v_{i+1}) \hookrightarrow \GL(X)$ and $\iota'_j : \U(E v_j \oplus E v_j^*) \hookrightarrow \U(V)$ be the natural embeddings.
Let $s_i$ be the simple reflection with respect to $X_i$ and $\tilde{s}_i$ the representative of $s_i$ as above.
Put $w_i = s_i s_{m-i}$ and $\tilde{w}_i = \tilde{s}_i \tilde{s}_{m-i}$ for $1 \le i \le r-1$, and 
\[
 w_r = 
 \begin{cases}
  s_r & \text{if $m$ is even,} \\
  s_r s_{r+1} s_r & \text{if $m$ is odd}
 \end{cases}
 \quad \text{and} \quad
 \tilde{w}_r = 
 \begin{cases}
  \tilde{s}_r & \text{if $m$ is even,} \\
  \tilde{s}_r \tilde{s}_{r+1} \tilde{s}_r & \text{if $m$ is odd.}
 \end{cases}
\]
More explicitly, we have
\[
 \tilde{w}_i = m_P \left( \iota_i \begin{pmatrix} & 1 \\ -1 & \end{pmatrix}
 \right)
\]
for $1 \le i \le r-1$ and 
\[
 \tilde{w}_r = \iota'_r \left(
 \begin{pmatrix}
 & 1 \\
 \varepsilon &
 \end{pmatrix}
 \begin{pmatrix}
 \kappa_V & \\
 & (\kappa_V^c)^{-1}
 \end{pmatrix}
 \right)
 \cdot (-1_{V_{\an}}).
\]
Put
\begin{align*}
 x_i & = w_{k-1} \cdots w_{i+1} w_i, &
 y_j & = w_j w_{j+1} \cdots w_{r-1} w_r w_{r-1} \cdots w_{j+1} w_j, \\
 \tilde{x}_i & = \tilde{w}_{k-1} \cdots \tilde{w}_{i+1} \tilde{w}_i, &
 \tilde{y}_j & = \tilde{w}_j \tilde{w}_{j+1} \cdots \tilde{w}_{r-1} \tilde{w}_r \tilde{w}_{r-1} \cdots w_{j+1} \tilde{w}_j
\end{align*}
for $1 \le i \le k-1$ and $1 \le j \le k$.
Let $w_T$ be the representative of $w$ in the Weyl group for $T$ which preserves the set of roots of $T$ in $B \cap M_P$.
Then $w_T$ has a reduced expression
\[
 w_T = y_k x_1 y_k x_2 \cdots y_k x_{k-1} y_k
\]
and hence $\tilde{w}^\LS$ is defined by
\[
 \tilde{w}^\LS = \tilde{y}_k \tilde{x}_1 \tilde{y}_k \tilde{x}_2 \cdots \tilde{y}_k \tilde{x}_{k-1} \tilde{y}_k.
\]
If we put $\tilde{x}_i' = \tilde{w}_{k-1}^{-1} \cdots \tilde{w}_{i+1}^{-1} \tilde{w}_i^{-1}$, then we have $\tilde{y}_k \tilde{x}_i = \tilde{x}'_i \tilde{y}_i$, so that 
\[
 \tilde{w}^\LS = \tilde{x}'_1 \tilde{y}_1 \tilde{x}'_2 \tilde{y}_2 \cdots \tilde{x}'_{k-1} \tilde{y}_{k-1} \tilde{y}_k.
\]
On the other hand, we have
\[
 \tilde{x}'_i = m_P
 \begin{pmatrix}
  1_{i-1} & & \\
  & & -1_{k-i} \\
  & 1 & 
 \end{pmatrix}
\]
and
\[
 \tilde{y}_j =
 \iota'_j \left(
 \begin{pmatrix}
 & 1 \\
 \varepsilon &
 \end{pmatrix}
 \begin{pmatrix}
 \kappa_V & \\
 & (\kappa_V^c)^{-1}
 \end{pmatrix}
 \right)
 \cdot m_P \begin{pmatrix} 1_{j-1} & & \\ & (-1)^{r-j} & \\ & & - 1_{k-j}
 \end{pmatrix}
 \cdot (-1_{V_0}).
\]
In particular, $\tilde{x}_i'$ commutes with $\tilde{y}_j$ if $i > j$, so that
\[
 \tilde{w}^\LS = \tilde{x}'_1 \tilde{x}'_2 \cdots \tilde{x}'_{k-1} \tilde{y}_1 \tilde{y}_2 \cdots \tilde{y}_{k-1} \tilde{y}_k.
\]
Since $\tilde{x}'_1 \cdots \tilde{x}'_{k-1} = m_P(J)$ and 
\begin{align*}
 \tilde{y}_1 \cdots \tilde{y}_k
 & = \prod_{j=1}^k
 \iota'_j \left(
 \begin{pmatrix}
 & 1 \\
 \varepsilon &
 \end{pmatrix}
 \begin{pmatrix}
 \kappa_V & \\
 & (\kappa_V^c)^{-1}
 \end{pmatrix}
 \right)
 \cdot m_P((-1)^{r-1} \cdot 1_k) \cdot (-1_{V_0})^k \\
 & = \prod_{j=1}^k
 \iota'_j 
 \begin{pmatrix}
 & 1 \\
 \varepsilon &
 \end{pmatrix}
 \cdot m_P((-1)^{r-1} \cdot \kappa_V \cdot 1_k) \cdot (-1_{V_0})^k,
\end{align*}
the assertion follows.
\end{proof}

Next, we consider the case $\epsilon(V) = -1$.
Let $V^+$ be the $m$-dimensional $\varepsilon$-Hermitian space with $\epsilon(V^+) = +1$.
We may assume that $V^+ = X \oplus V_0^+ \oplus X^*$ for some $m_0$-dimensional $\varepsilon$-Hermitian space $V_0^+$ with $\epsilon(V_0^+) = +1$.
Let $P^+$ be the maximal parabolic subgroup of $\U(V^+)$ stabilizing $X$ and $M_{P^+}$ its Levi component stabilizing $X^*$, so that $M_{P^+} \cong \GL(X) \times \U(V_0^+)$.
Fix an isomorphism $V_0^+ \otimes_F \bar{F} \cong V_0 \otimes_F \bar{F}$ as $\varepsilon$-Hermitian spaces over $E \otimes_F \bar{F}$ and extend it to an isomorphism $V^+ \otimes_F \bar{F} \cong V \otimes_F \bar{F}$ whose restriction to $(X \otimes_F \bar{F}) \oplus (X^*\otimes_F \bar{F})$ is the identity map.
This induces a pure inner twist $(\xi, z)$, i.e.~$\xi : \U(V^+) \rightarrow \U(V)$ is an inner twist and $z \in Z^1(\Gamma, \U(V^+))$ is a $1$-cocyle such that $\xi^{-1} \circ \sigma \circ \xi \circ \sigma^{-1} = \Ad(z(\sigma))$ for all $\sigma \in \Gamma$.
Then $P^+ = \xi^{-1}(P)$ and $\xi$ induces an inner twist $\xi : M_{P^+} \rightarrow M_P$ whose restriction to $\GL(X)$ is the identity map.
Moreover, $z$ satisfies the assumption in \cite[\S 2.4.1]{kmsw}.
Let $w^+$ be the nontrivial element in the relative Weyl group for $M_{P^+}$ and $\tilde{w}^+ \in \U(V^+)$ the representative of $w^+$ as above.
Then the representative of $w$ defined in \cite[\S 2.3]{kmsw} is $\xi(\tilde{w}^+)$, which is equal to $\tilde{w}$.

We use the normalizing factor $r(w, \tau_s \otimes \sigma_0)$ defined as follows.
Let $\lambda(E/F, \psi)$ be the Langlands $\lambda$-factor (see \cite[\S 5]{deligne}) and put
\[
 \lambda(w, \psi) = 
 \begin{cases}
 \lambda(E/F, \psi)^{(k-1)k/2} & \text{if $m$ is even;} \\
 \lambda(E/F, \psi)^{(k+1)k/2} & \text{if $m$ is odd.}
 \end{cases}
\]
Let $\phi_\tau$ and $\phi_0$ be the $L$-parameters of $\tau$ and $\sigma_0$ respectively.
Let $\As^+$ be the Asai representation of the $L$-group of $\Res_{E/F} \GL_k$ and $\As^- = \As^+ \otimes \omega_{E/F}$ its twist (see \cite[\S 7]{ggp1}).
If we set 
\[
 r(w, \tau_s \otimes \sigma_0) = 
 \lambda(w, \psi) \cdot \gamma(s, \phi_\tau \otimes \phi_0^\vee, \psi_E)^{-1} \cdot \gamma(2s, \As^{(-1)^m} \circ \phi_\tau, \psi)^{-1},
\]
then by \cite[Lemmas 2.2.3 and 2.3.1]{kmsw}, the normalized intertwining operator
\[
 \R(w, \tau_s \otimes \sigma_0) := |\kappa_V|^{k \rho_P} \cdot r(w,\tau_s \otimes \sigma_0)^{-1} \cdot \M(\tilde{w}, \tau_s \otimes \sigma_0)
\]
is holomorphic at $s=0$ and satisfies
\[
 \R(w, w(\tau_s \otimes \sigma_0)) \circ \R(w, \tau_s \otimes \sigma_0) = 1.
\]
Here, the factor $|\kappa_V|^{k \rho_P}$ arises because the Haar measure on $U_P$ defined in \cite[\S 2.2]{kmsw} with respect to $\spl$ is equal to $|\kappa_V|^{k \rho_P} \, du$.

Now assume that $w(\tau \otimes \sigma_0) \cong \tau \otimes \sigma_0$, which is equivalent to $(\tau^c)^\vee \cong \tau$.
We take the unique isomorphism
\[
 \A_w : \V_{\tau} \otimes \V_{\sigma_0} \longrightarrow \V_{\tau} \otimes \V_{\sigma_0}
\]
such that:
\begin{itemize}
 \item $\A_w \circ (w(\tau \otimes \sigma_0))(m) = (\tau \otimes \sigma_0)(m) \circ \A_w$ for all $m \in M_P$;
 \item $\A_w = \A'_w \otimes 1_{\V_{\sigma_0}}$ with an isomorphism $\A'_w : \V_\tau \rightarrow \V_\tau$ such that $\Lambda \circ \A_w = \Lambda$. Here, $\Lambda : \V_\tau \rightarrow \CC$ is the unique (up to a scalar) Whittaker functional with respect to the Whittaker datum $(N_k, \psi_{N_k})$, 
where $N_k$ is the group of unipotent upper triangular matrices in $\GL_k(E)$ and $\psi_{N_k}$ is the generic character of $N_k$ given by $\psi_{N_k}(x) = \psi_E(x_{1,2} + \cdots + x_{k-1,k})$.
\end{itemize}
Note that $\A_w^2 = 1_{\V_{\tau} \otimes \V_{\sigma_0}}$.
We define a self-intertwining operator 
\[
  R(w, \tau \otimes \sigma_0) :
 \Ind^{\U(V)}_P(\tau \otimes \sigma_0) \longrightarrow 
 \Ind^{\U(V)}_P(\tau \otimes \sigma_0)
\]
by
\[
 R(w,\tau \otimes \sigma_0) \Phi(h) = \A_w(\R(w,\tau \otimes \sigma_0) \Phi(h)).
\]
By construction,
\[
 R(w,\tau \otimes \sigma_0)^2 = 1.
\]

\subsection{Weil representations}

In this subsection, we recall some explicit formulas for the Weil representations. 

Let $\WW$ be a finite dimensional vector space over $F$ equipped with a nondegenerate symplectic form $\langle \cdot , \cdot \rangle_\WW : \WW \times \WW \rightarrow F$.
Let $\H(\WW) = \WW \oplus F$ be the associated Heisenberg group, i.e.~the multiplication law is given by 
\[
 (w, t) \cdot (w', t') = \left( w+w', t+t' + \frac{1}{2} \langle w, w' \rangle_\WW \right)
\]
for $w, w' \in \WW$ and $t, t' \in F$.
Fix maximal totally isotropic subspaces $\XX$ and $\XX^*$ of $\WW$ such that $\WW = \XX \oplus \XX^*$.
Let $\rho$ be the Heisenberg representation of $\H(\WW)$ on $\s(\XX^*)$ with central character $\psi$.
Namely, 
\[
 \rho((x+x', t)) \varphi(x'_0)
 = \psi(t + \langle x'_0, x \rangle_\WW + \tfrac{1}{2} \langle x', x \rangle_\WW)
 \varphi(x'_0+x')
\]
for $\varphi \in \s(\XX^*)$, $x \in \XX$, $x', x'_0 \in \XX^*$, and $t \in F$.

In \S \ref{SS:Weil}, we have introduced the Weil representations for unitary groups.
To define these representations, we have fixed the additive character $\psi$ and the pair of characters $(\chi_V, \chi_W)$.
For simplicity, we write:
\begin{itemize}
 \item $\omega$ for the Weil representation $\omega_{\psi, \chi_V, \chi_W, V, W}$ of $\U(V) \times \U(W)$ on a space $\s$;
 \item $\omega_0$ for the Weil representation $\omega_{\psi, \chi_V, \chi_W, V, W_0}$ of $\U(V) \times \U(W_0)$ on a space $\s_0$;
 \item $\omega_{00}$ for the Weil representation $\omega_{\psi, \chi_V, \chi_W, V_0, W_0}$ of $\U(V_0) \times \U(W_0)$ on a space $\s_{00}$.
\end{itemize}
We take a mixed model
\[
 \s = \s(V \otimes Y^*) \otimes \s_0
\]
of $\omega$, where we regard $\s$ as a space of functions on $V \otimes Y^*$ with values in $\s_0$.
Similarly, we take a mixed model
\[
 \s_0 = \s(X^* \otimes W_0) \otimes \s_{00}
\]
of $\omega_0$, where we regard $\s_0$ as a space of functions on $X^* \otimes W_0$ with values in $\s_{00}$.
Also, we write:
\begin{itemize}
\item $\rho_0$ for the Heisenberg representation of $\H(V \otimes W_0)$ on $\s_0$ with central character $\psi$;
\item $\rho_{00}$ for the Heisenberg representation of $\H(V_0 \otimes W_0)$ on $\s_{00}$ with central character $\psi$.
\end{itemize}

Using \cite[Theorem 3.1]{k}, we can derive the following formulas for the Weil representations $\omega$ and $\omega_0$.
Put $\Delta = \delta^2 \in F^{\times}$.
As in \cite[Appendix]{rangarao}, let $\gamma_F(\psi)$ be the Weil index of the character $x \mapsto \psi(x^2)$ of second degree and set 
\[  \gamma_F(a,\psi) = \frac{\gamma_F(\psi_a)}{\gamma_F(\psi)} \]
for $a \in F^{\times}$, where $\psi_a(x) = \psi(ax)$.
Note that $\gamma_F(\Delta, \psi) = \lambda(E/F, \psi)^{-1}$.
For $\varphi \in \s$ and $x \in V \otimes Y^*$, we have 
\begin{align*}
 (\omega(h) \varphi)(x) & = \omega_0(h) \varphi(h^{-1} x), &
 h & \in \U(V), \\
 (\omega(g_0) \varphi)(x) & = \omega_0(g_0) \varphi(x), &
 g_0 & \in \U(W_0), \\
 (\omega(m_Q(a)) \varphi)(x) & = \chi_V(\det a) |\det a|^{m/2} \varphi(a^*x), &
 a & \in \GL(Y), \\ 
 (\omega(u_Q(b)) \varphi)(x) & = \rho_0((b^* x, 0)) \varphi(x), &
 b & \in \Hom(W_0, Y), \\
 (\omega(u_Q(c)) \varphi)(x) & = \psi(\tfrac{1}{2} \langle cx, x \rangle) \varphi(x), &
 c & \in \Herm(Y^*,Y), \\
 (\omega(w_Q) \varphi)(x) & = \gamma_V^{-k} 
 \int_{V \otimes Y} \varphi(-I_Y^{-1}y) \psi(\langle y, x\rangle) \, dy, & &
\end{align*}
where 
\[
 \gamma_V = 
 \begin{cases}
  \omega_{E/F}(\det V) \cdot \gamma_F(-\Delta, \psi)^m \cdot \gamma_F(-1,\psi)^{-m}
  & \text{if $\varepsilon = +1$;} \\
  \chi_V(\delta)^{-1} \cdot \omega_{E/F}(\delta^{-m} \cdot \det V) \cdot \gamma_F(-\Delta, \psi)^m \cdot \gamma_F(-1,\psi)^{-m}
  & \text{if $\varepsilon = -1$.}
 \end{cases}
\]
Also, for $\varphi_0 \in \s_0$ and $x \in X^* \otimes W_0$, we have
\begin{align*}
 (\omega_0(g_0) \varphi_0)(x) & = \omega_{00}(g_0) \varphi_0 (g_0^{-1} x), &
 g_0 & \in \U(W_0), \\
 (\omega_0(h_0) \varphi_0)(x) & = \omega_{00}(h_0) \varphi_0(x), & 
 h_0 & \in \U(V_0), \\
 (\omega_0(m_P(a)) \varphi_0)(x)
 & = \chi_W(\det a) |\det a|^{n_0/2} \varphi_0(a^*x), &
 a & \in \GL(X), \\
 (\omega_0(u_P(b)) \varphi_0)(x) & = \rho_{00}((b^*x, 0)) \varphi_0(x), &
 b & \in \Hom(V_0, X), \\
 (\omega_0(u_P(c)) \varphi_0)(x) & = \psi(\tfrac{1}{2} \langle cx, x \rangle ) \varphi_0(x), &
 c & \in \Herm(X^*, X), \\
 (\omega_0(w_P) \varphi_0)(x)
 & = \gamma_{W}^{-k} \int_{X \otimes W_0} \varphi_0(-I_X^{-1}y) \psi(\langle y,x \rangle) \, dy, & & \\
 (\rho_0((y+y',0)) \varphi_0)(x) & = \psi(\langle x, y \rangle + \tfrac{1}{2}\langle y', y \rangle)
 \varphi_0(x + y'), &
 y & \in X \otimes W_0, \, y' \in X^* \otimes W_0, \\
 (\rho_0((y_0,0)) \varphi_0)(x) & = \rho_{00}((y_0,0)) \varphi_0(x), & y_0 & \in V_0 \otimes W_0,
\end{align*}
where 
\[
 \gamma_W = 
 \begin{cases}
  \chi_W(\delta)^{-1} \cdot \omega_{E/F}(\delta^{-n} \cdot \det W) \cdot \gamma_F(-\Delta, \psi)^n \cdot \gamma_F(-1,\psi)^{-n}
  & \text{if $\varepsilon = +1$;} \\
  \omega_{E/F}(\det W) \cdot \gamma_F(-\Delta, \psi)^n \cdot \gamma_F(-1,\psi)^{-n}
  & \text{if $\varepsilon = -1$.}
 \end{cases}
\]

\subsection{Zeta integrals of Godement--Jacquet}

In this subsection, we review the theory of local factors for $\GL_k$ developed by Godement--Jacquet \cite{gj}.

Let $\tau$ be an irreducible smooth representation of $\GL_k(E)$ on a space $\V_{\tau}$ with central character $\omega_{\tau}$.
For any character $\chi$ of $E^{\times}$, we realize the representation $\tau \chi := \tau \otimes (\chi \circ \det)$ on $\V_{\tau}$ by setting $(\tau \chi)(a) v := \chi(\det a) \tau(a) v$ for $a \in \GL_k(E)$ and $v \in \V_{\tau}$. Put $\tau_s := \tau |\cdot|^s$ for $s \in \CC$.
Let $\tau^c$ be the representation of $\GL_k(E)$ on $\V_{\tau}$ defined by $\tau^c(a) = \tau(a^c)$.
We write 
\[
 L(s, \tau) = L(s, \phi_\tau) \quad \text{and} \quad
 \epsilon(s, \tau, \psi_E) =  \epsilon(s, \phi_\tau, \psi_E)
\]
for the standard $L$-factor and $\epsilon$-factor of $\tau$, where $\phi_\tau$ is the $k$-dimensional representation of $\WD_E$ associated to $\tau$ and $\psi_E$ is the nontrivial additive character of $E$ defined by $\psi_E = \psi \circ \Tr_{E/F}$.
Then the standard $\gamma$-factor of $\tau$ is defined by
\[
 \gamma(s, \tau, \psi_E) = \epsilon(s, \tau, \psi_E) \cdot \frac{L(1-s, \tau^\vee)}{L(s, \tau)},
\]
where $\tau^\vee$ is the contragredient representation of $\tau$.

For $s \in \CC$, $\phi \in \s(\mathrm{M}_k(E))$, and a matrix coefficient $f$ of $\tau$, put
\[
 Z(s, \phi, f) =  \int_{\GL_k(E)} \phi(a) f(a) |\det a|^s \, da,
\]
where we have fixed a Haar measure $da$ on $\GL_k(E)$.
This integral is absolutely convergent for $\Re(s) \gg 0$ and admits a meromorphic continuation to $\CC$.
Moreover, 
\[
 \frac{Z(s+\frac{k-1}{2}, \phi, f)}{L(s,\tau)}
\]
is an entire function of $s$.
If $\tau$ is square-integrable, then $Z(s, \phi, f)$ is absolutely convergent for $\Re(s) > \frac{k-1}{2}$ by \cite[Proposition 1.3]{gj}.

Let $\hat{\phi} \in \s(\mathrm{M}_k(E))$ be the Fourier transform of $\phi$ defined by 
\[
 \hat{\phi}(x) = \int_{\mathrm{M}_k(E)} \phi(y) \psi_E(\Tr(xy)) \, dy,
\]
where $dy$ is the self-dual Haar measure on ${\mathrm{M}_k(E)}$ with respect to the pairing $(x, y) \mapsto \psi_E(\Tr(xy))$.
Let $\check{f}$ be the matrix coefficient of $\tau^\vee$ given by $\check{f}(a) = f(a^{-1})$.
Then the local functional equation asserts that
\[
 Z(-s+\tfrac{k+1}{2}, \hat{\phi}, \check{f}) = \gamma(s, \tau, \psi_E) \cdot Z(s+\tfrac{k-1}{2}, \phi, f). 
\]

\section{\textbf{Proof of Theorem \ref{T:ind}}}

Now we can begin the proof of Theorem \ref{T:ind}.
This will be proved by an explicit construction of an equivariant map which realizes the theta correspondence. 

\subsection{Construction of equivariant maps}

Recall that we have identified $\GL(X)$ with $\GL_k(E)$ using the basis $\{ v_1, \ldots, v_k \}$.
Similarly, we identify $\GL(Y)$ with $\GL_k(E)$ using the basis $\{ w_1, \ldots, w_k \}$.
Thus we can define an isomorphism $i:\GL(Y) \rightarrow \GL(X)$ via these identifications.
Put 
\[
 e = v_1 \otimes w_1^* + \cdots + v_k \otimes w_k^* \in X \otimes Y^*, \quad
 e^* = v^*_1 \otimes w_1 + \cdots + v_k^* \otimes w_k \in X^* \otimes Y.
\]
Then $i(a)^c e = a^* e$ and $(i(a)^c)^* e^* = a e^*$ for $a \in \GL(Y)$.

For $\varphi \in \s = \s(V \otimes Y^*) \otimes \s_0$, we define functions $\ff(\varphi)$, $\hat{\ff}(\varphi)$ on $\U(W) \times \U(V)$ with values in $\s_0$ by
\[
 \ff(\varphi)(gh)
 = (\omega(gh)\varphi)
 \begin{pmatrix}
  e \\
  0 \\
  0 
 \end{pmatrix}, \quad
 \hat{\ff}(\varphi)(gh) = \int_{X \otimes Y^*} (\omega(gh)\varphi)
 \begin{pmatrix}
  x \\
  0 \\
  0
 \end{pmatrix}
 \psi(\varepsilon \langle x, e^* \rangle) \, dx
\]
for $g \in \U(W)$ and $h \in \U(V)$.
Here, we write an element in $V \otimes Y^*$ as a block matrix
\[
 \begin{pmatrix}
  y_1 \\
  y_2 \\
  y_3
 \end{pmatrix}
\]
with $y_1 \in X \otimes Y^*$, $y_2 \in V_0 \otimes Y^*$, and $y_3 \in X^* \otimes Y^*$.
We also define functions $f(\varphi)$, $\hat{f}(\varphi)$ on $\U(W) \times \U(V)$ with values in $\s_{00}$ by
\[
 f(\varphi)(gh) = \ev(\ff(\varphi)(gh)), \quad
 \hat{f}(\varphi)(gh) = \ev(\hat{\ff}(\varphi)(gh)),
\]
where $\ev : \s_0 = \s(X^* \otimes W_0) \otimes \s_{00} \rightarrow \s_{00}$ is the evaluation at $0 \in X^* \otimes W_0$.
If $f = f(\varphi)$ or $\hat{f}(\varphi)$, then
\begin{align*}
 f(u u' gh) & = f(gh), & u & \in U_Q, \, u' \in U_P, \\ 
 f(g_0h_0gh) & = \omega_{00}(g_0h_0) f(gh), & g_0 & \in \U(W_0), \, h_0 \in \U(V_0), \\ 
 f(m_Q(a) m_P(i(a)^c) gh) & = (\chi_V \chi_W^c)(\det a) |\det a|^{\rho_P+\rho_Q} f(gh), & a & \in \GL(Y).
\end{align*}
Note that this realizes the bottom piece of Kudla's filtration \cite{kudla}.
(See also \cite[Lemma C.2]{gi}, but in which $\operatorname{Isom}(Y_a', X_a)$ is a typo and should be read as the set of invertible conjugate linear maps from $Y_a'$ to $X_a$.)

Let $\tau$ be an irreducible (unitary) square-integrable representation of $\GL_k(E)$ on a space $\V_{\tau}$.
We may regard $\tau$ as a representation of $\GL(X)$ or $\GL(Y)$ via the above identifications.
Let $\pi_0$ and $\sigma_0$ be irreducible tempered representations of $\U(W_0)$ and $\U(V_0)$ on spaces $\V_{\pi_0}$ and $\V_{\sigma_0}$ respectively.
Fix nonzero invariant nondegenerate bilinear forms $\langle \cdot , \cdot \rangle$ on $\V_{\tau} \times \V_{\tau^\vee}$, $\V_{\pi_0} \times \V_{\pi_0^\vee}$, and $\V_{\sigma_0} \times \V_{\sigma_0^\vee}$.
Let 
\[
 \langle \cdot, \cdot \rangle : (\V_{\tau} \otimes \V_{\sigma_0^\vee}) \times \V_{\tau^\vee} \longrightarrow \V_{\sigma_0^\vee}
\]
be the induced map.

Now assume that 
\[
 \sigma_0 =\Theta_{\psi, V_0, W_0}(\pi_0).
\]
We fix a nonzero $\U(V_0) \times \U(W_0)$-equivariant map
\[
 \T_{00}: \omega_{00} \otimes \sigma_0^\vee \longrightarrow \pi_0.
\]
For $\varphi \in \s$, $\Phi_s \in \Ind^{\U(V)}_P(\tau^c_s \chi_W^c \otimes \sigma_0^\vee)$, $g \in \U(W)$,
$\check{v} \in \V_{\tau^\vee}$, and $\check{v}_0 \in \V_{\pi_0^\vee}$, put
\[
 \langle \T_s(\varphi, \Phi_s)(g), \check{v} \otimes \check{v}_0 \rangle = 
 L(s-s_0 + \tfrac{1}{2}, \tau)^{-1}
 \cdot \int_{U_P \U(V_0) \backslash \U(V)}
 \langle \T_{00}(\hat{f}(\varphi)(gh), \langle \Phi_s(h), \check{v} \rangle ), \check{v}_0 \rangle \, dh,
\]
where we have fixed Haar measures on $\U(V)$ and $\U(V_0)$, and set
\[
 s_0 = \frac{m-n}{2} = \frac{m_0-n_0}{2}.
\]
Note that $\langle \Phi_s(h), \check{v} \rangle \in \V_{\sigma_0^\vee}$.

\begin{lem}
\label{lem:def}
The integral $\langle \T_s(\varphi, \Phi_s)(g), \check{v} \otimes \check{v}_0 \rangle$ is absolutely convergent for $\Re(s) > s_0 - \frac{1}{2}$ and admits a holomorphic continuation to $\CC$.
\end{lem}

\begin{proof}
We may assume that $\varphi = \varphi' \otimes \varphi_0$ and $\Phi_s(1) = v \otimes v_0$, where $\varphi' \in \s(V \otimes Y^*)$, $\varphi_0 \in \s_0$, $v \in \V_{\tau}$, and $v_0 \in \V_{\sigma_0^\vee}$.
By the Iwasawa decomposition, it suffices to consider the integral
\begin{equation}
\label{eq1}
 \int_{\GL(X)} \langle \T_{00}(\hat{f}(\varphi)(m_P(a)), \langle \Phi_s(m_P(a)), \check{v} \rangle ), \check{v}_0 \rangle
 |\det a|^{-2 \rho_P} \, da.
\end{equation}
Put
\[
 \phi(y)
 = \int_{X \otimes Y^*} \varphi'
 \begin{pmatrix}
  x \\
  0 \\
  0
 \end{pmatrix}
 \psi(\varepsilon \langle x, y \rangle ) \, dx
\]
for $y \in X^* \otimes Y$.
Then we have
\[
 \hat{f}(\varphi)(m_P(a))
 = \chi_W(\det a) |\det a|^{k + n_0/2} \phi(a^* e^*) \cdot \ev(\varphi_0)
\]
for $a \in \GL(X)$.
Hence we have
\begin{align*}
 \eqref{eq1} & = \langle \T_{00}(\ev(\varphi_0), v_0), \check{v}_0 \rangle
 \cdot \int_{\GL(X)} \phi(a^* e^*) \langle \tau(a^c) v, \check{v} \rangle |\det a|^{s-s_0+k/2} \, da.
\end{align*}
This completes the proof.
\end{proof}

Thus we obtain a $\U(V) \times \U(W)$-equivariant map
\[
 \T_s:\omega \otimes \Ind^{\U(V)}_P(\tau^c_s \chi_W^c \otimes \sigma_0^\vee) \longrightarrow
 \Ind^{\U(W)}_Q(\tau_s \chi_V \otimes \pi_0).
\]

\begin{lem}
\label{lem:def2}
If $\Re (s) < s_0 + \frac{1}{2}$, then we have
\begin{align*}
 & \langle \T_s(\varphi, \Phi_s)(g), \check{v} \otimes \check{v}_0 \rangle \\
 & = L(s-s_0+\tfrac{1}{2}, \tau)^{-1} \cdot \gamma(s-s_0+\tfrac{1}{2}, \tau, \psi_E)^{-1} \cdot \int_{U_P \U(V_0) \backslash \U(V)}
 \langle \T_{00}(f(\varphi)(gh), \langle \Phi_s(h), \check{v} \rangle ), \check{v}_0 \rangle \, dh.
\end{align*}
\end{lem}

\begin{proof}
We may assume that $\varphi = \varphi' \otimes \varphi_0$ and $\Phi_s(1) = v \otimes v_0$, where $\varphi' \in \s(V \otimes Y^*)$, $\varphi_0 \in \s_0$, $v \in \V_{\tau}$, and $v_0 \in \V_{\sigma_0^\vee}$.
Put $f(a) = \langle \tau(a) v, \check{v} \rangle$ for $a \in \GL(X)$.
Let $\phi \in \s(X^* \otimes Y)$ be as in the proof of Lemma \ref{lem:def}.
We define its Fourier transform $\hat{\phi} \in \s(X \otimes Y^*)$ by
\[
 \hat{\phi}(x)
 = \int_{X^* \otimes Y} \phi(y) \psi(-\varepsilon\langle x, y \rangle ) \, dy.
\]
By the Fourier inversion formula, we have
\[
 \hat{\phi}(x) = \varphi'
 \begin{pmatrix}
  x \\
  0 \\
  0
 \end{pmatrix}.
\]
Hence we have
\[
  f(\varphi)(m_P(a))
 = \chi_W(\det a) |\det a|^{n_0/2} \hat{\phi}(a^{-1} e) \cdot \ev(\varphi_0)
\]
for $a \in \GL(X)$.
If $s_0-\frac{1}{2} < \Re (s) < s_0+\frac{1}{2}$, then by the local functional equation, we have
\begin{align*}
 & \int_{\GL(X)} \langle \T_{00}(\hat{f}(\varphi)(m_P(a)), \langle \Phi_s(m_P(a)), \check{v} \rangle ), \check{v}_0 \rangle
 |\det a|^{-2 \rho_P} \, da \\
 & = \langle \T_{00}(\ev(\varphi_0), v_0), \check{v}_0 \rangle
 \cdot \int_{\GL(X)} \phi(a^* e^*) f(a^c) |\det a|^{s-s_0+k/2} \, da \\
 & = \langle \T_{00}(\ev(\varphi_0), v_0), \check{v}_0 \rangle
 \cdot \gamma(s-s_0+\tfrac{1}{2}, \tau, \psi_E)^{-1} \cdot 
 \int_{\GL(X)} \hat{\phi}(ae) \check{f}(a^c) |\det a|^{-s+s_0+k/2} \, da \\
 & = \langle \T_{00}(\ev(\varphi_0), v_0), \check{v}_0 \rangle
 \cdot \gamma(s-s_0+\tfrac{1}{2}, \tau, \psi_E)^{-1} \cdot
 \int_{\GL(X)} \hat{\phi}(a^{-1} e) f(a^c) |\det a|^{s-s_0-k/2} \, da \\
 & = \gamma(s-s_0+\tfrac{1}{2}, \tau, \psi_E)^{-1} \cdot
 \int_{\GL(X)} \langle \T_{00}(f(\varphi)(m_P(a)), \langle \Phi_s(m_P(a)), \check{v} \rangle ), \check{v}_0 \rangle
 |\det a|^{-2\rho_P} \, da.
\end{align*}
This completes the proof.
\end{proof}

\begin{lem}
\label{lem:nonzero}
Assume that $m \ge n$.
Let $\Phi \in \Ind^{\U(V)}_P(\tau^c \chi_W^c \otimes \sigma_0^\vee)$.
If $\Phi \ne 0$, then there exists $\varphi \in \s$ such that
\[
 \T_0(\varphi, \Phi) \ne 0.
\]
\end{lem}

\begin{proof}
Fix a special maximal compact subgroup $K$ of $\U(V)$.
We extend $\Phi$ to a holomorphic section $\Phi_s$ of $\Ind^{\U(V)}_P(\tau_s^c \chi_W^c \otimes \sigma_0^\vee)$ so that $\Phi_s|_K$ is independent of $s$.
We have 
\[
 L(s-s_0+\tfrac{1}{2}, \tau)^{-1} \cdot 
 \gamma(s-s_0+\tfrac{1}{2}, \tau, \psi_E)^{-1} 
 = L(-s+s_0+\tfrac{1}{2}, \tau^\vee)^{-1}
\]
up to an invertible function.
Since $\tau$ is square-integrable and $s_0 \ge 0$, the right-hand side is holomorphic and nonzero at $s=0$.
By Lemma \ref{lem:def2}, 
it suffices to show that there exist $\varphi \in \s$,
$\check{v} \in \V_{\tau^\vee}$, and $\check{v}_0 \in \V_{\pi_0^\vee}$ such that 
\begin{equation}
\label{eq:nonzero}
 \int_{U_P \U(V_0) \backslash \U(V)}
 \langle \T_{00}(f(\varphi)(h), \langle \Phi_s(h), \check{v} \rangle ), \check{v}_0 \rangle \, dh
\end{equation}
is nonzero and independent of $s$ for $\Re(s) \ll 0$.

Let $\varphi = \varphi' \otimes \varphi_0$,
where $\varphi' \in \s(V \otimes Y^*)$ and $\varphi_0 \in \s_0$.
Then we have
\[
 \eqref{eq:nonzero} = 
 \int_{U_P \U(V_0) \backslash \U(V)} \varphi'(h^{-1} x_0) \Psi_s(h) \, dh,
\]
where
\[
 x_0 = 
 \begin{pmatrix}
  e \\
  0 \\
  0
 \end{pmatrix}, \quad
 \Psi_s(h) = \langle \T_{00}(\ev(\omega_0(h) \varphi_0), \langle \Phi_s(h), \check{v} \rangle ), \check{v}_0 \rangle.
\]
We can choose $\varphi_0$, $\check{v}$, and $\check{v}_0$ so that 
$\Psi_s|_K$ is nonzero and independent of $s$.
Since $h \mapsto h^{-1} x_0$ induces a homeomorphism
\[
 U_P \U(V_0) \backslash \U(V) \overset{\sim}{\longrightarrow} \U(V) x_0
\]
and $\U(V) x_0$ is locally closed in $V \otimes Y^*$,
there exists $\varphi'$ such that $\supp \varphi' \cap \U(V) x_0 = K x_0$
and such that $\varphi'(k^{-1} x_0) = \overline{\Psi_s(k)}$ for all $k \in K$.
Hence we have
\[
 \eqref{eq:nonzero} =
 \int_{U_P \U(V_0) \backslash U_P \U(V_0) K} \varphi'(h^{-1} x_0) \Psi_s(h) \, dh
 = \int_{(U_P \U(V_0) \cap K) \backslash K} |\Psi_s(k)|^2 \, dk \ne 0.
\]
Since $\Psi_s|_K$ is independent of $s$, so is this integral.
This completes the proof.
\end{proof}

\subsection{Compatibilities with intertwining operators}

Now we shall prove a key property of the equivariant map we have constructed.

Let $w \in W(M_P)$ and $w' \in W(M_Q)$ be the nontrivial elements in the relative Weyl groups.
As in \S \ref{SS:normalization}, we take the representatives $\tilde{w} \in \U(V)$ of $w$ and $\tilde{w}' \in \U(W)$ of $w'$ defined by 
\begin{align*}
 \tilde{w}  & = w_P \cdot m_P((-1)^{m'} \cdot \kappa_V \cdot J) \cdot (- 1_{V_0})^k, \\
 \tilde{w}' & = w_Q \cdot m_Q((-1)^{n'} \cdot \kappa_W \cdot J) \cdot (- 1_{W_0})^k,
\end{align*}
where $m' = [\frac{m}{2}]$ and $n' = [\frac{n}{2}]$.
Having fixed $\tau$, $\pi_0$, and $\sigma_0$, we shall write 
\[
 \M(\tilde{w},s) = \M(\tilde{w},\tau^c_s \chi_W^c \otimes \sigma_0^\vee)
 \quad \text{and} \quad
 \M(\tilde{w}',s) = \M(\tilde{w}',\tau_s \chi_V \otimes \pi_0)
\]
for the unnormalized intertwining operators, which are defined by the integrals
\[
 \M(\tilde{w},s) \Phi_s(h) = \int_{U_P} \Phi_s(\tilde{w}^{-1} u h) \, du, \quad
 \M(\tilde{w}',s) \Psi_s(g) = \int_{U_Q} \Psi_s(\tilde{w}'^{-1} u g) \, du 
\]
for $\Phi_s \in \Ind^{\U(V)}_P(\tau^c_s \chi_W^c \otimes \sigma_0^\vee)$ and $\Psi_s \in \Ind^{\U(W)}_Q(\tau_s \chi_V \otimes \pi_0)$.
By the Howe duality, the diagram
\[
  \xymatrix{
  \omega \otimes 
  \Ind^{\U(V)}_P(\tau^c_s \chi_W^c \otimes \sigma_0^\vee)
  \ar@{->}[rr]^{\T_s} \ar@{->}[d]_{1 \otimes \M(\tilde{w},s)} & &
  \Ind^{\U(W)}_Q(\tau_s \chi_V \otimes \pi_0) \ar@{->}[d]^{\M(\tilde{w}',s)} \\
  \omega \otimes 
  \Ind^{\U(V)}_P(w(\tau^c_s \chi_W^c \otimes \sigma_0^\vee))
   \ar@{->}[rr]^{\T_{-s}}
  & & \Ind^{\U(W)}_Q(w'(\tau_s \chi_V \otimes \pi_0))}
\]
commutes up to a scalar.
The following proposition determines this constant of proportionality explicitly.

\begin{prop} \label{P:io}
For $\varphi \in \s$ and $\Phi_s \in \Ind^{\U(V)}_P(\tau^c_s \chi_W^c \otimes \sigma_0^\vee)$, we have
\begin{align*}
 & \M(\tilde{w}', s) \T_s(\varphi, \Phi_s) \\
 & = ( \gamma_V^{-1} \cdot \gamma_W
 \cdot \chi_V((-1)^{n'} \cdot \varepsilon \cdot \kappa_W^{-1})
 \cdot \chi_W((-1)^{m'-1} \cdot \kappa_V^{-1})
 \cdot (\chi_V^{-n} \chi_W^m)(\delta))^k
 \cdot \omega_{\tau}((-1)^{m'+n'-1} \cdot \kappa_V^c \kappa_W^{-1}) \\
 & \quad \times |\kappa_V|^{k(s+\rho_P)} \cdot |\kappa_W|^{-k(s+\rho_Q)}
 \cdot L(s-s_0+\tfrac{1}{2}, \tau)^{-1}
 \cdot L(-s-s_0+\tfrac{1}{2}, (\tau^c)^\vee) \cdot \gamma(-s-s_0+\tfrac{1}{2}, (\tau^c)^\vee, \psi_E) \\
 & \quad \times \T_{-s}(\varphi, \M(\tilde{w}, s) \Phi_s).
\end{align*}
\end{prop}

\begin{proof}
We may assume that $\Re(s) \gg 0$.
Let $\check{v} \in \V_{\tau^\vee}$ and $\check{v}_0 \in \V_{\pi_0^\vee}$.
Noting that $\det J = 1$, we have by definition
\begin{align*}
 \langle \M(\tilde{w}', s) \T_s(\varphi, \Phi_s)(g), \check{v} \otimes \check{v}_0 \rangle
 & = \omega_\tau((-1)^{n'} \cdot \kappa_W^{-1}) \cdot 
 \chi_V((-1)^{n'} \cdot \kappa_W^{-1})^k
 \cdot |\kappa_W|^{-k(s+\rho_Q)} \cdot \omega_{\pi_0}(-1)^k \\
 & \quad \times \langle \M(w_Q, s) \T_s(\varphi, \Phi_s)(g), \tau^\vee(J) \check{v} \otimes \check{v}_0 \rangle, \\
 \langle \T_{-s}(\varphi, \M(\tilde{w}, s) \Phi_s)(g), \check{v} \otimes \check{v}_0 \rangle
 & = \omega_\tau((-1)^{m'} \cdot (\kappa_V^c)^{-1})
 \cdot \chi_W((-1)^{m'} \cdot \kappa_V)^k
 \cdot |\kappa_V|^{-k(s+\rho_P)} \cdot \omega_{\sigma_0}(-1)^k \\
 & \quad \times \langle \T_{-s}(\varphi, \M(w_P, s) \Phi_s)(g), \tau^\vee(J) \check{v} \otimes \check{v}_0 \rangle,
\end{align*}
where $\omega_{\pi_0}$ and $\omega_{\sigma_0}$ are the central characters of $\pi_0$ and $\sigma_0$ respectively.
Since $\sigma_0 =\Theta_{\psi, \chi_V, \chi_W, V_0, W_0}(\pi_0)$, we know that 
\[
 \omega_{\sigma_0} = \nu \cdot \omega_{\pi_0}, 
\]
where $\nu$ is the character of $\Ker(\N_{E/F})$ defined by 
\[
 \nu(x/x^c) = (\chi_V^{-n_0} \chi_W^{m_0})(x)
\]
for $x \in E^{\times}$.
In particular, we have
\[
 \omega_{\pi_0}(-1) \cdot \omega_{\sigma_0}(-1) = 
 (\chi_V^{-n_0} \chi_W^{m_0})(\delta)
 = (\chi_V^{-n} \chi_W^{m})(\delta) \cdot \chi_V(-1)^k \cdot \chi_W(-1)^k.
\]
Thus it suffices to show that
\begin{align*}
 & L(s-s_0+\tfrac{1}{2}, \tau) \cdot \M(w_Q, s) \T_s(\varphi, \Phi_s) \\
 & = (\chi_V(-\varepsilon) \cdot \gamma_V^{-1} \cdot \gamma_W)^k \cdot \omega_{\tau}(-1) \\
 & \quad \times L(-s-s_0+\tfrac{1}{2}, (\tau^c)^\vee)  \cdot \gamma(-s-s_0+\tfrac{1}{2}, (\tau^c)^\vee, \psi_E) \cdot  \T_{-s}(\varphi, \M(w_P, s) \Phi_s).
\end{align*}
We have
\begin{align*}
 & L(s - s_0 + \tfrac{1}{2}, \tau) \cdot 
 \langle \M(w_Q, s) \T_s(\varphi, \Phi_s)(g), \check{v} \otimes \check{v}_0 \rangle \\
 & = L(s - s_0 + \tfrac{1}{2}, \tau) \cdot 
 \int_{U_Q} \langle \T_s(\varphi, \Phi_s)(w_Q^{-1} u g), \check{v} \otimes \check{v}_0 \rangle \, du \\
 & = \int_{U_Q} \int_{U_P \U(V_0) \backslash \U(V)}
 \langle \T_{00} (\hat{f}(\varphi)(w_Q^{-1} u g h), 
 \langle \Phi_s(h), \check{v} \rangle ), \check{v}_0 \rangle \, dh \, du \\
 & = \int_{U_P \U(V_0) \backslash \U(V)} \int_{U_Q}
 \langle \T_{00} (\hat{f}(\varphi)(w_Q^{-1} u g h), 
 \langle \Phi_s(h), \check{v} \rangle ), \check{v}_0 \rangle \, du \, dh.
\end{align*}
In Lemma \ref{lem:conv}(i) below, we shall show that these integrals are absolutely convergent,
so that this manipulation is justified.
By Lemma \ref{lem:def2}, we have
\begin{align*}
 & L(-s-s_0+\tfrac{1}{2}, (\tau^c)^\vee) \cdot \gamma(-s-s_0+\tfrac{1}{2}, (\tau^c)^\vee, \psi_E) \cdot 
 \langle \T_{-s}(\varphi, \M(w_P, s) \Phi_s)(g), \check{v} \otimes \check{v}_0 \rangle \\
 & = \int_{U_P \U(V_0) \backslash \U(V)}
 \langle \T_{00} (f(\varphi)(gh), \langle \M(w_P, s) \Phi_s(h), \check{v} \rangle ), \check{v}_0 \rangle \, dh \\
 & = \int_{U_P \U(V_0) \backslash \U(V)} \int_{U_P}
 \langle \T_{00} (f(\varphi)(gh), \langle \Phi_s(w_P^{-1} u h), \check{v} \rangle ), \check{v}_0 \rangle \, du \, dh \\
 & = \int_{\U(V_0) \backslash \U(V)}
 \langle \T_{00} (f(\varphi)(gh), \langle \Phi_s(w_P^{-1} h), \check{v} \rangle ), \check{v}_0 \rangle \, dh \\
 & = \int_{\U(V_0) \backslash \U(V)}
 \langle \T_{00} (f(\varphi)(g w_P h), \langle \Phi_s(h), \check{v} \rangle ), \check{v}_0 \rangle \, dh \\
 & = \int_{U_P \U(V_0) \backslash \U(V)} \int_{U_P}
 \langle \T_{00} (f(\varphi)(g w_P u h), \langle \Phi_s(h), \check{v} \rangle ), \check{v}_0 \rangle \, du \, dh \\
 & = \int_{U_P \U(V_0) \backslash \U(V)} \int_{U_P}
 \langle \T_{00} (f(\varphi)(g w_P u m_P(-1_X) h), \langle \Phi_s(m_P(-1_X) h), \check{v} \rangle ), \check{v}_0 \rangle \, du \, dh \\
 & = \omega_{\tau}(-1) \cdot \chi_W(-1)^k \cdot \int_{U_P \U(V_0) \backslash \U(V)} \int_{U_P}
 \langle \T_{00} (f(\varphi)(g w_P u m_P(-1_X) h), \langle \Phi_s(h), \check{v} \rangle ), \check{v}_0 \rangle \, du \, dh.
\end{align*}
In Lemma \ref{lem:conv}(ii) below, we shall show that these integrals are absolutely convergent,
so that this manipulation is justified.
Thus it remains to show that
\begin{equation}
\label{E:key}
 \chi_V(-\varepsilon)^k \cdot \gamma_{V}^k \cdot \int_{U_Q} \hat{f}(\varphi)(w_Q^{-1} u)\, du
 = \chi_W(-1)^k \cdot \gamma_{W}^k \cdot \int_{U_P} f(\varphi)(w_P u m_P(-1_X)) \, du.
\end{equation}

We may assume that $\varphi = \varphi' \otimes \varphi_0$,
where $\varphi' \in \s(V \otimes Y^*)$ and $\varphi_0 \in \s_0$.
We have $w_Q^{-1} = m_Q(-\varepsilon 1_Y) \cdot w_Q$ and 
\begin{align*}
 & \hat{\ff}(\varphi)(w_Q^{-1}) \\
 & = \int_{X \otimes Y^*} (\omega(w_Q^{-1}) \varphi)
 \begin{pmatrix}
  x \\
  0 \\
  0
 \end{pmatrix}
 \psi(\varepsilon \langle x, e^* \rangle ) \, dx \\
 & = \chi_V(-\varepsilon)^k \cdot \int_{X \otimes Y^*} (\omega(w_Q) \varphi)
 \begin{pmatrix}
  -\varepsilon x \\
  0 \\
  0
 \end{pmatrix}
 \psi(\varepsilon \langle x, e^* \rangle ) \, dx \\
 & = \chi_V(-\varepsilon)^k \cdot \int_{X \otimes Y^*} (\omega(w_Q) \varphi)
 \begin{pmatrix}
  x \\
  0 \\
  0
 \end{pmatrix}
 \psi(- \langle x, e^* \rangle ) \, dx \\
 & = \chi_V(-\varepsilon)^k \cdot \gamma_V^{-k} \cdot \int_{X \otimes Y^*}
 \left( \int_{X \otimes Y^*} \int_{V_0 \otimes Y^*} \int_{X^* \otimes Y^*}
 \varphi
 \begin{pmatrix}
  y_1 \\
  y_2 \\
  y_3 
 \end{pmatrix} 
 \psi(-\langle I_Y y_3, x \rangle) \, dy_3 \, dy_2 \, dy_1 \right)  
 \psi(- \langle x, e^* \rangle ) \, dx \\
 & = \chi_V(-\varepsilon)^k \cdot \gamma_V^{-k} \cdot \int_{X \otimes Y^*}
 \left( \int_{X \otimes Y^*} \int_{V_0 \otimes Y^*} \int_{X^* \otimes Y^*}
 \varphi
 \begin{pmatrix}
  y_1 \\
  y_2 \\
  y_3 
 \end{pmatrix} 
 \psi(\langle x, I_Yy_3 \rangle) \, dy_3 \, dy_2 \, dy_1 \right)  
 \psi(- \langle x, e^* \rangle ) \, dx \\
 & = \chi_V(-\varepsilon)^k \cdot \gamma_V^{-k} \cdot \int_{X \otimes Y^*} \int_{V_0 \otimes Y^*} \varphi
 \begin{pmatrix}
  y_1 \\
  y_2 \\
  I_Y^{-1} e^*
 \end{pmatrix}
 \, dy_2 \, dy_1.
\end{align*}
Hence, noting that $I_X^{-1} e = I_Y^{-1} e^* = e^{**}$, we have
\begin{align*}
 & \chi_V(-\varepsilon)^k \cdot \gamma_V^k \cdot \int_{\Herm(Y^*,Y)} \hat{\ff}(\varphi)(w_Q^{-1} u_Q(c)) \, dc \\
 & = \int_{\Herm(Y^*,Y)} \left( \int_{X \otimes Y^*} \int_{V_0 \otimes Y^*} \varphi
 \begin{pmatrix}
  y_1 \\
  y_2 \\
  e^{**}
 \end{pmatrix}
 \psi (\langle c y_1, e^{**} \rangle + \tfrac{1}{2} \langle cy_2, y_2 \rangle)
 \, dy_2 \, dy_1 \right) dc.
\end{align*}
We change the variables
\begin{align*}
 y_1 & = x_1 e^{**} \in X \otimes Y^*, & x_1 & \in \Hom(X^*,X), \\
 y_2 & = x_2 e^{**} \in V_0 \otimes Y^*, & x_2 & \in \Hom(X^*,V_0).
\end{align*}
Then the above integral is equal to 
\begin{align*}
 & \int_{\Herm(Y^*,Y)} \left( \int_{\Hom(X^*,X)} \int_{\Hom(X^*,V_0)} \varphi
 \begin{pmatrix}
  x_1 e^{**} \\
  x_2 e^{**} \\
  e^{**}
 \end{pmatrix}
 \psi (\langle c x_1 e^{**}, e^{**} \rangle + \tfrac{1}{2} \langle c x_2 e^{**}, x_2 e^{**} \rangle)
 \, dx_2 \, dx_1 \right) dc \\
 & = \int_{\Herm(Y^*,Y)} \left( \int_{\Hom(X^*,X)} \int_{\Hom(X^*,V_0)} \varphi
 \begin{pmatrix}
  x_1 e^{**} \\
  x_2 e^{**} \\
  e^{**}
 \end{pmatrix}
 \psi (-\langle x_1 e^{**}, c e^{**} \rangle - \tfrac{1}{2} \langle x_2^* x_2 e^{**}, c e^{**} \rangle)
 \, dx_2 \, dx_1 \right) dc \\
 & = \int_{\Herm(Y^*,Y)} \left( \int_{\Hom(X^*,X)} \int_{\Hom(X^*,V_0)} \varphi
 \begin{pmatrix}
  (x_1 - \tfrac{1}{2} x_2^* x_2) e^{**} \\
  x_2 e^{**} \\
  e^{**}
 \end{pmatrix}
 \psi (-\langle x_1 e^{**}, c e^{**} \rangle) \, dx_2 \, dx_1 \right) dc.
\end{align*}
By Lemma \ref{L:fi}, this integral is equal to
\[
 \int_{\Herm(X^*,X)} \int_{\Hom(X^*,V_0)} \varphi
 \begin{pmatrix}
  (c - \tfrac{1}{2} x_2^* x_2) e^{**} \\
  x_2 e^{**} \\
  e^{**}
 \end{pmatrix}
 dx_2 \, dc.
\]
Hence the left-hand side of \eqref{E:key} is equal to 
\begin{align*}
 & \chi_V(-\varepsilon)^k \cdot \gamma_V^k \cdot \int_{\Hom(W_0,Y)} \int_{\Herm(Y^*,Y)} 
 \hat{f}(\varphi)(w_Q^{-1} u_Q(c) u_Q(b)) \, dc \, db \\
 & = \int_{\Hom(W_0,Y)} \int_{\Herm(X^*,X)} \int_{\Hom(X^*,V_0)} \ev \left(
 (\omega(u_Q(b)) \varphi)
 \begin{pmatrix}
 (c - \tfrac{1}{2} x_2^* x_2) e^{**} \\
  x_2 e^{**} \\
  e^{**}
 \end{pmatrix} \right)
 dx_2 \, dc \, db \\[5pt]
 & = \int_{\Hom(W_0,Y)} \int_{\Herm(X^*,X)} \int_{\Hom(X^*,V_0)} 
 \varphi'
 \begin{pmatrix}
 (c - \tfrac{1}{2} x_2^* x_2) e^{**} \\
  x_2 e^{**} \\
  e^{**}
 \end{pmatrix} \\
 & \quad \times \psi( \tfrac{1}{2} \langle b^* e^{**}, b^* (c - \tfrac{1}{2} x_2^* x_2) e^{**} \rangle)
 \rho_{00}((b^* x_2 e^{**}, 0)) \varphi_0(b^* e^{**}) \, dx_2 \, dc \, db \\[5pt]
 & = \int_{\Hom(W_0,Y)} \int_{\Herm(X^*,X)} \int_{\Hom(X^*,V_0)} 
 \varphi'
 \begin{pmatrix}
 (c - \tfrac{1}{2} x_2^* x_2) e^{**} \\
  x_2 e^{**} \\
  e^{**}
 \end{pmatrix} \\
 & \quad \times \psi(- \tfrac{1}{2} \langle c b^* e^{**}, b^* e^{**} \rangle)
 \rho_{00}((x_2 b^* e^{**}, 0)) \varphi_0(b^* e^{**}) \, dx_2 \, dc \, db.
\end{align*}
Note that $\langle b^* e^{**}, b^* x_2^* x_2 e^{**} \rangle = \langle b b^* e^{**}, x_2^* x_2 e^{**} \rangle = 0$.

On the other hand, the right-hand side of \eqref{E:key} is equal to the product of $\chi_W(-1)^k \cdot \gamma_{W}^k$ and 
\begin{align*}
 & \int_{\Hom(V_0, X)} \int_{\Herm(X^*,X)}
 f(\varphi)(w_P u_P(c') u_P(b') m_P(-1_X)) \, dc' \, db' \\
 & = \int_{\Hom(V_0, X)} \int_{\Herm(X^*,X)} \varphi'
 \begin{pmatrix}
  -(c' + \frac{1}{2}b' b'^*) e^{**} \\
  -b'^* e^{**} \\
  e^{**}
 \end{pmatrix}
 \ev(\omega_0(w_P u_P(c') u_P(b') m_P(-1_X)) \varphi_0) \, dc' \, db'.
\end{align*}
We have
\begin{align*}
 &  \ev(\omega_0(w_P u_P(c') u_P(b') m_P(-1_X)) \varphi_0) \\
 & = \gamma_{W}^{-k} \cdot \int_{X^* \otimes W_0} (\omega_0(u_P(c') u_P(b') m_P(-1_X)) \varphi_0)(y) \, dy \\
 & = \gamma_{W}^{-k} \cdot \int_{X^* \otimes W_0} \psi(\tfrac{1}{2} \langle c'y, y \rangle) 
 (\omega_0(u_P(b') m_P(-1_X)) \varphi_0)(y) \, dy \\
 & = \gamma_{W}^{-k} \cdot \int_{X^* \otimes W_0} \psi(\tfrac{1}{2} \langle c'y, y \rangle)
 \rho_{00}((b'^*y,0)) (\omega_0(m_P(-1_X)) \varphi_0)(y) \, dy \\
 & = \chi_W(-1)^k \cdot \gamma_{W}^{-k} \cdot \int_{X^* \otimes W_0}
 \psi(\tfrac{1}{2} \langle c'y, y \rangle)
 \rho_{00}((b'^*y,0)) \varphi_0(-y) \, dy.
\end{align*}
Changing the variables
\begin{align*}
 b' & = -x_2^* \in \Hom(V_0, X), & x_2 & \in \Hom(X^*, V_0), \\
 c' & = -c \in \Herm(X^*,X), & c & \in \Herm(X^*,X), \\
 y & = - b^* e^{**} \in X^* \otimes W_0, & b & \in \Hom(W_0, Y),
\end{align*}
we see that the equality \eqref{E:key} holds.
This completes the proof.
\end{proof}

Let $\phi_{\tau}$, $\phi_0$, and $\phi_0'$ be the $L$-parameters of $\tau$, $\pi_0$, and $\sigma_0$ respectively.
As a consequence of Proposition \ref{P:io}, we deduce:

\begin{cor} \label{C:io}
For $\varphi \in \s$ and $\Phi_s \in \Ind^{\U(V)}_P(\tau^c_s \chi_W^c \otimes \sigma_0^\vee)$, we have
\[
 \R(w', \tau_s \chi_V \otimes \pi_0) \T_s(\varphi, \Phi_s)  
 = \alpha \cdot \beta(s) \cdot \T_{-s}(\varphi, \R(w, \tau_s^c \chi_W^c \otimes \sigma_0^\vee) \Phi_s),
\]
where
\begin{align*}
 \alpha
 & = ( \gamma_V^{-1} \cdot \gamma_W
 \cdot \chi_V((-1)^{n'} \cdot \varepsilon \cdot \kappa_W^{-1})
 \cdot \chi_W((-1)^{m'-1} \cdot \kappa_V^{-1})
 \cdot (\chi_V^{-n} \chi_W^m)(\delta))^k \\
 & \quad \times \omega_{\tau}((-1)^{m'+n'-1} \cdot \kappa_V^c \kappa_W^{-1}) \cdot \lambda(w, \psi) \cdot \lambda(w',\psi)^{-1}
\end{align*}
and
\begin{align*}
 \beta(s) & = L(s-s_0+\tfrac{1}{2}, \phi_\tau)^{-1}
 \cdot L(-s-s_0+\tfrac{1}{2}, (\phi_\tau^c)^\vee) \cdot \gamma(-s-s_0+\tfrac{1}{2}, (\phi^c_\tau)^\vee, \psi_E) \\
 & \quad \times |\kappa_V \kappa_W^{-1}|^{ks}
 \cdot \gamma(s, \phi_{\tau}^c \otimes \phi_0' \otimes \chi_W^c, \psi_E)^{-1} \cdot \gamma(s, \phi_{\tau} \otimes \phi_0^\vee \otimes \chi_V, \psi_E).
\end{align*}
\end{cor}

\begin{proof}
The corollary immediately follows from Proposition \ref{P:io} and the following facts:
\begin{itemize}
\item $\gamma(s, \As^+ \circ \phi_{\tau^c}, \psi) = \gamma(s, \As^+ \circ \phi_\tau, \psi)$;
\item for any conjugate self-dual character $\chi$ of $E^{\times}$, 
\[
 \gamma(s, \As^+ \circ \phi_{\tau \chi}, \psi) =
 \begin{cases}
  \gamma(s, \As^+ \circ \phi_\tau, \psi) & \text{if $\chi|_{F^{\times}} = \1_{F^{\times}}$;} \\
  \gamma(s, \As^- \circ \phi_\tau, \psi) & \text{if $\chi|_{F^{\times}} = \omega_{E/F}$.}
 \end{cases}
\]
\end{itemize}
\end{proof}

\subsection{Convergence of integrals}
\label{s:conv}

To finish the proof of Proposition \ref{P:io}, it remains to show the following convergence of the integrals.

\begin{lem}
\label{lem:conv}
Let $\varphi \in \s$, $\Phi_s \in \Ind^{\U(V)}_P(\tau^c_s \chi_W^c \otimes \sigma_0^\vee)$, $\check{v} \in \V_{\tau^\vee}$, and $\check{v}_0 \in \V_{\pi_0^\vee}$.
Assume that $\Re(s) \gg 0$.
\begin{enumerate}
\item The integral
\begin{equation}
\label{eq:conv1} 
 \int_{U_Q} \int_{U_P \U(V_0) \backslash \U(V)}
 \langle \T_{00} (\hat{f}(\varphi)(w_Q^{-1} u h), 
 \langle \Phi_s(h), \check{v} \rangle ), \check{v}_0 \rangle \, dh \, du 
\end{equation}
is absolutely convergent.
\item The integral
\begin{equation}
\label{eq:conv2} 
 \int_{U_P \U(V_0) \backslash \U(V)} \int_{U_P}
 \langle \T_{00} (f(\varphi)(h), \langle \Phi_s(w_P^{-1} u h), \check{v} \rangle ), \check{v}_0 \rangle \, du \, dh
\end{equation}
is absolutely convergent.
\end{enumerate}
\end{lem}

\begin{proof}
Put $t = \Re(s) \gg 0$.
Fix a special maximal compact subgroup $K$ of $\U(V)$.
We may assume that
\begin{itemize}
\item $\varphi = \varphi' \otimes \varphi_0$ for some $\varphi' \in \s(V \otimes Y^*)$ and $\varphi_0 \in \s_0$;
\item $\Phi_s|_K$ is independent of $s$;
\item $\Phi_s$ is $K_0$-fixed for some open compact subgroup $K_0$ of $K$;
\item $\supp \Phi_s = P k_0 K_0$ for some $k_0 \in K$;
\item $\Phi_s(k_0) = v \otimes v_0$ for some $v \in \V_{\tau}$ and $v_0 \in \V_{\sigma_0^\vee}$.
\end{itemize}
In particular, there exist maps $v:K \rightarrow \V_{\tau}$ and $v_0 : K \rightarrow \V_{\sigma_0^\vee}$ such that
\[
 \Phi_s(k) = v(k) \otimes v_0(k)
\]
for all $k \in K$.

Recall that $\tau$, $\pi_0$, and $\sigma_0$ are tempered and hence unitarizable.
We can choose invariant Hilbert space norms $\| \cdot \|$ on $\V_{\tau}$ and $\V_{\tau^\vee}$ so that
\[
 | \langle v, \check{v} \rangle | \le \| v \| \| \check{v} \|
\]
for all $v \in \V_{\tau}$ and $\check{v} \in \V_{\tau^\vee}$.
Similarly, we choose invariant Hilbert space norms on $\V_{\pi_0}$, $\V_{\sigma_0}$, and so on.
We may regard $\T_{00}$ as a $\U(V_0) \times \U(W_0)$-equivariant map $\T_{00}: \s_{00}  \rightarrow \V_{\sigma_0} \otimes \V_{\pi_0}$, i.e. 
\[
 \langle \T_{00}(\varphi_{00}), v_0 \otimes \check{v}_0 \rangle 
 = \langle \T_{00}(\varphi_{00}, v_0), \check{v}_0 \rangle
\]
for $\varphi_{00} \in \s_{00}$, $v_0 \in \V_{\sigma_0^\vee}$, and $\check{v}_0 \in \V_{\pi_0^\vee}$.
Then we have $\| \T_{00}(\omega_{00}(g_0h_0)\varphi_{00}) \| = \| \T_{00}(\varphi_{00}) \|$ for $g_0 \in \U(W_0)$ and $h_0 \in \U(V_0)$, and 
\[
 | \langle \T_{00}(\varphi_{00}, v_0), \check{v}_0 \rangle | \le
 \| \T_{00}(\varphi_{00}) \| \| v_0 \| \| \check{v}_0 \|.
\]
Fix $\check{v} \in \V_{\tau^\vee}$ and $\check{v}_0 \in \V_{\pi_0^\vee}$, and put
\[
 C = \| \check{v} \| \| \check{v}_0 \| \max_{k \in K} \| v(k) \| \| v_0(k) \|.
\]

Let $\Psi_t$ be the $K$-fixed element in $\Ind^{\U(V)}_P(|\det|^t \otimes \1_{\U(V_0)})$ such that $\Psi_t(1) = 1$.
Let $\ell$ denote the representation of $\U(V)$ on $\s(V \otimes Y^*)$ defined by $(\ell(h) \varphi')(x) = \varphi'(h^{-1} x)$.
Recall that $\ev : \s_0 \rightarrow \s_{00}$ is the evaluation at $0$.

First, we prove the absolute convergence of \eqref{eq:conv1}.
We have
\[
 \hat{f}(\varphi)(h) = \phi(\ell(h) \varphi')(e^*) \cdot \ev(\omega_0(h) \varphi_0),
\]
where $\phi:\s(V \otimes Y^*) \rightarrow \s(X^* \otimes Y)$ is defined by
\[
 \phi(\varphi')(y)
 = \int_{X \otimes Y^*} \varphi'
 \begin{pmatrix}
  x \\
  0 \\
  0
 \end{pmatrix}
 \psi(\varepsilon \langle x, y \rangle ) \, dx.
\]
Put 
\begin{align*}
 \hat{\xi}_s(g,h) & =
 \langle \T_{00}(\hat{f}(\varphi)(gh), \langle \Phi_s(h), \check{v} \rangle), \check{v}_0 \rangle \\
 & = \chi_W^c(\det a) |\det a|^{s+\rho_P} \langle \tau^c(a) v(k), \check{v} \rangle
 \cdot \langle \T_{00}(\hat{f}(\varphi)(gh), \sigma_0^\vee(h_0) v_0(k)), \check{v}_0 \rangle
\end{align*}
for $g \in \U(W)$, $h = u m_P(a) h_0 k \in \U(V)$, $u \in U_P$, $a \in \GL(X)$, $h_0 \in \U(V_0)$, and $k \in K$.
Then we have
\begin{align*}
 |\hat{\xi}_s(g,h)| & \le 
 |\det a|^{t+\rho_P} \| v(k) \| \| \check{v}\| \cdot 
 \| \T_{00}(\hat{f}(\varphi)(gh)) \| \| v_0(k) \| \| \check{v}_0 \| \\
 & \le C \cdot \Psi_t(h) \cdot \| \T_{00}(\hat{f}(\varphi)(gh)) \|
\end{align*}
and 
\begin{align*}
 \| \T_{00}(\hat{f}(\varphi)(h)) \| & = \| \T_{00}(\hat{f}(\varphi)(m_P(a) k)) \| \\
 & = |\det a|^{k+n_0/2} |\phi(\ell(k) \varphi')(a^* e^*)| \cdot \| \T_{00}(\ev(\omega_0(k) \varphi_0)) \|.
\end{align*}
Hence we have
\begin{align*}
 \int_{U_P \U(V_0) \backslash \U(V)} |\hat{\xi}_s(g,h)| \, dh 
 & \le C \cdot \int_{U_P \U(V_0) \backslash \U(V)} \Psi_t(h) \| \T_{00}(\hat{f}(\varphi)(gh)) \| \, dh \\
 & = C \cdot \int_{\GL(X)} \int_K |\det a|^{t-\rho_P} \| \T_{00}(\hat{f}(\varphi)(g m_P(a)k)) \| \, dk \, da \\
 & < \infty
\end{align*}
since the last integral is the zeta integral of Godement--Jacquet associated to the trivial representation of $\GL(X)$.
Put
\[
 \hat{\Xi}_t(g) = 
 C \cdot \int_{U_P \U(V_0) \backslash \U(V)} \Psi_t(h) \| \T_{00}(\hat{f}(\varphi)(gh)) \| \, dh.
\]
Then we have
\[
 \hat{\Xi}_t(u m_Q(a) g_0 g) = |\det a|^{t+\rho_Q} \hat{\Xi}_t(g)
\]
for $u \in U_Q$, $a \in \GL(Y)$, $g_0 \in \U(W_0)$, and $g \in \U(W)$, 
i.e.~$\hat{\Xi}_t \in \Ind^{\U(W)}_Q(|\det|^t \otimes \1_{\U(W_0)})$.
Hence we have
\[
 \int_{U_Q} \int_{U_P \U(V_0) \backslash \U(V)} |\hat{\xi}_s(w_Q^{-1} u, h)| \, dh \, du 
 \le \int_{U_Q} \hat{\Xi}_t(w_Q^{-1} u) \, du < \infty.
\]

Next, we prove the absolute convergence of \eqref{eq:conv2}.
We have
\[
 f(\varphi)(h) = \hat{\phi}(\ell(h) \varphi')(e) \cdot \ev(\omega_0(h) \varphi_0),
\]
where $\hat{\phi}:\s(V \otimes Y^*) \rightarrow \s(X \otimes Y^*)$ is defined by
\[
 \hat{\phi}(\varphi')(x) = \varphi'
 \begin{pmatrix}
  x \\
  0 \\
  0
 \end{pmatrix}.
\]
Put
\begin{align*}
 \xi_s(h,h')
 & = \langle \T_{00} (f(\varphi)(h'), \langle \Phi_s(h), \check{v} \rangle ), \check{v}_0 \rangle \\
 & = \chi_W^c(\det a) |\det a|^{s+\rho_P} \langle \tau^c(a) v(k), \check{v} \rangle
 \cdot \langle \T_{00}(f(\varphi)(h'), \sigma_0^\vee(h_0) v_0(k)), \check{v}_0 \rangle
\end{align*}
for $h = u m_P(a) h_0 k, h' \in \U(V)$, $u \in U_P$, $a \in \GL(X)$, $h_0 \in \U(V_0)$, and $k \in K$.
Then we have
\begin{align*}
 |\xi_s(h,h')|
 & \le |\det a|^{t+\rho_P}  \| v(k) \| \| \check{v} \|
 \cdot \| \T_{00}(f(\varphi)(h')) \| \| v_0(k) \| \| \check{v}_0 \| \\
 & \le C \cdot \Psi_t(h) \cdot \| \T_{00}(f(\varphi)(h')) \|.
\end{align*}
Hence we have
\[
 \int_{U_P} |\xi_s(w_P^{-1} u h, h')| \, du
 \le C \cdot \| \T_{00}(f(\varphi)(h')) \| \cdot \int_{U_P} \Psi_t(w_P^{-1} u h) \, du < \infty.
\]
Put
\[
 \Xi_t(h) = C \cdot \| \T_{00}(f(\varphi)(h)) \| \cdot \M(w_P,t) \Psi_t(h),
\]
where 
\[
 \M(w_P,t) \Psi_t(h) =  \int_{U_P} \Psi_t(w_P^{-1} u h) \, du.
\]
Then we have
\[
 \Xi_t(u m_P(a) h_0h) = C \cdot |\det a|^{-t+\rho_P+n_0/2} |\hat{\phi}(\ell(h) \varphi')(a^{-1} e)| \cdot \| \T_{00}(\ev(\omega_0(h) \varphi_0))\| \cdot \M(w_P,t) \Psi_t(h) 
\]
for $u \in U_P$, $a \in \GL(X)$, $h_0 \in \U(V_0)$, and $h \in \U(V)$.
Hence, putting 
\[
 C' = C \cdot \max_{k \in K} \| \ev(\omega_0(k) \varphi_0)\|
 \cdot \M(w_P,t) \Psi_t(1),
\]
we have
\begin{align*}
 \int_{U_P \U(V_0) \backslash \U(V)} \int_{U_P} |\xi_s(w_P^{-1} u h, h)| \, du  \, dh
 & \le \int_{U_P \U(V_0) \backslash \U(V)} \Xi_t(h) \, dh \\
 & \le C' \cdot \int_{\GL(X)} \int_K
 |\det a|^{-t-\rho_P + n_0/2}
 |\hat{\phi}(\ell(k) \varphi')(a^{-1} e)| \, dk \, da \\
 & < \infty
\end{align*}
since the last integral is the zeta integral of Godement--Jacquet associated to the trivial representation of $\GL(X)$.
\end{proof}

\subsection{Completion of the proof}

Now assume that $\varepsilon = +1$ and $m = n+1$.
Let $\phi$ be a tempered but non-square-integrable $L$-parameter for $\U(W_n^\pm)$.
Since $\phi$ is not square-integrable, we can write
\[
 \phi = (\phi_\tau \otimes \chi_V) \oplus \phi_0 \oplus
 ((\phi_\tau \otimes \chi_V)^c)^\vee
\]
for some irreducible (unitary) square-integrable representation $\tau$ of $\GL_k(E)$ and tempered $L$-parameter $\phi_0$ for $\U(W_{n_0}^\pm)$,
where $k$ is a positive integer and $n_0 = n-2k$.
Fix $\epsilon' = \pm 1$, and set $W = W_n^{\epsilon'}$ and $W_0 = W_{n_0}^{\epsilon'}$.
Let $\pi = \pi(\eta) \in \Pi_{\phi}$ be an irreducible tempered representation of $\U(W)$ with associated character $\eta \in \Irr(S_{\phi})$.
Then $\pi$ is an irreducible constituent of $\Ind^{\U(W)}_Q(\tau \chi_V \otimes \pi_0)$ for some irreducible tempered representation $\pi_0 = \pi_0(\eta_0) \in \Pi_{\phi_0}$ of $\U(W_0)$ with associated character $\eta_0 \in \Irr(S_{\phi_0})$ such that
\[
 \eta|_{S_{\phi_0}} = \eta_0. 
\]

Fix $\epsilon = \pm 1$, and set $V = V_{n+1}^{\epsilon}$ and $V_0 = V_{n_0+1}^{\epsilon}$.
Suppose that $\sigma := \Theta_{\psi, V, W}(\pi) \ne 0$.
By the argument as in \cite[pp.~1674--1676]{gs}, we see that $\sigma_0 := \Theta_{\psi,V_0, W_0}(\pi_0) \ne 0$ and $\sigma$ is an irreducible constituent of $\Ind^{\U(V)}_P(\tau \chi_W \otimes \sigma_0)$.
This implies that $\sigma^\vee$ is an irreducible constituent of $\Ind^{\U(V)}_P(\tau^c \chi_W^c \otimes \sigma_0^\vee)$.
By Theorem \ref{T:al-equal}, $\sigma  = \sigma(\eta') \in \Pi_{\phi'}$ and $\sigma_0 = \sigma_0(\eta_0') \in \Pi_{\phi_0'}$ are irreducible tempered representations of $\U(V)$ and $\U(V_0)$ respectively, with $L$-parameters
\[
 \phi' = (\phi \otimes \chi_V^{-1} \chi_W) \oplus \chi_W \quad \text{and} \quad
 \phi'_0 = (\phi_0 \otimes \chi_V^{-1} \chi_W) \oplus \chi_W, 
\]
and associated characters $\eta' \in \Irr(S_{\phi'})$ and $\eta_0' \in \Irr(S_{\phi_0'})$ such that
\[
 \eta'|_{S_{\phi_0'}} = \eta_0'.
\]
We need to show that $\eta'|_{S_{\phi}} = \eta$.

Consider a commutative diagram
\[
 \xymatrix{
  S_{\phi} \ar@{->}[r] & S_{\phi'} \\
  S_{\phi_0} \ar@{->}[u] \ar@{->}[r] & S_{\phi_0'} \ar@{->}[u]}
\]
of natural embeddings.
Since $n_0 < n$, we know that (P2)$_{n_0}$ holds by assumption, so that 
\[
 \eta_0'|_{S_{\phi_0}} = \eta_0.
\]
Hence, we conclude that 
\[
 \eta'|_{S_{\phi_0}} = (\eta'|_{S_{\phi_0'}})|_{S_{\phi_0}}
 = \eta'_0|_{S_{\phi_0}} = \eta_0 = \eta|_{S_{\phi_0}}.
\]
In particular, if $S_{\phi_0} = S_{\phi}$, then $\eta'|_{S_\phi} = \eta$ as desired.

Finally, we assume that $S_{\phi_0} \ne S_{\phi}$, which is the case if and only if $\phi_\tau$ is conjugate orthogonal and $\phi_\tau \otimes \chi_V$ is not contained in $\phi_0$.
Then the component group $S_{\phi}$ is of the form
\[
 S_{\phi} = S_{\phi_0} \times (\ZZ /2 \ZZ) a_1, 
\]
where the extra copy of $\ZZ / 2 \ZZ$ arises from the summand $\phi_\tau \otimes \chi_V$ in $\phi$.
Since we already know that $\eta'|_{S_{\phi_0}} = \eta|_{S_{\phi_0}}$, it suffices to show that $\eta'(a_1) = \eta(a_1)$.
To see this, we recall the $\U(V) \times \U(W)$-equivariant map
\[
 \T_0 : \omega \otimes \Ind^{\U(V)}_P(\tau^c \chi_W^c \otimes \sigma_0^\vee)
 \longrightarrow \Ind^{\U(W)}_Q(\tau \chi_V \otimes \pi_0).
\]
Since $\T_0(\varphi, \Phi) \in \pi$ for $\varphi \in \s$ and $\Phi \in \sigma^\vee$, it follows by \eqref{eq:lir}, Lemma \ref{lem:nonzero}, and Corollary \ref{C:io} that
\[
 \epsilon(W)^k \cdot \eta(a_1)
 = \alpha \cdot \beta(0) \cdot \epsilon(V)^k \cdot \eta_{\sigma^\vee}(a_1),
\]
where $\alpha$ and $\beta(s)$ are as in Corollary \ref{C:io}, 
and $\eta_{\sigma^\vee} \in \Irr(S_{(\phi')^\vee})$ is the irreducible character associated to $\sigma^\vee$.
But we know that
\[
 \eta_{\sigma^\vee}(a_1) = \eta'(a_1)
 \times 
 \begin{cases}
  1 & \text{if $n$ is even;} \\
  \omega_{E/F}(-1)^k & \text{if $n$ is odd.}
 \end{cases}
\]
Thus it remains to show that
\[
 \epsilon(V)^k \cdot \epsilon(W)^k \cdot \alpha \cdot \beta(0) = 
 \begin{cases}
  1 & \text{if $n$ is even;} \\
  \omega_{E/F}(-1)^k & \text{if $n$ is odd.}
 \end{cases}
\]

First, we compute $\epsilon(V)^k \cdot \epsilon(W)^k \cdot \alpha$ when $n$ is even.
In this case, we see that $\gamma_V = \epsilon(V) \cdot \lambda(E/F, \psi)$ and $\gamma_W = \epsilon(W) \cdot \chi_W(\delta)^{-1}$.
Hence we have
\begin{align*}
 \epsilon(V)^k \cdot \epsilon(W)^k \cdot \alpha
 & = ( \lambda(E/F, \psi)^{-1} \cdot \chi_W(\delta)^{-1}
 \cdot \chi_V(-1)^{n'} \cdot \chi_W(-1)^{n'}
 \cdot (\chi_V^{-n} \chi_W^{n+1})(\delta))^k \\
 & \quad \times \omega_{\tau}(-1)^{2n'} \cdot
 \lambda(E/F, \psi)^{(k+1)k/2} \cdot \lambda(E/F, \psi)^{-(k-1)k/2} \\
 & = 1.
\end{align*}
Next, we compute $\epsilon(V)^k \cdot \epsilon(W)^k \cdot \alpha$ when $n$ is odd.
In this case, we see that $\gamma_V = \epsilon(V)$ and $\gamma_W = \epsilon(W) \cdot \chi_W(\delta)^{-1} \cdot \lambda(E/F, \psi)$.
Hence we have
\begin{align*}
 \epsilon(V)^k \cdot \epsilon(W)^k \cdot \alpha
 & = ( \chi_W(\delta)^{-1} \cdot \lambda(E/F, \psi)
 \cdot \chi_V((-1)^{n'-1} \cdot \delta^{-1})
 \cdot \chi_W((-1)^{n'-1} \cdot \delta^{-1})
 \cdot (\chi_V^{-n} \chi_W^{n+1})(\delta))^k \\
 & \quad \times \omega_{\tau}(-1)^{2n'-1} \cdot \lambda(E/F, \psi)^{(k-1)k/2} \cdot \lambda(E/F, \psi)^{-(k+1)k/2} \\
 & = \omega_{E/F}(-1)^k \cdot \omega_\tau(-1) \\
 & = \omega_{E/F}(-1)^k,
\end{align*}
where the last equality follows because $\omega_\tau|_{F^{\times}} = \1_{F^{\times}}$.

Finally, we compute $\beta(0)$.
Noting that $s_0 = \frac{1}{2}$, $(\phi_\tau^c)^\vee = \phi_\tau$, and $\phi'_0 = (\phi_0 \otimes \chi_V^{-1} \chi_W) \oplus \chi_W$, we see that
\begin{align*}
 \beta(s) & = L(s, \phi_\tau)^{-1}
 \cdot L(-s, \phi_\tau) \cdot \gamma(-s, \phi_\tau, \psi_E) \cdot \gamma(s, \phi_{\tau}, \psi_E)^{-1} \\
 & = \frac{\epsilon(-s, \phi_\tau, \psi_E)}{\epsilon(s, \phi_\tau, \psi_E)}
 \cdot \frac{L(1+s, \phi_\tau^\vee)}{L(1-s, \phi_\tau^\vee)}.
\end{align*}
Since $\tau$ is square-integrable, $L(s, \phi_\tau^\vee)$ is holomorphic and nonzero at $s=1$, and hence
\[
 \beta(0) = 1.
\]
Thus, we have shown the desired formula for $\epsilon(V)^k \cdot \epsilon(W)^k \cdot \alpha \cdot \beta(0)$ and completed the proof of Theorem \ref{T:ind}.

\begin{rem}
Using Theorem \ref{T:equal} (instead of Theorem \ref{T:al-equal}) and the above argument, one can also prove the analog of Theorem \ref{T:ind} for (P1).
Indeed, this can be reduced to the computation of $\epsilon(V)^k \cdot \epsilon(W)^k \cdot \alpha \cdot \beta(0)$
when $\varepsilon = +1$, $m=n$, and $\phi_\tau$ is conjugate symplectic,
in which case one sees that 
\[
 \epsilon(V)^k \cdot \epsilon(W)^k \cdot \alpha = 
 \begin{cases}
  \omega_{\tau}(\delta) \cdot \omega_{E/F}(-1)^k & \text{if $n$ is even;} \\
  \omega_{\tau}(\delta) & \text{if $n$ is odd,}
 \end{cases}
\]
and 
\[
 \beta(0) = \epsilon(\tfrac{1}{2}, \phi_\tau, \psi_E)
\]
as desired.
\end{rem}

\section{\textbf{Generic case}}
\label{s:generic}

So far, we have verified the Fourier--Jacobi case (FJ) of the Gross--Prasad conjecture for tempered $L$-parameters for $\U(W_n) \times \U(W_n)$.
As in the proof of \cite[Theorem 19.1]{ggp1}, this implies (FJ) for tempered $L$-parameters for $\U(W_n) \times \U(W_{n+2k})$ with $k>0$.
In this section, we extend (FJ) to the case of generic $L$-parameters.

\subsection{Generic $L$-parameters}

Let $V$ be an $n$-dimensional $\varepsilon$-Hermitian space.
Recall that an $L$-parameter $\phi$ for $\U(V)$ is generic if, by definition, its associated $L$-packet $\Pi_{\phi}$ contains generic representations (i.e.~those which possess some Whittaker models).
In Proposition \ref{p:gp-rallis} below, we shall show that $\phi$ is generic if and only if its adjoint $L$-factor $L(s, \Ad \circ \phi) =  L(s, \As^{(-1)^n} \circ \phi)$ is holomorphic at $s=1$.

Let $\phi$ be an $L$-parameter for $\U(V)$, so that we may write
\[
 \phi = \rho \oplus \phi_0 \oplus (\rho^c)^{\vee} \quad \text{with} \quad
 \rho = \bigoplus_{i=1}^r \rho_i |\cdot|^{s_i},
\]
where 
\begin{itemize}
 \item $\rho_i$ is a $k_i$-dimensional tempered representation of $\WD_E$,
 \item $s_i$ is a real number such that $s_1 > \cdots > s_r > 0$,
 \item $\phi_0$ is a tempered $L$-parameter for $\U(V_0)$, where $V_0$ is the $\varepsilon$-Hermitian space of dimension $n-2(k_1+\cdots+k_r)$ such that $\epsilon(V_0) = \epsilon(V)$.
\end{itemize}
As mentioned in \S \ref{SS:llc}, by the construction of the local Langlands correspondence, the representations in the Vogan $L$-packet $\Pi_{\phi}$ are given by
\begin{equation}
\label{E:LQ}
 \text{the unique irreducible quotient of the standard module $\Ind \bigl( \bigl( \bigotimes_{i=1}^r \tau_i |\cdot|^{s_i} \bigr) \otimes \pi_0 \bigr)$}
\end{equation}
for $\pi_0 \in \Pi_{\phi_0}$, where $\Ind$ is the appropriate parabolic induction and $\tau_i$ is the irreducible tempered representation of $\GL_{k_i}(E)$ associated to $\rho_i$.
If $\phi$ is generic, then we have the following result of Heiermann \cite{hei}, which extends a result of M{\oe}glin--Waldspurger \cite[Corollaire 2.14]{mw} for special orthogonal groups and symplectic groups.

\begin{prop}
\label{prop:heiermann}
Let $\phi$ be a generic $L$-parameter for $\U(V)$.
Then the standard modules in \eqref{E:LQ} are all irreducible,
so that the $L$-packet $\Pi_{\phi}$ consists of standard modules.
\end{prop}

\subsection{Local theta correspondence}

Proposition \ref{prop:heiermann} has consequences for the local theta correspondence.
Let $V$ be an $m$-dimensional Hermitian space and $W$ an $n$-dimensional skew-Hermitian space.
Consider the theta correspondence for $\U(V) \times \U(W)$ relative to a pair of characters $(\chi_V, \chi_W)$.
Let $\phi$ be an $L$-parameter for $\U(W)$ and $\pi$ a representation of $\U(W)$ in $\Pi_\phi$.
If $m = n$, then by Theorem \ref{T:equal}, we have $\theta_{\psi,V,W}(\pi) \in \Pi_{\theta(\phi)}$ (if nonzero) with $\theta(\phi)  = \phi  \otimes \chi_V^{-1} \chi_W$, so that $L(s, \Ad \circ \theta(\phi)) =  L(s, \Ad \circ \phi)$.
Thus $\theta(\phi)$ is generic if and only if $\phi$ is.
On the other hand, if $m=n+1$, then by Theorem \ref{T:al-equal}, we have $\theta_{\psi,V,W}(\pi) \in \Pi_{\theta(\phi)}$ (if nonzero) with
\[  \theta(\phi)  = (\phi \otimes \chi_V^{-1} \chi_W) \oplus \chi_W. \]
In this case, it is possible that $\theta(\phi)$ is nongeneric even if $\phi$ is. 
More precisely, since
\[  L(s, \Ad \circ \theta(\phi)) =  L(s, \Ad \circ \phi)  \cdot  L(s,  \phi \otimes \chi_V^{-1}) \cdot L(s, \omega_{E/F}), \] 
$\theta(\phi)$ is generic if and only if $\phi$ is generic and does not contain $\chi_V |\cdot|^{\pm \frac{k+1}{2}} \boxtimes \mathrm{Sym}^{k-1}$ for any positive integer $k$, where $\mathrm{Sym}^{k-1}$ is the unique $k$-dimensional irreducible representation of $\SL_2(\CC)$.
Hence we see that for all but finitely many choices of $\chi_V$ (depending on $\phi$), $\theta(\phi)$ is generic if $\phi$ is.

\begin{prop}
\label{P:generic-theta}
Let $\phi$ be an $L$-parameter for $\U(W)$ and $\pi$ a representation of $\U(W)$ in $\Pi_\phi$.
Then we have:
\begin{enumerate}
\item Assume that $m=n$.
If $\phi$ is generic (so that $\theta(\phi)$ is also generic), then $\Theta_{\psi,V,W}(\pi) = \theta_{\psi,V,W}(\pi)$.
\item Assume that $m=n+1$.
If $\phi$ is generic and does not contain $\chi_V |\cdot|^{\pm \frac{k+1}{2}} \boxtimes \mathrm{Sym}^{k-1}$ for any positive integer $k$ (so that $\theta(\phi)$ is also generic), then $\Theta_{\psi,V,W}(\pi) = \theta_{\psi,V,W}(\pi)$.
\end{enumerate}
\end{prop}

\begin{proof}
We shall give the proof of (ii) since the proof of (i) is similar.
We may assume that $\Theta_{\psi,V,W}(\pi) \ne 0$.
If $\phi$ is tempered, then $\Theta_{\psi,V,W}(\pi)$ is irreducible and tempered by \cite[Proposition C.4(i)]{gi}.
In general, by Proposition \ref{prop:heiermann}, $\pi$ is a standard module of the form $\Ind ((\bigotimes_{i=1}^r \tau_i |\cdot|^{s_i}) \otimes \pi_0)$ as in \eqref{E:LQ}.
Then by \cite[Proposition C.4(ii)]{gi}, $\Theta_{\psi,V,W}(\pi)$ is a quotient of the standard module
\[ \Ind \bigl(  \bigl(  \bigotimes_{i=1}^r \tau_i \chi_V^{-1} \chi_W |\cdot|^{s_i} \bigr) \otimes \Theta_{\psi,V_0,W_0}(\pi_0) \bigr). \]
Since $\theta(\phi)$ is generic as well, Proposition \ref{prop:heiermann} implies that this standard module is irreducible, so that $\Theta_{\psi,V,W}(\pi)$ is irreducible.
\end{proof}

\subsection{(B) for generic $L$-parameters}

For special orthogonal groups, M{\oe}glin--Waldspurger \cite{mw} extended the Bessel case (B) of the Gross--Prasad conjecture from tempered $L$-parameters to generic $L$-parameters.
We carry out the analogous extension for unitary groups.

\begin{prop}
\label{P:generic-B}
The statement (B) holds for all generic $L$-parameters for $\U(V_n) \times \U(V_{n+2k+1})$.
\end{prop}

To prove Proposition \ref{P:generic-B}, we adapt the proof of M{\oe}glin--Waldspurger \cite{mw} to the case of unitary groups.
For any (not necessarily irreducible) smooth representations $\pi$ and $\pi'$ of $\U(V_n)$ and $\U(V_{n+2k+1})$ respectively, we write $m(\pi, \pi')$ or $m(\pi', \pi)$ for 
\[
 \dim_\CC \Hom_H(\pi \otimes \pi', \nu)
\]
with the subgroup $H$ of $\U(V_n) \times \U(V_{n+2k+1})$ and the character $\nu$ of $H$ as in \cite[\S 12]{ggp1}.
Then as explained in \cite[\S 3]{mw}, Proposition \ref{P:generic-B} follows from (B) for all tempered $L$-parameters (which was proved by Beuzart-Plessis \cite{bp1}, \cite{bp2}, \cite{bp3}), together with Proposition \ref{prop:heiermann} and the following proposition:

\begin{prop}
\label{prop:mult}
Let $\pi = \Ind ((\bigotimes_{i=1}^r \tau_i |\cdot|^{s_i}) \otimes \pi_0)$ be a smooth representation of $\U(V_n)$, where 
\begin{itemize}
 \item $\tau_i$ is an irreducible tempered representation of $\GL_{k_i}(E)$,
 \item $s_i$ is a real number such that $s_1 \ge \cdots \ge s_r \ge 0$,
 \item $\pi_0$ is an irreducible tempered representation of $\U(V_{n-2(k_1+\dots+k_r)})$.
\end{itemize}
Likewise, let $\pi' = \Ind((\bigotimes_{j=1}^{r'} \tau'_j |\cdot|^{s'_j}) \otimes \pi_0')$ be a smooth representation of $\U(V_{n+2k+1})$ with analogous data $\tau_j'$, $k_j'$, $s_j'$, $\pi_0'$.
Then we have
\[
 m(\pi, \pi') = m(\pi_0, \pi'_0).
\]
\end{prop}

\begin{proof}
Since the proof is similar to that of \cite[Proposition 1.3]{mw}, we shall only give a sketch of the proof.
First, we prove that $m(\pi, \pi') \le m(\pi_0, \pi'_0)$.
\begin{enumerate}
 \item Let $\sigma = \Ind(\tau_0 |\cdot|^{s_0} \otimes \sigma_0)$ be a smooth representation of $\U(V_{n+1})$, where
\begin{itemize}
 \item $\tau_0$ is an irreducible (unitary) square-integrable representation of $\GL_{k_0}(E)$,
 \item $s_0$ is a real number, 
 \item $\sigma_0$ is a smooth representation of $\U(V_{n-2k_0+1})$ of finite length.
\end{itemize}
Assume that $s_0 \ge s_1$ (which is interpreted as $s_0 \ge 0$ when $r=0$).
Then as in \cite[Lemme 1.4]{mw}, we have $m(\pi, \sigma) \le m(\pi, \sigma_0)$.
\item Let $\sigma$ be as in (i).
Assume that
\begin{itemize}
 \item $\tau_0$ is supercuspidal;
 \item if a representation $\tau_\sharp \otimes \pi_\sharp$ with
\begin{itemize}
 \item an irreducible smooth representation $\tau_\sharp$ of a general linear group;
 \item an irreducible smooth representation $\pi_\sharp$ of a general linear group or a unitary group
\end{itemize}
intervenes in a Jacquet module of $\tau_i^\vee$, $\tau_i^c$, or $\pi_0^\vee$ as a subquotient, then $\tau_0 |\cdot|^s$ does not intervene in the supercuspidal support of $\tau_\sharp$ for any $s \in \RR$.
\end{itemize}
Then by \cite[Theorem 15.1]{ggp1} (see also \cite[Lemme 1.5]{mw}), we have $m(\pi, \sigma) = m(\pi, \sigma_0)$.
\item To prove $m(\pi, \pi') \le m(\pi_0, \pi'_0)$ in general, we may assume that $\tau_i$, $\tau_j'$ are square-integrable for all $i$, $j$.
As in \cite[\S 1.6]{mw}, we argue by induction on
\[
 l := \sum_{\substack{1 \le i \le r \\ s_i \ne 0}} k_i + \sum_{\substack{1 \le j \le r' \\ s'_j \ne 0}} k'_j.
\]
If $l=0$, then it follows by \cite[\S\S 14--15]{bp3} combined with (ii) that $m(\pi, \pi') = m(\pi_0, \pi'_0)$.
Suppose that $l \ne 0$.
\begin{enumerate}
 \item If $k=0$ and $s_1' \ge s_1$ (in particular $r'\ge 1$), then by (i), we have $m(\pi, \pi') \le m(\pi, \pi'')$, where $\pi'' = \Ind((\bigotimes_{j=2}^{r'} \tau'_j |\cdot|^{s'_j}) \otimes \pi_0')$.
 By induction hypothesis, we have $m(\pi, \pi'') \le m(\pi_0, \pi_0')$.
 \item If $s_1 \ge s_1'$ (in particular $r \ge 1$), then we can reduce to (a) by using (ii).
 \item If $s_1' \ge s_1$ (in particular $r' \ge 1$), then we can reduce to (b) by using (ii).
\end{enumerate}
This proves the assertion (see \cite[\S 1.6]{mw} for details).
\end{enumerate}

Next, we prove that $m(\pi, \pi') \ge m(\pi_0, \pi'_0)$.
By (ii), we may assume that $k=0$.
If $m(\pi_0, \pi_0') = 0$, then there is nothing to prove.
If $m(\pi_0, \pi_0') \ne 0$, then by \cite{agrs}, \cite[Corollary 15.3]{ggp1}, it suffices to show that $m(\pi, \pi') \ge 1$.
Put
\[
 \pi_z = \Ind ((\bigotimes_{i=1}^r \tau_i |\cdot|^{z_i}) \otimes \pi_0)
 \quad \text{and} \quad
 \pi'_{z'} = \Ind ((\bigotimes_{j=1}^{r'} \tau'_j |\cdot|^{z'_j}) \otimes \pi_0')
\]
for $z = (z_1, \dots, z_r) \in \CC^r$ and $z' = (z'_1, \dots, z'_{r'}) \in \CC^{r'}$.
As in \cite[Lemme 1.7]{mw}, we can define a $\Delta(\U(V_n) \times \U(V_n))$-equivariant map
\[
 \mathcal{L}_{z,z'} : \pi_z \otimes (\pi_z)^\vee \otimes \pi'_{z'} \otimes (\pi'_{z'})^\vee \longrightarrow \CC
\]
by (meromorphic continuation of) an integral of matrix coefficients, which is absolutely convergent for $(z, z')$ near $(\sqrt{-1} \RR)^r \times (\sqrt{-1} \RR)^{r'}$.
Since $m(\pi_0, \pi_0') \ne 0$, it follows by \cite[Th\'eor\`eme 14.3.1, Proposition 15.2.1, Proposition 15.3.1]{bp3} that the map $(z, z') \mapsto \mathcal{L}_{z, z'}$ is not identically zero.
In particular, the leading term of $\mathcal{L}_{z,z'}$ at $z = (s_1, \dots, s_r)$ and $z' = (s_1', \dots, s'_{r'})$ is nonzero and hence $m(\pi, \pi') \ge 1$ (see \cite[\S 1.8]{mw} for details).
This completes the proof.
\end{proof}

\subsection{(FJ) for generic $L$-parameters}

In view of Propositions \ref{P:generic-theta} and \ref{P:generic-B}, one may repeat the see-saw argument in \S \ref{s:see-saw} for generic $L$-parameters, using (P1) and (P2) (which were shown for all $L$-parameters) to prove:

\begin{prop}
\label{P:FJ-generic}
The statement (FJ) holds for all generic $L$-parameters for $\U(W_n) \times \U(W_n)$.
\end{prop}

Here, in repeating the see-saw argument, one may choose a character $\chi_V$ so that the condition of Proposition \ref{P:generic-theta}(ii) holds.
Finally, Proposition \ref{P:FJ-generic} together with \cite[Theorem 19.1]{ggp1} implies:

\begin{cor}
The statement (FJ) holds for all generic $L$-parameters for $\U(W_n) \times \U(W_{n+2k})$.
\end{cor}

\appendix

\section{\textbf{Addendum to \cite{gi}}}

In this appendix, we elaborate on some results of \cite[Appendix C]{gi} which are used in the proof of Theorem \ref{T:al-equal}.
In particular,
\begin{itemize}
\item we fill in some missing details in the proof of \cite[Proposition C.1(ii)]{gi} and streamline its proof by exploiting the recently established Howe duality \cite{gt1}, \cite{gt2};
\item we extend some results of Mui\'c \cite[Lemma 4.2 and Theorem 5.1(i)]{muic} (used in the proof of \cite[Proposition C.1(ii)]{gi}), which were written only for symplectic-orthogonal dual pairs, to cover all dual pairs considered in \cite{gi}, streamlining some of his proofs in the process.
\end{itemize}

\subsection{The issues}

Let us be more precise.
We freely use the notation of \cite[\S C.1]{gi}.

Let $\pi$ be an irreducible square-integrable representation of $G(W)$ such that $\sigma_0 := \Theta_{\tilde{V},W,\Chi,\psi}(\pi) \ne 0$.
By the bullet point on \cite[p.~645]{gi}, together with the Howe duality, $\sigma_0$ is irreducible and square-integrable.
Then we showed that
\begin{enumerate}
 \item any irreducible subquotient of $\Theta_{V,W,\Chi,\psi}(\pi)$ is tempered in the first bullet point on \cite[p.~646]{gi};
 \item $\sigma := \theta_{V,W,\Chi,\psi}(\pi)$ is an irreducible constituent of $I^{H(V)}_{Q(Y_1)}(\chi_W \otimes \sigma_0)$ in the the second bullet point on \cite[p.~646]{gi},
\end{enumerate}
and claimed that
\begin{enumerate}[resume]
\item any irreducible subquotient of $\Theta_{V,W,\Chi,\psi}(\pi)$ is not square-integrable in the third bullet point on \cite[p.~646]{gi};
\item any irreducible subquotient of $\Theta_{V,W,\Chi,\psi}(\pi)$ is a subrepresentation of $I^{H(V)}_{Q(Y_1)}(\chi_W \otimes \sigma_0')$ for some irreducible smooth representation $\sigma_0'$ of $H(\tilde{V})$ in the fourth bullet point on \cite[p.~646]{gi}.
\end{enumerate}
However, in the third and fourth bullet points on \cite[p.~646]{gi}, we have used results of Mui\'c \cite[Lemma 4.2 and Theorem 5.1(i)]{muic}, which were written only for symplectic-orthogonal dual pairs.
Moreover, we have not given the proof of (iv): we have simply asserted that it is true as if it is obvious (which it is not).
Thus, we need to give the details of the proof of (iii) and (iv), as well as that of the results of Mui\'c for all dual pairs considered in \cite{gi}.

\subsection{Proof of (iii)}

First, we address (iii).
Our original argument in \cite{gi} used \cite[Lemma 4.2 and Theorem 5.1(i)]{muic}, which we state and prove in Lemma \ref{lem:theta-scsupp} and Corollary \ref{C:muicsq} below.
Here, we give a more streamlined argument using the recently established Howe duality \cite{gt1}, \cite{gt2}.

Let $\sigma'$ be an irreducible subquotient of $\Theta_{V,W,\Chi,\psi}(\pi)$.
Suppose that $\sigma'$ is square-integrable.
Since $\Theta_{V,W,\Chi,\psi}(\pi)$ is of finite length and tempered by (i), it follows by \cite[Corollaire III.7.2]{w0} that $\sigma'$ is in fact a quotient of $\Theta_{V,W,\Chi,\psi}(\pi)$.
Hence we must have $\sigma' \cong \sigma$ by the Howe duality.
But $\sigma$ is not square-integrable by (ii), which is a contradiction.
This completes the proof of (iii).

\subsection{Proof of  \cite[Lemma 4.2]{muic}}

For the proof of (iv), we will need the following result of Mui\'c \cite[Lemma 4.2]{muic}.

\begin{lem}[Mui\'c]
\label{lem:theta-scsupp}
Let $G(W) \times H(V)$ be an arbitrary reductive dual pair as in \cite[\S 3]{gi}.
Let $\pi$ be an irreducible smooth representation of $G(W)$.
Then all irreducible subquotients of $\Theta_{V,W,\Chi,\psi}(\pi)$ have the same supercuspidal support.
\end{lem}

\begin{proof}
We may assume that $\Theta_{V,W,\Chi,\psi}(\pi) \ne 0$.
Since $\Theta_{V,W,\Chi,\psi}(\pi)$ is of finite length, it follows by the theory of the Bernstein center \cite{bernstein} that
\[
 \Theta_{V,W,\Chi,\psi}(\pi) = \sigma_1 \oplus \dots \oplus \sigma_r
\]
for some smooth representations $\sigma_i$ of $H(V)$ of finite length such that
\begin{itemize}
 \item for each $i$, all irreducible subquotients of $\sigma_i$ have the same supercuspidal support, say, $\supp \sigma_i$;
 \item if $i \ne j$, then $\supp \sigma_i \ne \supp \sigma_j$.
\end{itemize}
Of course, if we were willing to appeal to the Howe duality, then it would follow immediately that $r=1$, so that the lemma is proved.
However, we may appeal to an older result of Kudla.
Namely, Kudla's supercuspidal support theorem \cite{kudla} (see also \cite[Proposition 5.2]{gi} and the references therein) says that the supercuspidal support of $\theta_{V,W,\Chi,\psi}(\pi)$ is determined by that of $\pi$.
Hence we must have $r=1$.
\end{proof}

\subsection{Plancherel measures}

To prove (iv), we will also need the following property of Plancherel measures.
We freely use the convention of \cite[Appendix B]{gi}.

\begin{lem}
\label{lem:plancherel}
Let $G(W)$ be an arbitrary classical group as in \cite[\S 2]{gi}.
Let $\pi$ be an irreducible tempered representation of $G(W)$ such that
\[
 \pi \subset I^{G(W)}_P(\tau_1 \otimes \dots \otimes \tau_r \otimes \pi_0),
\]
where $P$ is a parabolic subgroup of $G(W)$ with Levi component $\GL_{k_1}(E) \times \dots \times \GL_{k_r}(E) \times G(W_0)$, $\tau_i$ is an irreducible (unitary) square-integrable representation of $\GL_{k_i}(E)$, and $\pi_0$ is an irreducible square-integrable representation of $G(W_0)$.
Let $\tau$ be an irreducible (unitary) square-integrable representation of $\GL_k(E)$ and put
\[
 \mathcal{I}(\tau) = \{ i \, | \, \tau_i \cong \tau \}.
\]
Then we have
\[
 \ord_{s=0} \mu(\tau_s \otimes \pi) = 2 \cdot \# \mathcal{I}(\tau) + 2 \cdot \# \mathcal{I}((\tau^c)^\vee) + \ord_{s=0} \mu(\tau_s \otimes \pi_0).
\]
Moreover, we have
\[
 \ord_{s=0} \mu(\tau_s \otimes \pi_0) = 
 \begin{cases}
  \text{$0$ or $2$} & \text{if $(\tau^c)^\vee \cong \tau$;} \\
  0 & \text{if $(\tau^c)^\vee \ncong \tau$.} 
 \end{cases}
\]
\end{lem}

\begin{proof}
By the multiplicativity of Plancherel measures (see \cite[\S B.5]{gi}), we have
\[
 \mu(\tau_s \otimes \pi)
 = \left( \prod_{i=1}^r \mu(\tau_s \otimes \tau_i) \cdot \mu(\tau_s \otimes (\tau_i^c)^\vee) \right)
 \cdot \mu(\tau_s \otimes \pi_0).
\]
For any irreducible (unitary) square-integrable representation $\tau'$ of $\GL_{k'}(E)$, we have
\[
 \mu(\tau_s \otimes \tau') = \gamma(s, \tau \times (\tau')^\vee, \psi_E) \cdot \gamma(-s, \tau^\vee \times \tau', \bar{\psi}_E)
\]
and hence
\[
 \ord_{s=0} \mu(\tau_s \otimes \tau') = 
 \begin{cases}
  2 & \text{if $\tau \cong \tau'$;} \\
  0 & \text{if $\tau \ncong \tau'$,}
 \end{cases}
\]
which reflects the triviality of $R$-groups for general linear groups.
This proves the first assertion.
The second assertion follows from \cite[Corollaire IV.1.2]{w0} if $(\tau^c)^\vee \cong \tau$ and \cite[Proposition IV.2.2]{w0} if $(\tau^c)^\vee \ncong \tau$.
\end{proof}

\subsection{Proof of (iv)}

Now we prove (iv).
Let $\sigma'$ be an irreducible subquotient of $\Theta_{V,W,\Chi,\psi}(\pi)$.
By (i) and (iii), we have
\[
 \sigma' \subset I^{H(V)}_Q(\tau_1 \otimes \dots \otimes \tau_r \otimes \sigma_0')
\]
for some $r \ge 1$ and irreducible square-integrable representations $\tau_i$ and $\sigma_0'$ of $\GL_{k_i}(E)$ and $H(V_0)$ respectively, where $Q$ is a parabolic subgroup of $H(V)$ with Levi component $\GL_{k_1}(E) \times \dots \times \GL_{k_r}(E) \times H(V_0)$.
We need to show that $\tau_i = \chi_W$ for some $i$.

By Lemma \ref{lem:theta-scsupp} and the multiplicativity of Plancherel measures, we have
\[
 \mu((\chi_W)_s \otimes \sigma') = \mu((\chi_W)_s \otimes \sigma).
\]
By (ii) and Lemma \ref{lem:plancherel}, the right-hand side has a zero at $s=0$ of order at least $4$.
Hence, by Lemma \ref{lem:plancherel} again, we must have $\tau_i = \chi_W$ for some $i$.
This completes the proof of (iv).

\begin{rem}
In the proof of (iii) and (iv), we have used some results of Waldspurger \cite{w0}, which were written only for connected reductive linear algebraic groups.
However, it is straightforward to extend them to the cases of (disconnected) orthogonal groups and (nonlinear) metaplectic groups.
\end{rem}

\subsection{Proof of \cite[Theorem 5.1(i)]{muic}}

As we noted above, we have used \cite[Theorem 5.1(i)]{muic} besides \cite[Lemma 4.2]{muic} in our original argument in \cite{gi}.
Although it is not necessary for the proof of (iii) and (iv) (because of the use of the Howe duality), we shall give a proof here.
In fact, we prove the following more general result by refining the argument in the proof of (iv).

\begin{lem}
\label{L:plansc}
Let $G(W)$ be an arbitrary classical group as in \cite[\S 2]{gi}.
Let $\pi$ be an irreducible tempered representation of $G(W)$ such that
\[
 \pi \subset I^{G(W)}_P(\tau_1 \otimes \dots \otimes \tau_r \otimes \pi_0),
\]
where $P$ is a parabolic subgroup of $G(W)$ with Levi component $\GL_{k_1}(E) \times \dots \times \GL_{k_r}(E) \times G(W_0)$, $\tau_i$ is an irreducible (unitary) square-integrable representation of $\GL_{k_i}(E)$, and $\pi_0$ is an irreducible square-integrable representation of $G(W_0)$.
Likewise, let $\pi'$ be an irreducible tempered representation of $G(W)$ such that
\[
 \pi' \subset I^{G(W)}_{P'}(\tau'_1 \otimes \dots \otimes \tau'_{r'} \otimes \pi_0')
\]
with analogous data $P'$, $r'$, $\tau'_i$, $\pi'_0$.
Assume that
\[
 \mu(\tau_s \otimes \pi) = \mu(\tau_s \otimes \pi')
\]
for all irreducible (unitary) square-integrable representations $\tau$ of $\GL_k(E)$ for all $k \ge 1$.
Then we have $r=r'$ and 
\[
 \{ \tau_1, \dots, \tau_r, (\tau_1^c)^\vee, \dots, (\tau_r^c)^\vee \} = \{ \tau_1', \dots, \tau_r', ((\tau_1')^c)^\vee, \dots, ((\tau_r')^c)^\vee \}
\]
as multi-sets.
Moreover, we have
\[ 
 \mu(\tau_s \otimes \pi_0) = \mu(\tau_s \otimes \pi_0')
\]
for all irreducible (unitary) square-integrable representations $\tau$ of $\GL_k(E)$ for all $k \ge 1$.
\end{lem}

\begin{proof}
Note that the second assertion is an immediate consequence of the first assertion and the multiplicativity of Plancherel measures.
To prove the first assertion, it suffices to show that
\begin{equation}
\label{eq:plansc}
 \# \mathcal{I}(\tau) + \# \mathcal{I}((\tau^c)^\vee) = \# \mathcal{I}'(\tau) + \# \mathcal{I}'((\tau^c)^\vee) 
\end{equation}
for any irreducible (unitary) square-integrable representation $\tau$ of $\GL_k(E)$, where $\mathcal{I}(\tau) = \{ i \, | \, \tau_i \cong \tau \}$ and $\mathcal{I}'(\tau) = \{ i \, | \, \tau_i' \cong \tau \}$.
If $(\tau^c)^\vee \cong \tau$, then by Lemma \ref{lem:plancherel}, we have
\[
 4 \cdot \# \mathcal{I}(\tau) + \alpha = 4 \cdot \# \mathcal{I}'(\tau) + \alpha'
\]
for some $0 \le \alpha, \alpha' \le 2$.
This forces $\# \mathcal{I}(\tau) = \# \mathcal{I}'(\tau)$, so that \eqref{eq:plansc} holds.
If $(\tau^c)^\vee \ncong \tau$, then \eqref{eq:plansc} is a direct consequence of Lemma \ref{lem:plancherel}.
This completes the proof.
\end{proof}

The following corollary (which is \cite[Theorem 5.1(i)]{muic}) is now immediate:

\begin{cor}[Mui\'c]  \label{C:muicsq}
Suppose that $\pi$ and $\pi'$ are irreducible tempered representations of $G(W)$ which have the same supercuspidal support.
If $\pi$ is square-integrable, then so is $\pi'$.
\end{cor}

\begin{proof}
If $\pi$ and $\pi'$ have the same supercuspidal support, then the multiplicativity of Plancherel measures implies that
\[ \mu(\tau_s \otimes \pi) = \mu(\tau_s \otimes \pi') \]
for all irreducible (unitary) square-integrable representations $\tau$ of $\GL_k(E)$ for all $k \ge 1$.
The assertion then follows from Lemma \ref{L:plansc}. 
\end{proof}

\subsection{Some variant}

Finally, admitting the local Langlands correspondence, we shall state a variant of Lemma \ref{L:plansc} in terms of $L$-parameters.

Let $G(W)$ be an arbitrary classical group as in \cite[\S 2]{gi}.
To each irreducible tempered representation $\pi$ of $G(W)$, the local Langlands correspondence assigns an $L$-parameter $\phi$, which we regard as a semisimple representation of $\WD_E$ as described in \cite[\S 8]{ggp1}.
Moreover, for any irreducible tempered representation $\tau$ of $\GL_k(E)$ with associated $L$-parameter $\phi_\tau$, Langlands' conjecture on Plancherel measures \cite[Appendix II]{l} says that
\begin{equation}
\label{eq:pl-gamma}
 \mu(\tau_s \otimes \pi) = \gamma(s, \phi_\tau \otimes \phi^\vee, \psi_E) \cdot \gamma(-s, \phi_\tau^\vee \otimes \phi, \bar{\psi}_E) \cdot \gamma(2s, R \circ \phi_\tau, \psi) \cdot \gamma(-2s, R \circ \phi_\tau^\vee, \bar{\psi}),
\end{equation}
where
\[
 R = 
 \begin{cases}
  \mathrm{Sym}^2 & \text{if $G(W)$ is odd orthogonal or metaplectic;} \\
  \wedge^2 & \text{if $G(W)$ is even orthogonal or symplectic;} \\
  \As^+ & \text{if $G(W)$ is even unitary;} \\
  \As^- & \text{if $G(W)$ is odd unitary.}
 \end{cases}
\]
(We note here that $\mathrm{Asai}$ in the bottom of \cite[p.~650]{gi} is a typo.)
In fact, \eqref{eq:pl-gamma} immediately follows from \cite[Proposition 2.3.1]{a}, \cite[Proposition 3.3.1]{mok}, \cite[Lemmas 2.2.3]{kmsw} (together with induction in stages) for classical groups considered there.
(See also \S \ref{SS:normalization} in the case of unitary groups.)
In other words, recalling the definitions of of Plancherel measures and normalized intertwining operators, we see that \eqref{eq:pl-gamma} is a consequence of a property of normalized intertwining operators.
Also, in the case of metaplectic groups, \eqref{eq:pl-gamma} follows from the case of odd orthogonal groups combined with \cite[Proposition 10.1]{gs}.

\begin{lem}
Let $\pi$ and $\pi'$ be irreducible tempered representations of $G(W)$ with associated $L$-parameters $\phi$ and $\phi'$ respectively.
Assume that
\[
 \mu(\tau_s \otimes \pi) = \mu(\tau_s \otimes \pi') 
\]
for all irreducible (unitary) square-integrable representations $\tau$ of $\GL_k(E)$ for all $k \ge 1$.
Then we have
\[
 \phi = \phi'.
\]
\end{lem}

\begin{proof}
For any irreducible (unitary) square-integrable representation $\tau$ of $\GL_k(E)$ with associated $L$-parameter $\phi_\tau$, we have
\[
 \gamma(s, \phi_\tau \otimes \phi^\vee, \psi_E) \cdot \gamma(-s, \phi_\tau^\vee \otimes \phi, \bar{\psi}_E) 
 = \gamma(s, \phi_\tau \otimes (\phi')^\vee, \psi_E) \cdot \gamma(-s, \phi_\tau^\vee \otimes \phi', \bar{\psi}_E)
\]
by assumption and \eqref{eq:pl-gamma}.
Comparing the orders of zero at $s=0$, we see that the multiplicities of $\phi_\tau$ in $\phi$ and $\phi'$ are equal (see also \cite[Lemma 12.3]{gs}).
This completes the proof.
\end{proof}

\section{\textbf{Generic $L$-packets and adjoint $L$-factors}}

In this appendix, we prove a conjecture of Gross--Prasad and Rallis \cite[Conjecture 2.6]{gp1} under a certain working hypothesis.

\subsection{Notation}

Let $G$ be a connected reductive algebraic group defined and quasi-split over $F$.
Fix a Borel subgroup $B$ of $G$ over $F$ and a maximal torus $T$ in $B$ over $F$.
Let $N$ be the unipotent radical of $B$, so that $B = T N$.
If $P$ is a parabolic subgroup of $G$ over $F$, we say that $P$ is standard (relative to $B$) if $P \supset B$.
If $P$ is a standard parabolic subgroup of $G$ over $F$, then we have a Levi decomposition $P = MU$, where $M$ is the unique Levi component of $P$ such that $M \supset T$ and $U$ is the unipotent radical of $P$.
We call $M$ a standard Levi subgroup of $G$.
Let $W^M = \Norm_M(T)/T$ be the Weyl group of $M$ and $w^M_0$ the longest element in $W^M$.
Put
\[
 \aa_M^* = \Rat(M) \otimes_\ZZ \RR, \quad \aa_M = \Hom_\ZZ(\Rat(M), \RR),
\]
where $\Rat(M)$ is the group of algebraic characters of $M$ defined over $F$.
We write $\langle \cdot, \cdot \rangle : \aa_M^* \times \aa_M \rightarrow \RR$ for the natural pairing.
Let $\aa^*_{M,\CC} = \aa_M^* \otimes_\RR \CC$ be the complexification of $\aa_M^*$.
Let $A_M$ be the split component of the center of $M$ and $\Sigma(P)$ the set of reduced roots of $A_M$ in $P$.
We may regard $\Sigma(P)$ as a subset of $\aa_M^* \cong \Rat(A_M) \otimes_\ZZ \RR$.
For $\alpha \in \Sigma(P)$, let $\alpha^\vee \in \aa_M$ denote its corresponding coroot.
Put
\[
 (\aa_M^*)^+ = \{ \lambda \in \aa_M^{*} \, | \, \langle \lambda, \alpha^\vee \rangle > 0 \ \text{for all} \ \alpha \in \Sigma(P) \}.
\]
We define a homomorphism $H_M : M \rightarrow \aa_M$ by requiring that
\[
 |\chi(m)|_F = q^{-\langle \chi, H_M(m) \rangle}
\]
for all $\chi \in \Rat(M)$ and $m \in M$, where $q$ is the cardinality of the residue field of $F$.

Let $\pi$ be an irreducible smooth representation of $M$.
For $\lambda \in \aa^*_{M,\CC}$, we define a representation $\pi_\lambda$ of $M$ by $\pi_\lambda(m) = q^{- \langle \lambda, H_M(m) \rangle} \pi(m)$.
We write
\[
 I^G_P(\pi_\lambda) := \Ind^G_P(\pi_\lambda)
\]
for the induced representation of $G$.
If $\pi$ is tempered and $\Re(\lambda) \in (\aa^*_M)^+$, then $I^G_P(\pi_\lambda)$ has a unique irreducible quotient $J^G_P(\pi_\lambda)$.

Let $\widehat{M}$ be the dual group of $M$ and ${}^L M = \widehat{M} \rtimes W_F$ the $L$-group of $M$.
Let $Z(\widehat{M})$ be the center of $\widehat{M}$.
We write $\iota_M : {}^L M \hookrightarrow {}^L G$ for the natural embedding.
If $\phi: \WD_F \rightarrow {}^L M$ is an $L$-parameter, we say that $\phi$ is tempered if the projection of $\phi(W_F)$ to $\widehat{M}$ is bounded.
For $\lambda \in \aa^*_{M,\CC}$, we define an $L$-parameter $\phi_\lambda : \WD_F \rightarrow {}^L M$ by $\phi_\lambda = a_\lambda \cdot \phi$, where $a_\lambda \in Z^1(W_F, Z(\widehat{M}))$ is a $1$-cocycle which determines the character $m \mapsto q^{- \langle \lambda, H_M(m) \rangle}$ of $M$.

\subsection{Hypothesis}
\label{ss:gpr-hyp}

In this appendix, we admit the local Langlands correspondence for any standard Levi subgroup $M$ of $G$:
\[
 \Irr(M) = \bigsqcup_{\phi} \Pi_{\phi},
\]
where the disjoint union on the right-hand side runs over all equivalence classes of $L$-parameters $\phi$ for $M$ and $\Pi_{\phi}$ is a finite set of representations of $M$, the so-called $L$-packet.
More precisely, we will use the following properties of the local Langlands correspondence:
\begin{enumerate}
 \item $\pi \in \Pi_\phi$ is tempered if and only if $\phi$ is tempered.
 \item $\Pi_{\phi_\lambda} = \{ \pi_\lambda \, | \, \pi \in \Pi_\phi \}$ for $\lambda \in \aa_{M,\CC}^*$.
 \item If $\phi$ is an $L$-parameter for $G$, then replacing $\phi$ by its $\widehat{G}$-conjugate if necessary, we can write 
 \[
  \phi = \iota_M \circ (\phi_M)_{\lambda_0},
 \]
 where 
 \begin{itemize}
  \item $M$ is a standard Levi subgroup of $G$, 
  \item $\phi_M$ is a tempered $L$-parameter for $M$, 
  \item $\lambda_0 \in (\aa_M^*)^+$.
 \end{itemize}
 Then we have
 \[
  \Pi_{\phi} = \{ J^G_P(\pi_{\lambda_0}) \, | \, \pi \in \Pi_{\phi_M} \},
 \]
 where $P$ is the standard parabolic subgroup of $G$ with Levi component $M$.
 Note that $\pi \in \Pi_{\phi_M}$ is tempered by (i) and $\pi_{\lambda_0}$ has $L$-parameter $(\phi_M)_{\lambda_0}$ by (ii).
 \item If $\phi$ is a tempered $L$-parameter for $M$, then for any generic character $\psi_{N_M}$ of $N_M := N \cap M$, $\Pi_\phi$ contains a $(N_M, \psi_{N_M})$-generic representation $\pi$ of $M$ (see \cite[Conjecture 9.4]{shahidi}).
 Moreover, we have
\[
 \gamma^{\Sh}(s, \pi_\lambda, r_M, \psi) = \gamma(s, r_M \circ \phi_\lambda, \psi),
\]
where the left-hand side is Shahidi's $\gamma$-factor \cite{shahidi} and $r_M$ is the adjoint representation of ${}^L M$ on $\Lie({}^L U)$.
In fact, we only need the equality up to an invertible function.
\end{enumerate}

The above hypothesis is known to hold for general linear groups by \cite{ht, he, scholze} and for classical groups by \cite{a, mok}.

\subsection{A conjecture of Gross--Prasad and Rallis}

If $\phi$ is an $L$-parameter for $G$, we say that $\phi$ is generic if its associated $L$-packet $\Pi_\phi$ contains a $(N, \psi_N)$-generic representation of $G$ for some generic character $\psi_N$ of $N$.

\begin{prop}
\label{p:gp-rallis}
Let $\phi$ be an $L$-parameter for $G$.
Then, under the hypothesis in \S \ref{ss:gpr-hyp}, $\phi$ is generic if and only if $L(s, \Ad \circ \phi)$ is holomorphic at $s=1$.
Here, $\Ad$ is the adjoint representation of ${}^L G$ on its Lie algebra $\Lie({}^L G)$.
\end{prop}

\subsection{Proof of Proposition \ref{p:gp-rallis}}

Fix an $L$-parameter $\phi$ for $G$ and write $\phi = \iota_M \circ (\phi_M)_{\lambda_0}$ as in (iii).
Then by (iii), $\phi$ is generic if and only if $J^G_P(\pi_{\lambda_0})$ is $(N, \psi_N)$-generic for some $\pi \in \Pi_{\phi_M}$ and some generic character $\psi_N$ of $N$, in which case $\pi$ is necessarily $(N_M, \psi_N|_{N_M})$-generic by a result of Rodier \cite{rodier}, \cite[Corollary 1.7]{cs}.
Here, we have also used the fact that for any element $w$ in $W^G$, there exists a representative $\tilde{w}$ of $w$ (depending on $\psi_N$) such that $\psi_N$ is compatible with $\tilde{w}$ (see \cite[\S 2]{shahidi2}, \cite[\S 1.2]{cpss}).
Now we invoke the following result of Heiermann--Mui\'c \cite[Proposition 1.3]{hm}.

\begin{lem}
\label{l:gen-key}
Let $\psi_N$ be a generic character of $N$ and $\pi$ an irreducible tempered $(N_M, \psi_N|_{N_M})$-generic representation of $M$.
Then $J^G_P(\pi_{\lambda_0})$ is $(N, \psi_N)$-generic if and only if $\gamma^{\Sh}(0, \pi_\lambda, r_M, \psi)$ is holomorphic at $\lambda=\lambda_0$.
\end{lem}

\begin{proof}
Since the assertion in \cite[Proposition 1.3]{hm} is slightly different, we include a proof for the convenience of the reader.
We realize the representation $I^G_P(\pi_\lambda)$ by using the unique (up to a scalar) Whittaker functional on $\pi$ with respect to $(N_M, \psi_N|_{N_M})$.
Then we can define a Whittaker functional
\[
 \Lambda(\pi_\lambda) : I^G_P(\pi_\lambda) \longrightarrow \CC
\]
with respect to $(N,\psi_N)$ by (holomorphic continuation of) the Jacquet integral (see \cite[Proposition 3.1]{shahidi-certain}).
By \cite{rodier}, \cite[Corollary 1.7]{cs}, $\Lambda(\pi_\lambda)$ is a basis of $\Hom_N(I^G_P(\pi_\lambda), \psi_N)$ for all $\lambda \in \aa_{M,\CC}^*$.
Put $w = w^G_0 w^M_0$ and choose its representative $\tilde{w}$ so that $\psi_N$ is compatible with $\tilde{w}$.
As in \S \ref{SS:normalization}, we can define an unnormalized intertwining operator
\[
 \M(\tilde{w}, \pi_\lambda) : I^G_P(\pi_\lambda) \longrightarrow I^G_{w(P)}(w(\pi_\lambda))
\]
by (meromorphic continuation of) an integral which is absolutely convergent for $\Re(\lambda) \in (\aa^*_M)^+$ (see \cite[Proposition IV.2.1]{w0}), where $w(P)$ is the standard parabolic subgroup of $G$ with Levi component $w M w^{-1}$.
Then we have
\begin{equation}
\label{eq:local-coeff}
 \Lambda(\pi_\lambda) = C(\tilde{w}, \pi_\lambda) \cdot \Lambda(w(\pi_\lambda)) \circ \M(\tilde{w}, \pi_\lambda) 
\end{equation}
for some meromorphic function $C(\tilde{w}, \pi_\lambda)$, the so-called local coefficient.
Here, $C(\tilde{w}, \pi_\lambda)$ depends on the choice of Haar measures in the definitions of $\Lambda(\pi_\lambda)$, $\Lambda(w(\pi_\lambda))$, $\M(\tilde{w}, \pi_\lambda)$, but we ignore the normalization of Haar measures since it does not affect the proof.
Since $J^G_P(\pi_{\lambda_0})$ is isomorphic to the image of $\M(\tilde{w}, \pi_{\lambda_0})$ and the functor $\Hom_N(\, \cdot \, , \psi_N)$ is exact, $J^G_P(\pi_{\lambda_0})$ is $(N, \psi_N)$-generic if and only if the restriction of $\Lambda(w(\pi_{\lambda_0}))$ to the image of $\M(\tilde{w}, \pi_{\lambda_0})$ is nonzero.
By \eqref{eq:local-coeff}, this condition is equivalent to the holomorphy of $C(\tilde{w}, \pi_\lambda)$ at $\lambda = \lambda_0$.
On the other hand, by the definition of Shahidi's $\gamma$-factor, we have
\[
 C(\tilde{w}, \pi_\lambda) = \gamma^{\Sh}(0, \pi_\lambda, r_M, \psi)
\]
up to an invertible function.
(Note that the convention in \cite{shahidi} is different from ours: the homomorphism $H_M$ is normalized so that $|\chi(m)|_F = q^{\langle \chi, H_M(m) \rangle}$ in \cite{shahidi}.
This is why we have $\gamma^{\Sh}(0, \pi_\lambda, r_M, \psi)$ on the right-hand side rather than $\gamma^{\Sh}(0, \pi_\lambda, r_M^\vee, \bar{\psi})$.)
This completes the proof.
\end{proof}

Now it follows by Lemma \ref{l:gen-key} combined with (iv) that $\phi$ is generic if and only if
\begin{equation}
\label{eq:L(1)/L(0)}
 \frac{L(1, r_M^\vee \circ (\phi_M)_\lambda)}{L(0, r_M \circ (\phi_M)_\lambda)} 
\end{equation}
is holomorphic at $\lambda = \lambda_0$.
We consider the analytic property of \eqref{eq:L(1)/L(0)}.
For $\alpha \in \Sigma(P)$, let $A_\alpha$ be the identity component of $\Ker(\alpha)$, $M_\alpha$ the centralizer of $A_\alpha$ in $G$, and $U_\alpha$ the root subgroup associated to $\alpha$.
Then $M_\alpha$ is a Levi subgroup of $G$ (but not necessarily a Levi component of a standard parabolic subgroup of $G$) and $M U_\alpha$ is a maximal parabolic subgroup of $M_\alpha$.
We may regard $\aa_{M_\alpha}$ as a subspace of $\aa_M$.
Put
\[
 (\aa_{M}^{M_\alpha})^* = \{ \lambda \in \aa_M^* \, | \, \langle \lambda, H \rangle = 0 \ \text{for all} \ H \in \aa_{M_\alpha} \}.
\]
For $\lambda \in \aa_{M,\CC}^*$, let $\lambda^{M_\alpha}$ denote its orthogonal projection to $(\aa_M^{M_\alpha})^* \otimes_\RR \CC$.
We can write
\[
 \lambda^{M_\alpha} = s_\alpha \cdot \varpi_\alpha 
\]
for some $s_\alpha = s_\alpha(\lambda) \in \CC$, where $\varpi_\alpha \in (\aa_{M}^{M_\alpha})^*$ is the unique element such that $\langle \varpi_\alpha, \alpha^\vee \rangle = 1$.
Then we have
\[
 \eqref{eq:L(1)/L(0)} = 
 \prod_{\alpha \in \Sigma(P)} \frac{L(1 - s_\alpha, r_\alpha^\vee \circ \phi_M)}{L(s_\alpha, r_\alpha \circ \phi_M)},
\]
where $r_\alpha$ is the adjoint representation of ${}^L M$ on $\Lie({}^L U_\alpha)$.
Note that $L(s, r_\alpha \circ \phi_M)$ is holomorphic and nonzero for $\Re(s) > 0$ since $\phi_M$ is tempered.
If we put $s_{0,\alpha} = s_\alpha(\lambda_0) \in \RR$, then we have $s_{0,\alpha} > 0$ for all $\alpha \in \Sigma(P)$ since $\lambda_0 \in (\aa^*_M)^+$.
Hence \eqref{eq:L(1)/L(0)} is holomorphic at $\lambda = \lambda_0$ if and only if
\[
 \prod_{\alpha \in \Sigma(P)} L(1 - s_\alpha - s_{0,\alpha}, r_\alpha^\vee \circ \phi_M)
\]
is holomorphic at $\lambda = 0$.
Since the $L$-factors have no zeros, this condition is equivalent to the holomorphy of $L(s - s_{0,\alpha}, r_\alpha^\vee \circ \phi_M)$ at $s=1$ for all $\alpha \in \Sigma(P)$, which in turn is equivalent to the holomorphy of
\[
 L(s, r_M^\vee \circ (\phi_M)_{\lambda_0}) = \prod_{\alpha \in \Sigma(P)}  L(s - s_{0,\alpha}, r_\alpha^\vee \circ \phi_M)  
\]
at $s=1$.
Thus, we have shown that $\phi$ is generic if and only if $L(s, r_M^\vee \circ (\phi_M)_{\lambda_0})$ is holomorphic at $s=1$.

On the other hand, we have
\[
 L(s, \Ad \circ \phi) = 
 L(s, r_M \circ (\phi_M)_{\lambda_0}) \cdot L(s, \Ad_M \circ (\phi_M)_{\lambda_0}) \cdot L(s, r_M^\vee \circ (\phi_M)_{\lambda_0}),
\]
where $\Ad_M$ is the adjoint representation of ${}^L M$ on $\Lie({}^L M)$.
Since $\phi_M$ is tempered and $s_{0,\alpha} > 0$ for all $\alpha \in \Sigma(P)$,
\[
 L(s, r_M \circ (\phi_M)_{\lambda_0}) = \prod_{\alpha \in \Sigma(P)}  L(s + s_{0,\alpha}, r_\alpha \circ \phi_M)
\]
and $L(s, \Ad_M \circ (\phi_M)_{\lambda_0}) = L(s, \Ad_M \circ \phi_M)$ are holomorphic and nonzero for $\Re(s) > 0$.
Hence $L(s, \Ad \circ \phi)$ is holomorphic at $s=1$ if and only if $L(s, r_M^\vee \circ (\phi_M)_{\lambda_0})$ is holomorphic at $s=1$.
This completes the proof of Proposition \ref{p:gp-rallis}.

\begin{rem}
If $G$ is a classical group, then one has the following variant of Proposition \ref{p:gp-rallis} which does not rely on the local Langlands correspondence.
Fix a generic character $\psi_N$ of $N$.
If $\pi$ is an irreducible $(N,\psi_N)$-generic representation of $G$, let $\Pi$ be its functorial lift to the general linear group established in \cite{ckpss01, ckpss04, kk04, kk05, cpss11} (see \cite[Definition 7.1]{ckpss04} for the precise definition in the case when $G$ is split over $F$).
Put 
\[
 L^{\Sh}(s, \pi, {\Ad}) := L^{\Sh}(s, \Pi, R),
\]
where the right-hand side is Shahidi's $L$-factor \cite{shahidi} and 
\[
 R = 
 \begin{cases}
  \mathrm{Sym}^2 & \text{if $G$ is odd special orthogonal;} \\
  \wedge^2 & \text{if $G$ is even special orthogonal or symplectic;} \\
  \As^+ & \text{if $G$ is even unitary;} \\
  \As^- & \text{if $G$ is odd unitary.}
 \end{cases}
\]
If $\pi$ is tempered, then so is $\Pi$ (see \cite[Proposition 7.4]{ckpss04} when $G$ is split over $F$ and \cite[Proposition 8.6]{kk05} when $G$ is even unitary) and hence $L^{\Sh}(s, \pi, {\Ad})$ is holomorphic and nonzero for $\Re(s)>0$ (see \cite[Proposition 7.2]{shahidi}).
If we admit the local Langlands correspondence, then by \cite{he2}, we have $L^{\Sh}(s, \pi, {\Ad}) = L(s, {\Ad} \circ \phi)$, where $\phi$ is the $L$-parameter of $\pi$.

Now let $P$ be a standard parabolic subgroup of $G$ with Levi component $M$ and $\pi$ an irreducible tempered $(N_M, \psi_N|_{N_M})$-generic representation of $M$.
For any $\lambda \in \aa_{M,\CC}^*$, one has the $L$-factor $L^{\Sh}(s, I^G_P(\pi_\lambda), {\Ad})$ as above since the set of $\lambda$ such that $I^G_P(\pi_\lambda)$ is irreducible and $(N,\psi_N)$-generic is Zariski dense in $\aa_{M,\CC}^*$.
Then by the above argument (together with the multiplicativity), one can show that for $\lambda_0 \in (\aa_M^*)^+$, $J^G_P(\pi_{\lambda_0})$ is $(N, \psi_N)$-generic if and only if $L^{\Sh}(s, I^G_P(\pi_{\lambda_0}), {\Ad})$ is holomorphic at $s=1$.
\end{rem}


\begin{thebibliography}{99}

\bibitem{agrs}
A.~Aizenbud, D.~Gourevitch, S.~Rallis, and G.~Schiffmann,
\emph{Multiplicity one theorems},
Ann. of Math. \textbf{172} (2010), 1407--1434.

\bibitem{a}
J.~Arthur,
\emph{The endoscopic classification of representations: orthogonal and symplectic groups}, 
Colloquium Publications \textbf{61}, American Mathematical Society, 2013.

\bibitem{bernstein}
J.~N.~Bernstein,
\emph{Le ``centre'' de Bernstein},
Travaux en Cours, Representations of reductive groups over a local field, pp.~1--32, Hermann, 1984.

\bibitem{bp1}
R.~Beuzart-Plessis,
\emph{Expression d'un facteur epsilon de paire par une formule int\'egrale},
Canad. J. Math. \textbf{66} (2014), 993--1049.

\bibitem{bp2}
\bysame,
\emph{Endoscopie et conjecture raffin\'ee de Gan--Gross--Prasad pour les groupes unitaires}, 
Compos. Math., to appear.

\bibitem{bp3}
\bysame,
\emph{La conjecture locale de Gross--Prasad pour les repr\'esentations temp\'er\'ees des groupes unitaires},
\href{http://arxiv.org/abs/1205.2987}{\texttt{arXiv:1205.2987}}.

\bibitem{cs}
W.~Casselman and J.~Shalika,
\emph{The unramified principal series of $p$-adic groups. II. The Whittaker function},
Compos. Math. \textbf{41} (1980), 207--231.

\bibitem{cl}
P.-H.~Chaudouard and G.~Laumon,
\emph{Le lemme fondamental pond{\'e}r{\'e}. II. {\'E}nonc{\'e}s cohomologiques},
Ann. of Math. \textbf{176} (2012), 1647--1781.

\bibitem{clozel}
L.~Clozel,
\emph{On the cohomology of Kottwitz's arithmetic varieties},
Duke Math. J. \textbf{72} (1993), 757--795.

\bibitem{ckpss01}
J.~W.~Cogdell, H.~H.~Kim, I.~I.~Piatetski-Shapiro, and F.~Shahidi,
\emph{On lifting from classical groups to $\mathrm{GL}_N$},
Publ. Math. Inst. Hautes \'Etudes Sci. \textbf{93} (2001), 5--30.

\bibitem{ckpss04}
\bysame,
\emph{Functoriality for the classical groups},
Publ. Math. Inst. Hautes \'Etudes Sci. \textbf{99} (2004), 163--233.

\bibitem{cpss}
J.~W.~Cogdell, I.~I.~Piatetski-Shapiro, and F.~Shahidi,
\emph{Partial Bessel functions for quasi-split groups},
Automorphic representations, $L$-functions and applications: progress and prospects,
Ohio State Univ. Math. Res. Inst. Publ. \textbf{11}, pp.~95--128, de Gruyter, 2005.

\bibitem{cpss11}
\bysame,
\emph{Functoriality for the quasisplit classical groups},
On certain $L$-functions,
Clay Math. Proc. \textbf{13}, pp.~117--140, Amer. Math. Soc., 2011.

\bibitem{deligne}
P.~Deligne,
\emph{Les constantes des \'equations fonctionnelles des fonctions $L$},
Modular functions of one variable, II,
Lecture Notes in Math. \textbf{349}, pp.~501--597, Springer-Verlag, 1973.

\bibitem{ggp1}
W.~T.~Gan, B.~H.~Gross, and D.~Prasad,
\emph{Symplectic local root numbers, central critical $L$-values, and restriction problems in the representation theory of classical groups},
Ast\'erisque \textbf{346} (2012), 1--109.

\bibitem{ggp2} 
\bysame,
\emph{Restrictions of representations of classical groups: examples},
Ast\'erisque \textbf{346} (2012), 111--170.

\bibitem{gi}
W.~T.~Gan and A.~Ichino,
\emph{Formal degrees and local theta correspondence},
Invent. Math. \textbf{195} (2014), 509--672.

\bibitem{gqt}
W.~T.~Gan, Y.~Qiu, and S.~Takeda, 
\emph{The regularized Siegel--Weil formula (the second term identity) and the Rallis inner product formula},
Invent. Math. \textbf{198} (2014), 739--831.

\bibitem{gs}
W.~T.~Gan and G.~Savin,
\emph{Representations of metaplectic groups I: epsilon dichotomy and local Langlands correspondence},
Compos. Math. \textbf{148} (2012), 1655--1694.

\bibitem{gt1}
W.~T.~Gan and S.~Takeda,
\emph{On the Howe duality conjecture in classical theta correspondence},
\href{http://arxiv.org/abs/1405.2626}{\texttt{arXiv:1405.2626}}.

\bibitem{gt2}
\bysame,
\emph{A proof of the Howe duality conjecture},
\href{http://arxiv.org/abs/1407.1995}{\texttt{arXiv:1407.1995}}.

\bibitem{gj}
R.~Godement and H.~Jacquet,
\emph{Zeta functions of simple algebras},
Lecture Notes in Math. \textbf{260}, Springer-Verlag, 1972.

\bibitem{gp1}
B.~H.~Gross and D.~Prasad,
\emph{On the decomposition of a representation of $\mathrm{SO}_n$ when restricted to $\mathrm{SO}_{n-1}$},
Canad. J. Math. \textbf{44} (1992), 974--1002.

\bibitem{gp2}
\bysame,
\emph{On irreducible representations of $\mathrm{SO}_{2n+1} \times \mathrm{SO}_{2m}$},
Canad. J. Math. \textbf{46} (1994), 930--950.

\bibitem{hks}
M.~Harris, S.~S.~Kudla, and W.~J.~Sweet,~Jr., 
\emph{Theta dichotomy for unitary groups}, 
J. Amer. Math. Soc. \textbf{9} (1996), 941--1004.

\bibitem{ht}
M.~Harris and R.~Taylor,
\emph{The geometry and cohomology of some simple Shimura varieties},
Annals of Mathematics Studies \textbf{151}, Princeton University Press, 2001.

\bibitem{hei}
V.~Heiermann,
\emph{A note on standard modules and Vogan $L$-packets},
\href{http://arxiv.org/abs/1504.04524}{\texttt{arXiv:1504.04524}}.

\bibitem{hm}
V.~Heiermann and G.~Mui\'c,
\emph{On the standard modules conjecture}, 
Math. Z. \textbf{255} (2007), 847--853.

\bibitem{he}
G.~Henniart,
\emph{Une preuve simple des conjectures de Langlands pour $\mathrm{GL}(n)$ sur un corps $p$-adique},
Invent. Math. \textbf{139} (2000), 439--455.

\bibitem{he2}
\bysame,
\emph{Correspondance de Langlands et fonctions $L$ des carr\'es ext\'erieur et sym\'etrique},
Int. Math. Res. Not. (2010), 633--673.

\bibitem{i}
A.~Ichino,
\emph{On the local theta correspondence and $R$-groups},
Compos. Math. \textbf{140} (2004), 301--316.

\bibitem{ka}
T.~Kaletha,
\emph{Genericity and contragredience in the local Langlands correspondence},
Algebra Number Theory \textbf{7} (2013), 2447--2474.

\bibitem{kmsw}
T.~Kaletha, A.~M\'inguez, S.~W.~Shin, and P.-J.~White.
\emph{Endoscopic classification of representations: inner forms of unitary groups}, 
\href{http://arxiv.org/abs/1409.3731}{\texttt{arXiv:1409.3731}}.

\bibitem{kk04}
H.~H.~Kim and M.~Krishnamurthy,
\emph{Base change lift for odd unitary groups},
Functional analysis VIII,
Various Publ. Ser. (Aarhus) \textbf{47}, pp.~116--125, Aarhus Univ., 2004.

\bibitem{kk05}
\bysame,
\emph{Stable base change lift from unitary groups to $\mathrm{GL}_n$},
Int. Math. Res. Pap. (2005), 1--52.

\bibitem{kudla}
S.~S.~Kudla,
\emph{On the local theta-correspondence},
Invent. Math. \textbf{83} (1986), 229--255.

\bibitem{k}
\bysame,
\emph{Splitting metaplectic covers of dual reductive pairs},
Israel J. Math. \textbf{87} (1994), 361--401.

\bibitem{l}
R.~P.~Langlands,
\emph{On the functional equations satisfied by Eisenstein series},
Lecture Notes in Math. \textbf{544}, Springer-Verlag, 1976.

\bibitem{ls}
R.~P.~Langlands and D.~Shelstad,
\emph{On the definition of transfer factors},
Math. Ann. \textbf{278} (1987), 219--271. 

\bibitem{lr}
E.~M.~Lapid and S.~Rallis,
\emph{On the local factors of representations of classical groups},
Automorphic representations, $L$-functions and applications: progress and prospects,
Ohio State Univ. Math. Res. Inst. Publ. \textbf{11}, pp.~309--359, de Gruyter, 2005.

\bibitem{m}
C.~M{\oe}glin,
\emph{Conjecture d'Adams pour la correspondance de Howe et filtration de Kudla},
Arithmetic geometry and automorphic forms,
Adv. Lect. Math. \textbf{19}, pp.~445--503, Int. Press, 2011.

\bibitem{mw}
C.~M{\oe}glin and J.-L.~Waldspurger,
\emph{La conjecture locale de Gross--Prasad pour les groupes sp\'eciaux orthogonaux: le cas g\'en\'eral},
Ast\'erisque \textbf{347} (2012), 167--216.

\bibitem{mw-2014}
\bysame,
\emph{Stabilisation de la formule des traces tordue X: stabilisation spectrale},
\href{http://arxiv.org/abs/1412.2981}{\texttt{arXiv:1412.2981}}.

\bibitem{mok}
C.~P.~Mok,
\emph{Endoscopic classification of representations of quasi-split unitary groups},
Mem. Amer. Math. Soc., to appear.

\bibitem{muic}
G.~Mui\'c,
\emph{On the structure of theta lifts of discrete series for dual pairs $(\mathrm{Sp}(n),\mathrm{O}(V))$},
Israel J. Math. \textbf{164} (2008), 87--124.

\bibitem{psr}
I.~I.~Piatetski-Shapiro and S.~Rallis,
\emph{$L$-functions for the classical groups}, 
Explicit constructions of automorphic $L$-functions,
Lecture Notes in Math. \textbf{1254}, pp.~1--52, Springer-Verlag, 1987.

\bibitem{p1}
D.~Prasad,
\emph{On the local Howe duality correspondence},
Int. Math. Res. Not. (1993), 279--287.

\bibitem{p2}
\bysame,
\emph{Theta correspondence for unitary groups},
Pacific J. Math. \textbf{194} (2000), 427--438.

\bibitem{rangarao}
R.~Ranga Rao,
\emph{On some explicit formulas in the theory of Weil representation},
Pacific J. Math. \textbf{157} (1993), 335--371.

\bibitem{rodier}
F.~Rodier,
\emph{Whittaker models for admissible representations of reductive $p$-adic split groups},
Harmonic analysis on homogeneous spaces,
Proc. Sympos. Pure Math. \textbf{26}, pp.~425--430, Amer. Math. Soc., 1973.

\bibitem{scholze}
P.~Scholze,
\emph{The local Langlands correspondence for $\mathrm{GL}_n$ over $p$-adic fields},
Invent. Math. \textbf{192} (2013), 663--715.

\bibitem{shahidi-certain}
F.~Shahidi,
\emph{On certain $L$-functions},
Amer. J. Math. \textbf{103} (1981), 297--355.

\bibitem{shahidi}
\bysame,
\emph{A proof of Langlands' conjecture on Plancherel measures; complementary series for $p$-adic groups},
Ann. of Math. \textbf{132} (1990), 273--330.

\bibitem{shahidi2}
\bysame,
\emph{Local coefficients as Mellin transforms of Bessel functions: towards a general stability},
Int. Math. Res. Not. (2002), 2075--2119.

\bibitem{sh}
S.~W.~Shin,
\emph{Automorphic Plancherel density theorem},
Israel J. Math. \textbf{192} (2012), 83--120.

\bibitem{sun}
B.~Sun,
\emph{Multiplicity one theorems for Fourier--Jacobi models},
Amer. J. Math. \textbf{134} (2012), 1655--1678.

\bibitem{sz}
B.~Sun and C.-B.~Zhu,
\emph{Conservation relations for local theta correspondence},
J. Amer. Math. Soc., to appear.

\bibitem{v}
D.~A.~Vogan,~Jr.,
\emph{The local Langlands conjecture},
Representation theory of groups and algebras,
Contemp. Math. \textbf{145}, pp.~305--379, Amer. Math. Soc., 1993.

\bibitem{w}
J.-L.~Waldspurger,
\emph{D\'emonstration d'une conjecture de dualit\'e de Howe dans le cas $p$-adique, $p \ne 2$},
Festschrift in honor of I.~I.~Piatetski-Shapiro on the occasion of his sixtieth birthday, Part I,
Israel Math. Conf. Proc. \textbf{2}, pp.~267--324, Weizmann, 1990.

\bibitem{w0}
\bysame,
\emph{La formule de Plancherel pour les groupes $p$-adiques (d'apr\`es Harish-Chandra)},
J. Inst. Math. Jussieu \textbf{2} (2003), 235--333.

\bibitem{w1}
\bysame,
\emph{Une formule int\'egrale reli\'ee \`a la conjecture locale de Gross--Prasad},
Compos. Math. \textbf{146} (2010), 1180--1290.

\bibitem{w2}
\bysame,
\emph{Une formule int\'egrale reli\'ee \`a la conjecture locale de Gross--Prasad, 2\`eme partie: extension aux repr\'esentations temp\'er\'ees},
Ast\'erisque \textbf{346} (2012), 171--312.

\bibitem{w3}
\bysame,
\emph{Calcul d'une valeur d'un facteur $\epsilon$ par une formule int\'egrale},
Ast\'erisque \textbf{347} (2012), 1--102.

\bibitem{w4}
\bysame,
\emph{La conjecture locale de Gross--Prasad pour les repr\'esentations temp\'er\'ees des groupes sp\'eciaux orthogonaux},
Ast\'erisque \textbf{347} (2012), 103--165.

\bibitem{wallach}
N.~R.~Wallach,
\emph{On the constant term of a square integrable automorphic form},
Operator algebras and group representations, Vol.~II,
Monogr. Stud. Math. \textbf{18}, pp.~227--237, Pitman, 1984.

\end{thebibliography}
\end{document}